\documentclass[smallextended]{svjour3}       
\smartqed  
\usepackage{hyperref}
\usepackage{lipsum}
\usepackage{latexsym,amsfonts,amsmath,amssymb}
\usepackage{tikz}
\usepackage{booktabs,ctable,threeparttable,boxedminipage}
\usepackage{graphicx}
\usepackage{color,caption,url}
\usepackage{subfigure}
\usepackage{epstopdf}

\usepackage{algorithm}
\usepackage{algpseudocode}
\usepackage{blkarray}
\usepackage{enumerate}
\usepackage{arydshln}
\usepackage{lineno}
\usepackage{float}  
\usepackage{changes}
\numberwithin{figure}{section}
\numberwithin{table}{section}
\numberwithin{equation}{section}
\numberwithin{theorem}{section}
\numberwithin{lemma}{section}
\numberwithin{algorithm}{section}
\numberwithin{remark}{section}

\ifpdf
\DeclareGraphicsExtensions{.eps,.pdf,.png,.jpg}
\else
\DeclareGraphicsExtensions{.eps}
\fi
\let\oldequation\equation
\let\oldendequation\endequation
\renewenvironment{equation}
{\linenomathNonumbers\oldequation}
{\oldendequation\endlinenomath}
\let\oldalign\align
\let\oldendalign\endalign
\renewenvironment{align}
{\linenomathNonumbers\oldalign}
{\oldendalign\endlinenomath}
\begin{document}


\title{Double precision is not necessary for LSQR for solving discrete linear ill-posed problems
	\thanks{This work was supported in part by the National Natural Science Foundation of China under Grant No. 3192270206.
}}

\titlerunning{Double precision is not necessary for LSQR}  

\author{Haibo Li}


\institute{Haibo Li \at
Institute of Computing Technology, Chinese Academy of Sciences, Beijing 100190, China \\
              \email{haibolee1729@gmail.com}   
}


\maketitle

\begin{abstract}
The growing availability and usage of low precision floating point formats attracts many interests of developing lower or mixed precision algorithms for scientific computing problems. In this paper we investigate the possibility of exploiting mixed precision computing in LSQR for solving discrete linear ill-posed problems. Based on the commonly used regularization model for linear inverse problems, we analyze the choice of proper computing precision in the two main parts of LSQR, including the construction of Krylov subspace and updating procedure of iterative solutions. We show that, under some mild conditions, the Lanczos vectors can be computed using single precision without loss of any accuracy of the final regularized solution as long as the noise level is not extremely small. We also show that the most time consuming part for updating iterative solutions can be performed using single precision without sacrificing any accuracy. The results indicate that several highly time consuming parts of the algorithm can be implemented using lower precisions, and provide a theoretical guideline for implementing a robust and efficient mixed precision variant of LSQR for solving discrete linear ill-posed problems. Numerical experiments are made to test two mixed precision variants of LSQR and confirming our results.
\end{abstract}

\section{Introduction}
Although for most traditional scientific computing problems, computations are carried out with double precision (64-bit) rather than lower precisions such as single (32-bit) or half (16-bit) precision, on modern computing architectures, the performance of 32-bit operations is often at least twice as fast as that of 64-bit operations \cite{abdelfattah2021survey}, which stimulates the trial of using lower precision floating point formats in an algorithm. To exploit this computation power without sacrificing accuracy of the final result, numerical algorithms have to be designed that use lower or mixed precision formats. By using a combination of 64-bit and 32-bit (even 16-bit) floating point arithmetic, the performance of many numerical algorithms can be significantly enhanced while maintaining the 64-bit accuracy of the resulting solution. New mixed precision variants of many numerical linear algebra algorithms have been recently proposed, such as matrix multiplications \cite{Ahmad2019,Blanchard2020}, LU and QR matrix factorizations \cite{Lopez2020,Yang2021}, Krylov solvers \cite{Grat2019,Carson2020,gratton2021minimizing,durand2022accelerating} and many others \cite{Higham2019,Amestoy2021}.

In this paper, we investigate how to exploit mixed precision computing for solving discrete linear ill-posed problems. This type of problems typically arise from the numerical solution of inverse problems that appear in various applications of science and engineering, such as image deblurring, geophysics, computerized tomography and many others; see e.g., \cite{Natterer2001,Kaip2006,Hansen2006,Richter2016}. A basic linear inverse problem leads to a discrete linear system of the form
\begin{equation}\label{1.1}
	Ax = b, \ \ \ \ b = Ax_{ex}+e,
\end{equation}
where the matrix $A\in \mathbb{R}^{m\times n}$ with $m\geq n$ without loss of generality, and the right-hand side $b$ is a perturbed version of the unknown exact observation $b_{ex}=Ax_{ex}$. In this paper we suppose $e$ is a Gaussian white noise. The problem is ill-posed in the sense that $A$ is extremely ill-conditioned with its singular values decaying gradually towards zero without any noticeable gap, which leads to that the naive solution $x_{nai}=A^{\dag}b$ of \ref{1.1} is a poor approximation to the exact solution $x_{ex}=A^{\dag}b_{ex}$, where $``\dag "$ denotes the Moore-Penrose inverse of a matrix. Therefore, some forms of regularization must be used to deal with the noise $e$ in order to extract a good approximation to $x_{ex}$.

One of the popular regularization techniques is the Tikhonov regularization \cite{Tikhonov1977}, in which a quadratic penalty is added to the objective function:
\[ x_{\lambda} = \arg \min_{x \in \mathbb{R}^{n}}\{\lVert Ax-b \lVert^{2}
	+\lambda\lVert Lx\lVert^{2}\} ,\] 
where $\lambda>0$ is the regularization parameter and $L\in\mathbb{R}^{p\times n}$ is the regularization matrix. Throughout the rest of the paper $\|\cdot\|$ always denotes either the vector or matrix 2-norm. 
The proper choice of $L$ depends on the particular application, which should be chosen to yield a regularized solution with some known desired features of $x_{ex}$. A suitable value of $\lambda$ should have a good balance between the data fidelity term $\|Ax-b\|$ and the regularization term $\|Lx\|$, only in which case we can get a regularized solution that is a good approximation to $x_{ex}$. Although many types of regularization parameter choice rules have been proposed, such as discrepancy principle (DP)~\cite{Morozov1966}, unbiased predictive risk estimator \cite{Vogel2002}, generalized cross validation \cite{Golub1979} and L-curve criterion \cite{Hansen1992}, it is often computationally very expensive to choose a suitable $\lambda$ for large scale problems, since many different values of $\lambda$ must be tried to get $x_{\lambda}$.

For large scale ill-posed problems, iterative regularization method is the soundest choice. For standard-form regularization with $L=I_n$, Krylov subspace based methods such as LSQR \cite{Paige1982,Bjorck1988,Gazzola2015} are the most commonly used. The methods project \eqref{1.1} onto a sequence of lower-dimensional Krylov subspaces, and then solves the projected small scale problems, where the iteration number plays the role of regularization parameter \cite{Hansen1998,Hansen2010}. This approach usually exhibits semi-convergence: as the iteration proceeds, the iterative solution first approximates $x_{ex}$ while afterwards the noise $e$ starts to deteriorate the solution so that it gradually diverge from $x_{ex}$ and instead converges to $x_{nai}$. Therefore, the iteration must be stopped early properly by using a regularization parameter choice rule \cite{Borges2015,Reichel2020}. The semi-convergence behavior can be mitigated by using a hybrid method, which applies a standard regularization technique, such as Tikhonov regularization or truncated SVD, to the projected problem at each iteration \cite{Kilmer2001,Chung2008,Renaut2017}.

In this paper, we focus on the LSQR algorithm for iteratively solving large scale discrete linear ill-posed problems that is based on the Lanczos bidiagonalization \cite{Paige1982}. The motivation for this work is to answer whether lower precisions can be used in some parts of the algorithm while maintaining the 64-bit accuracy of the regularized solutions. The LSQR for linear ill-posed problems mainly includes two parts, that is the construction of Krylov subspace by Lanczos bidiagonalization and updating procedure of iterative solutions. In addition, a proper iteration number should be estimated to stop the iteration near the semi-convergence point. Since $b$ is contaminated by the noise $e$, we can never get a regularized solution with error as small as $\|e\|$ \cite{Engl2000,Hansen1998}, and this error is usually much bigger than the roundoff unit of double precision with the value $2^{-53}$. This fact inspires us that double precision may be not necessary for LSQR to compute a regularized solution with the same accuracy as the best regularized one. To the best of our knowledge, however, there is still no theoretical analysis about how to choose proper lower computing precision in LSQR for linear ill-posed problems. This issue is crucial for deciding which lower precision format should be used in each part of the algorithm to get a more efficient mixed precision implementation.

We study the lower precision computing for the two main parts of LSQR, including the construction of Krylov subspace and updating iterative solutions. In finite precision arithmetic, we implement the Lanczos bidiagonalization in LSQR with full reorthogonalization of Lanczos vector, which is a frequently used strategy to avoid slowing down and irregular convergence of iterative solutions \cite{Larsen1998,Hnetyn2009}. First, under an ideal model describing the linear ill-posed problem \eqref{1.1}, our result estimates an upper bound on the proper value of $\mathbf{u}$ corresponding to the used computing precision for constructing Lanczos vectors with full reorthogonalization, and it indicates that for not extremely small noise levels we can exploit single precision for this part without loss of accuracy of the final regularized solution. Second, for the updating procedure part, we theoretically show that, under a condition which can be almost always satisfied, the updated regularized solutions can be computed using single precision without sacrificing any accuracy. We also investigate the L-curve and discrepancy principle methods for estimating the optimal early stopping iteration combined with the mixed precision implementation of LSQR. Overall, the results theoretically show that several highly time consuming parts of LSQR for solving discrete linear ill-posed problems can be implemented using lower precisions, which has a great potential of defeating the double precision implementation in computation efficiency. The results can guide us towards a mixed precision implementation that is both robust and efficient, and should be considered by application developers for practical problems.

The paper is organized as follows. We start in Section \ref{sec2} with a brief review of regularization theory and algorithm of linear ill-posed inverse problems, and the IEEE 754 floating point standard. Next, in Section \ref{sec3}, we analyze the proper choice of $\mathbf{u}$ corresponding to the used computing precision for constructing Lanczos vectors with full reorthogonalization and give an upper bound. In Section \ref{sec4}, we analyze the possibility of using lower precision for updating iterative solutions and give a mixed precision variant of LSQR. We also discuss the estimation of optimal early stopping iteration for the mixed precision LSQR. Some numerical experimental results are presented in Section \ref{sec5} and we conclude the paper in Section \ref{sec6}.

\section{Preliminaries}\label{sec2}
In this section, we review some basic knowledge of regularization theory of linear ill-posed inverse problems and the LSQR regularization algorithm, we also review finite precision computing based on the IEEE 754 Standard floating point number system.

\subsection{\textbf{Regularization of linear ill-posed problems and LSQR}}
Suppose the singular value decomposition (SVD) of $A$ is
\begin{equation}\label{eqsvd}
	A=U\left(\begin{array}{c} \Sigma \\ \mathbf{0} \end{array}\right) V^{T} ,
\end{equation}
where $U = (u_1,\ldots,u_m)\in\mathbb{R}^{m\times m}$ and $V = (v_1,\ldots,v_n)\in\mathbb{R}^{n\times n}$ are orthogonal, $\Sigma = {\rm diag} (\sigma_1,\ldots,\sigma_n)\in\mathbb{R}^{n\times n}$ with singular values $\sigma_1\geq\sigma_2 \geq\cdots \geq\sigma_n>0$. The naive solution to \eqref{1.1} is
\[x_{nai}=\sum\limits_{i=1}^{n}\frac{u_i^{T}b}{\sigma_i}v_i =
\sum\limits_{i=1}^{n}\frac{u_i^{T}b_{ex}}{\sigma_i}v_i +
\sum\limits_{i=1}^{n}\frac{u_i^{T}e}{\sigma_i}v_i
=x_{ex}+\sum\limits_{i=1}^{n}\frac{u_i^{T}e}{\sigma_i}v_i .\]
Note that the second term in $x_{nai}$ is extremely large since $\sigma_{i}$ decay to zero, making $x_{nai}$ a meaningless solution. A direct regularized method is the truncated SVD (TSVD) method, which forms $x_k^{tsvd}$ by truncating the first $k$ components of $x_{nai}$ corresponding to large singular values:
$x_k^{tsvd}=\sum\limits_{i=1}^{k}\frac{u_i^{T}b}{\sigma_i}v_i$. Regularization theory of linear inverse problems can be used to investigate the accuracy of the regularized solution to \eqref{1.1}, among which the \textit{discrete Picard condition(DPC)} plays a central role. We give a brief review in the following.

The DPC for the exact right-hand side $b_{ex}$ can be written in the following popular simplifying model \cite{Hansen1998,Hansen2010}:
\begin{equation}\label{3.4}
	|u_{i}^{T}b_{ex}|=\rho_{0}\sigma_{i}^{1+\beta},  \ \ \beta>0, \ i=1,2,\dots,n,
\end{equation}
where $\beta$ is a model parameter that controls the decay rates of $|u_{i}^{T}b_{ex}|$\footnote{In \cite[\S 4.6]{Hansen2010}, the corresponding model is $|u_{i}^{T}b_{ex}|=\sigma_{i}^{1+\beta}$, which does not include the constant $\rho_{0}$. In fact, Hansen \cite[p.68]{Hansen2010} points out that ``while this is, indeed, a crude model, it reflects the overall behavior often found in real problems".}. The DPC implies that the noisy coefficients $u_{i}^{T}b$ gradually decay in average and are larger than $|u_{i}^{T}e|$ for $i=1,2\dots$ until the noise dominates. Suppose that the noise in $|u_{i}^{T}b|$ starts to dominate at $k=k^{*}+1$, i.e., $k^{*}$ is the transition point such that \cite[\S 3.5.1]{Hansen2010}
\begin{equation}\label{trans}
	|u_{k^{*}}^{T}b|\approx|u_{k^{*}}^{T}b_{ex}|> |u_{k^{*}}^{T}e|, 
	\ \ |u_{k^{*}+1}^{T}b_{ex}|\approx|u_{k^{*}+1}^{T}e|. 
\end{equation}
Under the assumption that $e\in\mathbb{R}^{m}$ is a Gaussian white noise, it is shown that $|u_{i}^{T}e|\approx m^{-1/2}\|e\|$ \cite[\S 3.5.1]{Hansen2010}. Thus the DPC for noisy $b$ can be written in the following form:
\begin{equation}\label{DPC}
	|u_{i}^{T}b|=|u_{i}^{T}b_{ex}+u_{i}^{T}e|\approx
	\left\{
	\begin{array}{l}
		\rho_{0}\sigma_{i}^{1+\beta} , \ 1\leq i \leq k^{*} ; \\
		m^{-1/2}\|e\| , \ i \geq k^{*}+1. 
	\end{array}
	\right. 
\end{equation}

The accuracy of the best regularized solution to \eqref{1.1} is closely connected with the DPC \eqref{DPC}, which can be reflected by the \textit{effective resolution limit} denoted by $\eta_{res}$. The effective resolution limit of \eqref{1.1} denotes the smallest coefficient $|v_{i}^{T}x_{ex}|$ that can be recovered from the given $A$ and noisy $b$ \cite[\S 4.1]{Hansen1998}, and it is shown in \cite[\S 4.5]{Hansen1998} that
$\eta_{res} \approx (m^{-1/2}\lVert e\lVert)^{\frac{\beta}{1+\beta}}$.
The accuracy of the best regularized solution $x_{opt}$ is dependent on $\eta_{res}$, that is, one can only hope that $x_{opt}$ reaches an accuracy corresponding to the effective resolution limit, which means
\begin{equation}\label{3.7}
	\frac{\|x_{opt}-x_{ex}\|}{\|x_{ex}\|} \geq C_1\varepsilon^{\frac{\beta}{1+\beta}}
\end{equation}
with a moderate constant $C_1$, where $\varepsilon = \|e\|/\|b_{ex}\|<1$ is the noise level; see e.g., \cite{Hansen1998,Engl2000} for more details. Note that $\|e\|<\|b_{ex}\|$; otherwise all information from $b_{ex}$ is lost in $b$. 
In particular, it is known from \cite{Hansen1998,Engl2000} that $x_{k^*}^{tsvd}$ is a 2-norm filtering best possible regularized solution of \eqref{1.1} when only deterministic 2-norm filtering regularization methods are taken into account.


For large scale ill-posed problems, the LSQR algorithm with early stopping is the most common used iterative regularization method, which is based on Lanczos bidiagonalization of $A$ with starting vector $b$, as described in Algorithm \ref{alg1}.
\begin{algorithm}[htb]
	\caption{$k$-step Lanczos bidiagonalization}\label{alg1}
	\begin{algorithmic}[1]
		\State Let $\beta_{1}=\|b\|, \ p_{1}=b/\beta_{1}$
		\State $\alpha_{1}=\|A^{T}p_{1}\|, \ q_1=A^{T}p_{1}/\alpha_{1}$ 
		\For{$j=1,2,\dots,k$}
		\State $s_{j}=Aq_{j}-\alpha_{j}p_{j}$
		\State $\beta_{j+1}=\|s_{j}\|$, \ $p_{j+1}=s_{j}/\beta_{j+1}$ 
		\State $r_{j}=A^{T}p_{j+1}-\beta_{j+1}q_{j}$ 
		\State $\alpha_{j+1}=\|r_{j}\|$, \ $q_{j+1}=r_{j}/\alpha_{j+1}$
		\EndFor
	\end{algorithmic}
\end{algorithm}

In exact arithmetic, the $k$-step Lanczos bidiagonalization produces a lower bidiagonal matrix 
\[ B_{k}=\begin{pmatrix}
		\alpha_{1} & & & \\
		\beta_{2} &\alpha_{2} & & \\
		&\beta_{3} &\ddots & \\
		& &\ddots &\alpha_{k} \\
		& & &\beta_{k+1} 
	\end{pmatrix} \in \mathbb{R}^{(k+1)\times k} ,\]
and two groups of Lanczos vectors $\{p_{1},\dots,p_{k+1}\}$ and $\{q_{1},\dots,q_{k+1}\}$ that are orthonormal bases of Krylov subspaces $\mathcal{K}_{k+1}(AA^{T},b)$ and $\mathcal{K}_{k+1}(A^{T}A,A^{T}b)$, respectively. The $k$-step Lanczos bidiagonalization can be written in the matrix form
\begin{align}
	& P_{k+1}(\beta_{1}e_{1}^{(k+1)}) = b , \label{2.1} \\
	& AQ_{k} = P_{k+1}B_{k} , \label{2.2} \\
	& A^{T}P_{k+1} = Q_{k}B_{k}^{T}+\alpha_{k+1}q_{k+1}(e_{k+1}^{(k+1)})^{T} , \label{2.3}
\end{align}
where $e_{i}^{(l)}$ denotes the $i$-th canonical basis vector of $\mathbb{R}^{l}$, and $P_{k+1}=(p_{1},\dots, p_{k+1})$ and $Q_{k}=(q_{1}, \dots, q_{k+1})$ are two orthonormal matrices. The LSQR for \eqref{1.1} is mathematically equivalent to the conjugate gradient (CG) method applied to the normal equation of $\min_{x\in\mathbb{R}^{n}}\|Ax-b\|$, i.e. $A^TAx=A^Tb$, which seeks approximations to $x_{ex}$ from the $k$ dimensional Krylov subspace $\mathcal{K}_{k}(A^{T}A,A^{T}b)=\mathcal{R}(V_{k})$ starting with $k=1$ onwards, and the iteration should be terminated at a proper step near the semi-convergence point to get a good regularized solution \cite{Bjorck1988}. At the $k$-th step, by \eqref{2.1} and \eqref{2.2} we have
\begin{align*}
		\min\limits_{x=Q_{k}y} \lVert Ax - b \lVert 
		= \min\limits_{y\in \mathbb{R}^{k}} \lVert B_{k}y -\beta_{1}e_{1}^{(k+1)}\lVert ,
\end{align*}
and thus the $k$-step LSQR solution is
\begin{equation}\label{2.6}
	x_{k}=Q_{k}y_{k}, \ \
	y_{k}=\mathrm{arg}\min\limits_{y\in \mathbb{R}^{k}} \lVert B_{k}y -\beta_{1}e_{1}^{(k+1)}\lVert 
	= B_{k}^{\dag}(\beta_{1}e_{1}^{(k+1)}) .
\end{equation}
We note that if $\alpha_{k+1}$ or $\beta_{k+1}$ is zero, then the iteration terminates and $x_{k}=x_{nai}$ \cite{Paige1982}. This case rarely happens in real computations and we assume that the iteration does not terminate.

From the above description, the computation of $x_k$ can be divided into two parts. The first part is the Lanczos bidiagonalization that generates two orthonormal bases of Krylov subspaces $\mathcal{K}_{k+1}(AA^{T},b)$ and $\mathcal{K}_{k+1}(A^{T}A,A^{T}b)$, respectively, while the second part is solving the projected problem \eqref{2.6} to obtain $x_k$. In the practical implementation, there is a recursive formula to update $x_{k+1}$ from $x_k$ without solving the projected least squares problems at each iteration. This updating procedure will be investigated in Section \ref{sec4}.


\subsection{\textbf{Finite precision computing}}
In practical computational tasks, the accuracy of a computed result and the time needed to complete the algorithm both heavily depend on the floating point format used for storage and arithmetic operations. Here we review the IEEE 754 Standard floating point number format, which is composed of a sign bit, an exponent $\eta$, and a significand $t$:
\[x = \pm \mu\times 2^{\eta-t},\] 
where $\mu$ is any integer in $[0,2^{t}-1]$ and $\eta$ is an integer in  $[\eta_{\text{min}}, \eta_{\text{max}}]$.
Roughly speaking, the length of the exponent determines the value range of a floating point format, and the length of the significand determines the relative accuracy of the format in that range. A short analysis of floating point operations \cite[Theorem 2.2]{Higham2002} shows that the relative error is controlled by the roundoff unit $\mathbf{u}:=\frac{1}{2}\cdot2^{1-t}$. Table \ref{tab2.1} shows main parameters of the three different floating point number formats.
\begin{table}[h]
	\centering
	\caption{Parameters for various floating-point formats. “Range” denotes the order of magnitude of the smallest positive (subnormal) $x_{min}^{s}$ and smallest and largest positive normalized floating-point numbers.}
\setlength{\tabcolsep}{1mm}{
\begin{tabular}{*{6}{c}}
	\toprule
	Type  & Size & \multicolumn{3}{c}{Range} & Roundoff unit \\
	\cmidrule(lr){3-5}
	& (bits)  & $x_{min}^{s}$ & $x_{min}$ & $x_{max}$ & $\mathbf{u}$ \\
	\midrule
	half precision   & 16 & $5.96\times10^{-8}$ & $6.10\times10^{-5}$ & $6.55\times10^{4}$ & $4.88\times  10^{-4}$ \\
	single precision  & 32 & $1.40\times10^{-45}$ & $1.18\times10^{-38}$ & $3.40\times10^{38}$ & $5.96\times10^{-8}$ \\
	double precision  & 64 & $4.94\times10^{-324}$ & $2.22\times10^{-308}$& $1.80\times10^{308}$ &$1.11\times10^{-16}$ \\
	\bottomrule
\end{tabular}}
\label{tab2.1}
\end{table}


In finite precision arithmetic, the Lanczos vectors $u_i$ and $v_i$ computed by the Lanczos bidiagonalization gradually lose their orthogonality, and it may slow down the convergence of iterative solutions and make the propagation of noise during the iterations rather irregular \cite{Meurant2006,Hnetyn2009}. A frequently used strategy is implementing the Lanczos bidiagonalization with full reorthogonalization (LBFRO) \footnote{In full reorthogonalization, $u_k$ and $v_k$ are reorthogonalized against all previous vectors $\{u_1,\dots, u_{k-1}\}$ and $\{v_1,\dots, v_{k-1}\}$ as soon as they have been computed. This adds an arithmetic cost of about $4(m + n)k^2$ flops, which is affordable if $k \ll\min\{m,n\}$.} to maintain stability of convergence. In the rest of the paper, we investigate the \textit{LSQR implemented with full reorthogonalization of Lanczos vectors} in finite precision. From now on, notations such as $P_k$, $B_k$, $\alpha_k$, etc. denote the computed quantities in finite precision computing.

Define the orthogonality level of Lanczos vectors $\{p_{1},\dots,p_{k}\}$ and $\{q_{1},\dots,q_{k}\}$ as
\[\mu_{k} =  \lVert \mathbf{SUT}(I_{k}-P_{k}^{T}P_{k})\lVert, \ \ \
\nu_{k} =  \lVert \mathbf{SUT}(I_{k}-Q_{k}^{T}Q_{k})\lVert,\]
where $\mathbf{SUT(\cdot)}$ denotes the strictly upper triangular part of a matrix. The following result has been established for the $k$-step Lanczos bidiagonalization with reorthogonalization (not necessarily full reorthogonalization) \cite{Li2022}. 
\begin{theorem}\label{thm2.1}
	For the $k$-step Lanczos bidiagonalization with reorthogonalization, if $\nu_{k+1}<1/2$ and $\mu_{k+1} < 1/2$, then there exist two orthornormal matrices $\bar{P}_{k+1}=(\bar{p}_{1},\dots,\bar{p}_{k+1})\in\mathbb{R}^{m\times(k+1)}$ and $\bar{Q}_{k+1}=(\bar{q}_{1},\dots,\bar{q}_{k+1})\in \mathbb{R}^{n\times (k+1)}$ such that
	\begin{align}
		& \bar{P}_{k+1}(\beta_{1}e_{1}^{(k+1)}) = b+\delta_b , \label{2.10} \\
		& (A+E)\bar{Q}_{k} = \bar{P}_{k+1}B_{k} , \label{2.11} \\
		& (A+E)^{T}\bar{P}_{k+1} = \bar{Q}_{k}B_{k}^{T}+\alpha_{k+1}\bar{q}_{k+1}(e_{k+1}^{(k+1)})^{T},
	\end{align}
	where $E$ and $\delta_b$ are perturbation matrix and vector, respectively. We have error bounds
	\[\lVert \bar{P}_{k+1}-P_{k+1} \lVert \leq 2\mu_{k+1} + \mathcal{O}(\mu_{k+1}^{2}) , \ \ \ 
		\lVert \bar{Q}_{k+1}-Q_{k+1} \lVert \leq \nu_{k+1} + \mathcal{O}(\nu_{k+1}^{2}) ,\]
	and
	\[\lVert E \lVert =
	\mathcal{O}(c(n,k)\lVert A \lVert(\mathbf{u}+\nu_{k+1}+\mu_{k+1})),
		\ \ \ \lVert \delta_b \lVert = O(\lVert b \lVert \mathbf{u}).\]
	where $c(n,k)$ is a moderately growing constant depends on $n$ and $k$.
\end{theorem}

For the LBFRO, the orthogonality levels of $u_i$ and $v_i$ are kept around $\mathcal{O}(\mathbf{u})$, thus by Theorem \ref{thm2.1} we have
\begin{align}
	& \lVert \bar{P}_{k+1}-P_{k+1}\| = \mathcal{O}(\mathbf{u}) , \ \ \ 
	  \lVert \bar{Q}_{k+1}-Q_{k+1} \lVert = \mathcal{O}(\mathbf{u}) , \label{2.15} \\
	& \lVert E \lVert = \mathcal{O}(c(n,k)\lVert A \lVert\mathbf{u}), \ \ \ 
	  \lVert \delta_b \lVert = \mathcal{O}(\lVert b \lVert \mathbf{u}). \label{2.16}
\end{align} 
This result will be used in the next section to estimate upper bound on the proper value of roundoff unit $\mathbf{u}$ corresponding to the used computing precision.

\section{Choice of computing precision for the construction of Krylov subspace}\label{sec3}
In this section, we investigate which lower precision format should be used for computing Lanczos vectors with full reorthogonalization. To this end, by assuming the LBFRO is implemented in finite precision computing with roundoff unit $\mathbf{u}$, we will give an upper bound on $\mathbf{u}$ such that the best computed regularized solution can achieve the same accuracy as the best regularized LSQR solution to \eqref{1.1} obtained in exact arithmetic. In this section, the updating procedure or \eqref{2.6} is assumed to be implemented in exact arithmetic.

\begin{theorem}\label{lem3.1}
	Suppose that the LBFRO in LSQR is implemented in finite precision with roundoff unit $\mathbf{u}$. If \eqref{2.6} is solved exactly, then the computed $x_k$ satisfies 
	\begin{equation}\label{3.1}
		\dfrac{\|x_{k}-\bar{x}_{k}\|}{\|\bar{x}_{k}\|} = \mathcal{O}(\mathbf{u}),
	\end{equation}
	where $\bar{x}_k$ is the exact $k$-th LSQR solution to the perturbed problem
	\begin{equation}\label{3.2}
		\min\limits_{x \in \mathbb{R}^{n}} \lVert (A+E)x-
		(b+\delta_b) \lVert .
	\end{equation}
\end{theorem}
\begin{proof}
	Since \eqref{2.6} is solved exactly, by Theorem \ref{thm2.1} we have
	\begin{align*}
		y_{k}
		&=  \mathrm{arg}\min\limits_{y\in \mathbb{R}^{k}} \lVert B_{k}y
		- \beta_{1}e_{1}^{(k+1)} \lVert 
		= \mathrm{arg}\min\limits_{y\in \mathbb{R}^{k}} \lVert \bar{P}_{k+1}B_{k}y - \bar{P}_{k+1}\beta_{1}e_{1}^{(k+1)} \lVert \\
		&= \mathrm{arg}\min\limits_{y\in \mathbb{R}^{k}} \lVert (A+E)\bar{Q}_{k}y-(b+\delta_b) \lVert .
	\end{align*}
	From Theorem \ref{thm2.1} we know that $\bar{Q}_k$ is the right orthonormal matrix generated by the Lanczos biagonalization of $A+E$ with starting vector $b+\delta_b$ in exact arithmetic, which means $\mathcal{R}(\bar{V}_k)=\mathcal{K}_{k}((A+E)^{T}(A+E),(A+E)^{T}(b+\delta_b))$. Let $\bar{x}_{k} = \bar{Q}_{k}y_k$. Then $\bar{x}_{k}$ is the $k$-th LSQR solution to \eqref{3.2}.
	Since $x_k=Q_ky_k$, by \eqref{2.15} we have
	$$\lVert \bar{x}_{k}-x_{k}\lVert \leq \|Q_{k}-\bar{Q}_{k}\|\|y_{k}\|=\mathcal{O}(\|y_{k}\|\mathbf{u})).$$
	Using $\|y_{k}\|=\|\bar{x}_{k}\|$, we obtain \eqref{3.1}. \qed
\end{proof} 

This result indicates that $x_{k}\approx \bar{x}_{k}$ within $\mathcal{O}(\mathbf{u})$. If we rewrite \eqref{3.2} as
\begin{equation}\label{3.3}
	(A+E)x = (b_{ex}+Ex_{ex}) + (e-Ex_{ex}+\delta_b),
\end{equation}
then $x_{ex}$ is the exact solution to $(A+E)x_{ex}=b_{ex}+Ex_{ex}$ and $e-Ex_{ex}+\delta_b$ is the noise term. Notice that $\|E\|/\|A\|=\mathcal{O}(c(n,k)\mathbf{u})$ and thus for a not big $k$ the singular values of $A+E$ decrease monotonically without noticeable gap until they tend to settle at a level of $\mathcal{O}(\|A\|\mathbf{u})$. Therefore, the linear system \eqref{3.3} inherits the ill-posedness of \eqref{1.1}. Moreover, if $\|-Ex_{ex}+\delta_b\|\ll\|e\|$, then $e-Ex_{ex}+\delta_b$ can be treated as a Gaussian noise. For a sufficiently small $\mathbf{u}$, we can hope that the best LSQR regularized solution to \eqref{3.3} has the same accuracy as that to \eqref{1.1}.

Suppose that the best LSQR regularized solution to \eqref{3.3} is $\bar{x}_{k_0}$. Let the $i$-th largest singular values of $A+E$ be $\bar{\sigma}_i$ and the corresponding left singular vector be $\bar{u}_i$. Applying the regularizaton theory and DPC to \eqref{1.1} and \eqref{3.3}, the accuracy of $\bar{x}_{k_0}$ will be the same as that of $x_{opt}$ if: 
\begin{itemize}
	\item[$\bullet $] the DPC of \eqref{3.3} inherits properties of the DPC of \eqref{1.1}, i.e., the DPC of \eqref{3.3} before the noise dominates should satisfies
	\begin{equation}\label{DPC2}
		|\bar{u}_{i}^{T}b| \approx \rho_{0}\bar{\sigma}_{i}^{1+\beta}, \ \ 1\leq i \leq k^{*};
	\end{equation}
	\item[$\bullet $] the effective resolution limit of \eqref{3.3} has the same accuracy as that of \eqref{1.1} \footnote{The assertion implicitly uses the order optimal property of the LSQR or its mathematical equivalent CG algorithm for linear inverse problems \cite[\S 7.3]{Engl2000}, which means that the best LSQR regularized solution can achieve the same accuracy as that of $x_{opt}$. Although plenty of numerical results confirm it without any exception, this property has not been rigorously proved for discrete linear ill-posed problems; see \cite[\S 6.3]{Hansen1998}and \cite{Chung2021} for more discussions.}.
\end{itemize} 
Note that Theorem \ref{lem3.1} implies that $\|x_{k_0}-\bar{x}_{k_0}/\|\bar{x}_{k_0}\|= \mathcal{O}(\mathbf{u})$. Therefore, for the above $\bar{x}_{k_0}$, the corresponding $x_{k_0}$ will have the same accuracy as $x_{opt}$ if
\begin{equation}\label{3.8}
	\mathbf{u}\ll C_1\varepsilon^{\frac{\beta}{1+\beta}},
\end{equation}
which implies that the best computed regularized solution among $x_k$ can achieve the same accuracy as $x_{opt}$ .

In the following we make some analysis about the above assertions.
\begin{lemma}\label{lem3.2}
	For the ill-posed problem \eqref{3.2} or its equivalence \eqref{3.3}, if \eqref{DPC2} holds and the following two conditions are satisfied:
	$\mathbf{(1).}$ $\|-Ex_{ex}+\delta_b\|\ll\|e\|$;
	$\mathbf{(2).}$ $\lVert E \lVert \ll \sigma_{k^{*}}$, then the effective resolution limit of \eqref{3.3} has the same accuracy as that of \eqref{1.1}. 
\end{lemma}
\begin{proof}
	Since the noise in problem \eqref{3.3} is $e-Ex_{ex}+\delta_b$, by Condition $\mathbf{(1)}$, the noise $e$ dominates, thus we can regard this noise as Gaussian. Condition $\mathbf{(2)}$ implies that  $\bar{\sigma}_{i}\approx \sigma_{i}$ until the noise in $|u_{i}^{T}b|$ starts to dominates, and thus the errors in $A+E$ starts to dominates at a point $\bar{k}^{*}>k^{*}$. Therefore, it follows from \cite[\S 4.5]{Hansen1998}, if \eqref{DPC2} holds, that
	\[\bar{\eta}_{res} \approx (m^{-1/2}\lVert e-Ex_{ex}+\delta_b\lVert)^{\frac{\beta}{1+\beta}} 
	\approx (m^{-1/2}\lVert e\lVert)^{\frac{\beta}{1+\beta}} ,\]
	where $\bar{\eta}_{res}$ is the effective resolution limit of \eqref{3.3}. \qed
\end{proof}

\begin{lemma}\label{lem3.3}
	For the iteration number $k$ not very big, if $\mathbf{u}$ satisfies 
	\begin{equation}\label{bnd1}
		\mathbf{u} \ll
		\min\{\varepsilon, (m^{-1/2}\varepsilon)^{\frac{1}{1+\beta}}\},
	\end{equation}
	then the relation \eqref{3.8} and Conditions $\mathbf{(1)}$ and $\mathbf{(2)}$ hold.
\end{lemma}
\begin{proof}
	Condition $\mathbf{(1)}$ holds if $\|Ex_{ex}\|\ll\|e\|$ and $\|\delta_b\|\ll\|e\|$. By \eqref{2.16}, these two equalities can be satisfied if $\mathbf{u}\ll\|e\|/\|b\|$ and $\|E\| \ll \|e\|/\|x_{ex}\|$. Since 
	\[\dfrac{\|e\|}{\|b\|}\geq
	\dfrac{\|e\|}{\|b_{ex}\|+\|e\|}=\dfrac{\varepsilon}{1+\varepsilon}>\dfrac{\varepsilon}{2},\]
	we have $\mathbf{u}\ll\|e\|/\|b\|$ if $\mathbf{u} \ll\varepsilon$. Note that
	\[ \dfrac{\|e\|}{\|x_{ex}\|}\leq \dfrac{\|A\|\|e\|}{\|b_{ex}\|}=\varepsilon\|A\|, \]
	and the value of $\|e\|/\|x_{ex}\|$ should not deviate too far from $\varepsilon\|A\|$ since $\|x_{ex}\|$ is usually a moderate quantity. Using $\lVert E \lVert =\mathcal{O}(c(n,k)\|A\|\mathbf{u})$, we have $\|E\| \ll \|e\|/\|x_{ex}\|$ if $\mathbf{u} \ll\varepsilon$ since $c(n,k)$ is a moderate quantity.
 
 	By \eqref{3.4} and \eqref{trans}, we have $\rho_{0}\sigma_{k^{*}+1}^{1+\beta}\approx|u_{k^{*}+1}^{T}e|\approx m^{-1/2}\|e\|$, which implies that Condition $\mathbf{(2)}$ will hold if
	\begin{equation}\label{sig}
		\|E\| \ll \sigma_{k^{*}+1} \approx (m^{-1/2}\rho_{0}^{-1}\|e\|)^{\frac{1}{1+\beta}}.
	\end{equation}
	By \eqref{3.4}, we have $\rho_{0}=|u_{1}^{T}b_{ex}|/\sigma_{1}^{1+\beta}$, and thus
	\begin{equation}\label{ineq1}
		(\rho_{0}^{-1}\|e\|)^{\frac{1}{1+\beta}}= \sigma_{1}\Big(\dfrac{\|e\|}{|u_{1}^{T}b_{ex}|}\Big)^{\frac{1}{1+\beta}}\geq \|A\|\Big(\dfrac{\|e\|}{\|b_{ex}\|}\Big)^{\frac{1}{1+\beta}}
		= \varepsilon^{\frac{1}{1+\beta}}\|A\|.
	\end{equation}
	Therefore Condition $\mathbf{(2)}$ will hold if $\|E\| \ll(m^{-1/2}\varepsilon)^{\frac{1}{1+\beta}}\|A\|$, which can be satisfied if $\mathbf{u}\ll(m^{-1/2}\varepsilon)^{\frac{1}{1+\beta}}$. By the above derivations, one can check that \eqref{3.8} and Conditions $\mathbf{(1)}$ and $\mathbf{(2)}$ hold if $\mathbf{u}$ satisfies \eqref{bnd1}. \qed
\end{proof}

In order to analyze \eqref{DPC2}, we adopt the following popular model describing the decay rates of $\sigma_{i}$ for different types of ill-posedness \cite{Hansen1998}:
\begin{equation}\label{decay}
	\sigma_{i}=
	\left\{
	\begin{array}{l}
		\zeta\rho^{-i} , \  \rho>1 \ \ \ \mathrm{severely \ ill\mbox{-}posed}; \\
		\zeta i^{-\alpha} , \ \alpha>1  \ \ \ \mathrm{moderately \ ill\mbox{-}posed};  \\
		\zeta i^{-\alpha} , \ 1/2 < \alpha\leq 1  \ \ \ \mathrm{mildly \ ill\mbox{-}posed}.
	\end{array}
	\right. 
\end{equation}

\begin{remark}\label{remark3.1}
	Notice that the model \eqref{decay} means that all the singular values of $A$ are simple. For the case that $A$ has multiple singular values, the model should be rewritten by the following modification; see \cite{Jia2020c} for using this modified model to analyze regularization effect of LSQR for the multiple singular values case. 
	First rewrite the SVD of $A$ as
	\[A=\widehat{U}
	\begin{pmatrix}
		\Sigma \\
		\mathbf{0}
	\end{pmatrix} 
	\widehat{V}^{T} ,\]
	where $\widehat{U}=(\widehat{U}_{1},\dots,\widehat{U}_{r},\widehat{U}_{\perp})$ with $\widehat{U}_{i}\in \mathbb{R}^{m\times l_{i}}$ and $\widehat{V}=(\widehat{V}_{1},\dots,\widehat{V}_{r})$ with $\widehat{V}_{i}\in \mathbb{R}^{n\times l_{i}}$ are column orthonormal, $\Sigma=\mbox{diag}(\hat{\sigma}_{1}I_{l_{1}},\dots,\hat{\sigma}_{r}I_{l_{r}})$ with the $r$ distinct singular values $\hat{\sigma}_{1}>\hat{\sigma}_{2}>\dots>\hat{\sigma}_{r}>0$, each $\hat{\sigma}_{i}$ is $l_{i}$ multiple and $l_{1}+l_{2}+\cdots +l_{r}=n$. Then the decay rate of $\hat{\sigma}_{i}$ can be written in the same form as \eqref{decay}. In this case, the DPC of \eqref{1.1} becomes
	\begin{equation}
		\|\widehat{U}_{i}^{T}b_{ex}\| = \rho_{0}\hat{\sigma}_{i}^{1+\beta}, \ \ i=1,2,\dots,r ,
	\end{equation}
	which states that, on average the (generalized) Fourier coefficients $\| \widehat{U}_{i}^{T}b_{ex}\|$ decay faster than $\hat{\sigma}_{i}$.
\end{remark}

Using the above model, we can give a sufficient condition under which the relation \eqref{DPC2} holds. For notational simplicity, we also write $\hat{\sigma}_i$ as $\sigma_i$ without causing confusions.
\begin{lemma}\label{lem3.4}
	Suppose that the iteration number $k$ is not very big and $\sigma_{i}-\sigma_{i+1} \gg \lVert E \lVert$ for $1\leq i\leq k^{*}$. If \eqref{bnd1} holds and
	\begin{equation}\label{3.16}
		\mathbf{u} \ll (m^{-1/2}\varepsilon)^{\frac{2+\beta}{1+\beta}}\left(\frac{\sigma_{k^{*}}}{\sigma_{k^{*}+1}}-1 \right),
	\end{equation}
	then the relation \eqref{DPC2} holds.
\end{lemma}
\begin{proof}
	There two cases needed to be proved.

	\textbf{Case 1.} $A$ has single singular values. Write the $i$-th left singular vector of $A+E$ as $\bar{u}_{i}=u_{i}+\delta_{u_{i}}$ where $\delta_{u_{i}}$ is an error vector. Since \eqref{bnd1} holds, we have $\bar{\sigma}_{i}\approx\sigma_{i}$ for $1\leq i \leq k^{*}$. Note \eqref{DPC2} implies 
	$|(u_{i}+\delta_{u_{i}})^{T}b| \approx \rho_{0}\bar{\sigma}_{i}^{1+\beta}  \approx  \rho_{0}\sigma_{i}^{1+\beta}$ for $1\leq i \leq k^{*}$,
	which can be satisfied if
	\begin{equation}\label{3.17}
		|\delta_{u_{i}}^{T}b| \ll \rho_{0}\sigma_{i}^{1+\beta} , \ \ 1\leq i\leq k^{*}.
	\end{equation}
	By the perturbation theorem of singular vectors \cite[Theorem 1.2.8]{Bjorck1996}, we have the perturbation bound 
	\[|\sin \theta(u_{i},\bar{u}_{i})|\leq \frac{\lVert E\lVert}{\sigma_{i}-\sigma_{i+1}-\lVert E \lVert}
	\approx \frac{\lVert E\lVert}{\sigma_{i}-\sigma_{i+1}}, \ \ 1\leq i \leq k^{*}\]
	under the assumption that $\sigma_{i}-\sigma_{i+1} \gg \lVert E \lVert$ for $1\leq i\leq k^{*}$, where $\theta(u_{i},\bar{u}_{i})$ is the angle between $u_{i}$ and $\bar{u}_{i}$. Thus we have
	\[\lVert \delta_{u_{i}}\lVert
	= 2|\sin (\theta(u_{i},\bar{u}_{i})/2)|\approx |\sin \theta(u_{i},\bar{u}_{i})| 
	\lesssim \frac{\lVert E\lVert}{\sigma_{i}-\sigma_{i+1}} .\]
	Therefore, \eqref{3.17} can be satisfied if
	\[\frac{\lVert E \lVert \lVert b \lVert}{\sigma_{i}-\sigma_{i+1}} \ll \rho_{0}\sigma_{i}^{1+\beta}, \ \ 1\leq i\leq k^{*},\]
	which is equivalent to
	\begin{equation}\label{3.18}
		\lVert E\lVert \ll \frac{\rho_{0}\sigma_{i}^{2+\beta}\big(1-\frac{\sigma_{i+1}}{\sigma_{i}}\big)}{\lVert b\lVert}, \ \ 1\leq i\leq k^{*}.
	\end{equation}
	Using the expression of $\sigma_{k^{*}+1}$ in \eqref{sig} and the model \eqref{decay}, the minimum of the right-hand term of the above inequality is achieved at $i=k^{*}$, which is
	\begin{align*}
		\frac{\rho_{0}\sigma_{k^{*}}^{2+\beta}\big(1-\frac{\sigma_{k^{*}+1}}{\sigma_{k^{*}}}\big)}{\lVert b\lVert} 
		&= \frac{\rho_{0}\sigma_{k^{*}+1}^{2+\beta}\big(\frac{\sigma_{k^{*}}}{\sigma_{k^{*}+1}}\big)^{2+\beta}\big(1-\frac{\sigma_{k^{*}+1}}{\sigma_{k^{*}}}\big)}{\lVert b\lVert}  \\
		&\approx \frac{\rho_{0}(m^{-1/2}\rho_{0}^{-1}\|e\|)^{\frac{2+\beta}{1+\beta}}\big(\frac{\sigma_{k^{*}}}{\sigma_{k^{*}+1}}\big)^{1+\beta}\big(\frac{\sigma_{k^{*}}}{\sigma_{k^{*}+1}}-1\big)}{\lVert b_{ex}\lVert} \\
		&\geq 
		\frac{(m^{-1/2}\|e\|)^{\frac{2+\beta}{1+\beta}} \big(\frac{\sigma_{k^{*}}}{\sigma_{k^{*}+1}}-1\big)}{\lVert b_{ex}\lVert\rho_{0}^{\frac{1}{1+\beta}}}.
	\end{align*}
	By \eqref{ineq1}, we have
	\begin{align*}
		\frac{(m^{-1/2}\|e\|)^{\frac{2+\beta}{1+\beta}}}{\lVert b_{ex}\lVert\rho_{0}^{\frac{1}{1+\beta}}}
		&= (m^{-1/2})^{\frac{2+\beta}{1+\beta}}(\rho_{0}^{-1}\|e\|)^{\frac{1}{1+\beta}}\frac{\|e\|}{\lVert b_{ex}\lVert} \\
		&\geq (m^{-1/2})^{\frac{2+\beta}{1+\beta}}\varepsilon^{\frac{1}{1+\beta}}\|A\|\varepsilon 
		= (m^{-1/2}\varepsilon)^{\frac{2+\beta}{1+\beta}}\|A\|.
	\end{align*}
	Therefore, by \eqref{3.17} and \eqref{3.18} and using $\lVert E \lVert = \mathcal{O}(c(n,k)\|A\|\mathbf{u})$,  we finally obtain the result.

	\textbf{Case 2.} $A$ has multiple singular values. Write the SVD of $A+E$ in a similar form as that of $A$, such that the left singular vectors can be written as $\bar{U}=(\bar{U}_{1},\dots,\bar{U}_{r},\bar{U}_{\perp})$. By the perturbation theorem of invariant singular subspaces \cite{Wedin1972}, we have
	$$\| \sin \Theta(\widehat{U}_{i}, \bar{U}_{i})\| 
	\leq \frac{\lVert E\lVert}{\hat{\sigma}_{i}-\hat{\sigma}_{i+1}-\lVert E \lVert}
	\approx \frac{\lVert E\lVert}{\hat{\sigma}_{i}-\hat{\sigma}_{i+1}}$$
	where $\| \sin \Theta(\widehat{U}_{i}, \bar{U}_{i})\|=\lVert \widehat{U}_{i}\widehat{U}_{i}^{T}-\bar{U}_{i}\bar{U}_{i}^{T}\lVert$ is the angle measure between subspaces spanned by $\widehat{U}_{i}$ and $\bar{U}_{i}$ \cite[\S 2.5]{Golub2013}. Notice that 
	\begin{align*}
		|\lVert \bar{U}_{i}^{T}b\lVert-\lVert\widehat{U}_{i}^{T}b\lVert| 
		&= |\lVert\bar{U}_{i}\bar{U}_{i}^{T}b\lVert-
		\lVert\widehat{U}_{i}\widehat{U}_{i}^{T}b\lVert| 
		\leq \lVert(\bar{U}_{i}\bar{U}_{i}^{T}-\widehat{U}_{i}\widehat{U}_{i}^{T})b\lVert \\
		&\leq \lVert b\lVert \| \sin \Theta(\widehat{U}_{i}, \bar{U}_{i})\|,
	\end{align*}
	where $|\lVert \bar{U}_{i}^{T}b\lVert-\lVert\widehat{U}_{i}^{T}b\lVert|$ is the corresponding version of the left-hand term of \eqref{3.17}. Using the same approach as that for analyzing \eqref{3.17}, we can obtain the result for the multiple singular values case. \qed
\end{proof}

Using model \eqref{decay}, the minimum of $\sigma_{i}-\sigma_{i+1}$ for $1\leq i\leq k$ is achieved at $i=k^*$. Thus
for $1\leq i \leq k^{*}$, we have
\[\sigma_{i}-\sigma_{i+1}=\sigma_{i+1}\big(\frac{\sigma_{i}}{\sigma_{i+1}}-1\big)\geq
	\left\{
	\begin{array}{l}
		\sigma_{k^{*}+1}(\rho -1)  \ \ \ \mathrm{severely \ ill\mbox{-}posed}; \\
		\sigma_{k^{*}+1}[(\frac{k^{*}+1}{k^{*}})^{\alpha}-1] \ \ \ \mathrm{moderately/mildly \ ill\mbox{-}posed}.
	\end{array}
	\right. \]
By \eqref{sig} and \eqref{ineq1}, we have $\sigma_{k^{*}+1}\gtrsim (m^{-1/2}\varepsilon)^{\frac{1}{1+\beta}}\|A\|$.
Using these two inequalities, one can check that if \eqref{bnd1} and \eqref{3.16} hold, then the assumption that $\sigma_{i}-\sigma_{i+1} \gg \lVert E \lVert$ for $1\leq i\leq k^{*}$ can be satisfied. 

Note that the semi-convergence point $k_0$ of LSQR for \eqref{3.3} is usually not big and thus $\lVert E \lVert = \mathcal{O}(c(n,k_0)\|A\|\mathbf{u})$ with $c(n,k_0)$ a moderate quantity. By Lemma \ref{lem3.3} and Lemma \ref{lem3.4}, if $\mathbf{u}$ satisfies \eqref{bnd1} and \eqref{3.16}, then relations \eqref{DPC2} and \eqref{3.8} holds and the effective resolution limit of \eqref{3.3} has the same accuracy as that of \eqref{1.1}. Therefore by Theorem \ref{lem3.1} $\bar{x}_{k_0}$ as well as $x_{k_0}$ will have the same accuracy as $x_{opt}$, which implies that the best computed regularized solution among $x_k$ can achive the same accuracy as $x_{opt}$. The result is summarized in the following theorem.

\begin{theorem}\label{thm3.1}
	Suppose that the LBFRO in LSQR is implemented in finite precision with roundoff unit $\mathbf{u}$ and \eqref{2.6} is solved exactly. If $\mathbf{u}$ satisfies
	\begin{equation}\label{bnd_u}
		\mathbf{u} \ll \varrho(m^{-1/2}\varepsilon)^{\frac{2+\beta}{1+\beta}}
	\end{equation}
	where
	\begin{equation}
		\varrho = \left\{
		\begin{array}{l}
			\min\{1, \rho -1\} \ \ \ \mathrm{severely \ ill\mbox{-}posed}; \\
			\min \{1, (\frac{k^{*}+1}{k^{*}})^{\alpha}-1 \} \ \ \ \mathrm{moderately/mildly \ ill\mbox{-}posed} ,
		\end{array}
		\right. 
	\end{equation}
	then the best computed regularized solution among $x_k$ can achieve the same accuracy as the best regularized LSQR solution to \eqref{1.1} obtained in exact arithmetic.
\end{theorem}

In Theorem \ref{thm3.1}, the parameters $\alpha$, $\beta$ and $\rho$ are unknown in practical computations. In fact, these parameters are ideal for simplifying singular value decaying and DPC models, and they are closely related to properties of a given ill-posed problem. However, it is instructive from them to get insight into a practical choice of $\mathbf{u}$. For severely ill-posed problems, $\rho-1$ is usually a constant not very small, while for moderately/mildly ill-posed problems, if $\varepsilon$ is very small and $\alpha$ is not big that means the singular values of $A$ decaying very slowly, then $k^{*}$ will be big and thus $(\frac{k^{*}+1}{k^{*}})^{\alpha}-1\approx \alpha/k^{*}$ will be very small. Therefore, for moderately/mildly ill-posed problems, if $\varepsilon$ is very small and $\alpha$ is not big, then \eqref{bnd_u} may give a too small upper bound on $\mathbf{u}$. 

Theorem \ref{thm3.1} implies that for noisy level $\varepsilon$ not very small, we can exploit lower precision for constructing Lanczos vectors in the LSQR for solving \eqref{1.1} without loss of any accuracy of final regularized solutions. We will use numerical examples to show that single precision is enough for the three types of linear ill-posed problem. We need to stress a special practical case that $k^*=n$ for a too small $\varepsilon$ and $\alpha$, which may be encountered in some image deblurring problems. In this case the noise amplification is tolerable even without regularization, which makes $x_{nai}$ a good approximation to $x_{ex}$, and the LSQR solves \eqref{1.1} in their standard manners as if they solved an ordinary other than ill-posed problem. Thus our result can not be applied to this case.

\section{Updating $x_k$ using lower precision}\label{sec4}
In this section, we discuss how to use lower precision for updating $x_k$ step by step. Suppose that the $k$-step LBFRO is implemented using the computing precision chosen as in Theorem \ref{thm3.1}. We first review the procedure for updating $x_k$ from $x_0=\mathbf{0}$ proposed in \cite{Paige1982}. First, the QR factorization  
\begin{equation}\label{givens}
	\hat{Q}_k\begin{pmatrix}
		B_k & \beta_{1}e_{1}^{(k+1)}
	\end{pmatrix}
	=\begin{pmatrix}
		R_k & f_k \\
		    & \bar{\phi}_{k+1}
	\end{pmatrix}=
	\left(\begin{array}{ccccc:c}		
		\rho_1& \theta_2 &  &  &  & \phi_1 \\		
		& \rho_2 & \theta_3 & & & \phi_2 \\		
		& & \ddots & \ddots & & \vdots \\ 		
		& & & \rho_{k-1} & \theta_k & \phi_{k-1} \\		
		& & &  & \rho_k & \phi_k \\	\hdashline	
		& & &  &  & \bar{\phi}_{k+1}	
	\end{array}\right)
\end{equation}
is performed using a series of Givens rotations, where at the $i$-th step the Givens rotation is chosen to zero out $\beta_{i+1}$:
\[\begin{pmatrix}
	c_i & s_i \\
	s_i & -c_i
\end{pmatrix}
\begin{pmatrix}
	\bar{\rho}_i & 0 & \bar{\phi}_i  \\
	\beta_{i+1} & \alpha_{i+1} & 0
\end{pmatrix}
\begin{pmatrix}
	\rho_i & \theta_{i+1} & \phi_i \\
	0 & \bar{\rho}_{i+1}  & \bar{\phi}_{i+1}
\end{pmatrix},\]
and the orthogonal matrix $Q_k$ is the product of these Givens rotation matrices. Since
\begin{equation}\label{res}
	\|B_ky-\beta_1e_{1}^{(k+1)}\|^2=\Big\|\hat{Q}_k\begin{pmatrix}
		B_k & \beta_{1}e_{1}^{(k+1)}
	\end{pmatrix}\begin{pmatrix}
		y \\ -1
	\end{pmatrix}\Big\|^2=\|R_ky-f_k\|^2+|\bar{\phi}_{k+1}|^2,
\end{equation}
the solution to $\min_{y\in\mathbb{R}^k}\|B_ky-\beta_{1}e_{1}^{(k)}\|$ is $y_k=R_{k}^{-1}f_k$. Factorize $R_k$ as
\[R_k=D_k\widehat{R}_k, \ \ 
D_k=\begin{pmatrix}
	\rho_1 & & & \\
	& \rho_2 & & \\
	& & \ddots & \\
	& & & \rho_k
\end{pmatrix}, \ 
\widehat{R}_k=\begin{pmatrix}
	1 & \theta_2/\rho_1 & & \\
	& 1 & \theta_3/\rho_2 & \\
	& & \ddots & \theta_k/\rho_{k-1} \\
	& & & 1
\end{pmatrix},\]
then we get
\begin{equation}\label{up1_x}
	x_k=Q_k y_k=Q_k R_{k}^{-1}f_k=(Q_k\widehat{R}_{k}^{-1})(D_{k}^{-1}f_k).
\end{equation}
Let $W_k=Q_k\widehat{R}_{k}^{-1}=(w_1,\dots,w_k)$. By using back substitution for solving $W_k\widehat{R}_k=Q_k$ we obtain the updating procedure for $x_i$ and $w_i$: 
\begin{equation}\label{up2_x}
	x_{i}=x_{i-1}+(\phi_{i}/\rho_{i})w_{i}, \ \ \  
	w_{i+1}=q_{i+1}-(\theta_{i+1}/\rho_{i})w_{i},
\end{equation}
which is described in Algorithm \ref{alg2}.

\begin{algorithm}[htb]
	\caption{Updating procedure}\label{alg2}
	\begin{algorithmic}[1]
		\State {Let $x_{0}=\mathbf{0},\ w_{1}=q_{1}, \ \bar{\phi}_{1}=\beta_{1},\ \bar{\rho}_{1}=\alpha_{1}$}
		\For{$i=1,2,\ldots,k,$}
		\State $\rho_{i}=(\bar{\rho}_{i}^{2}+\beta_{i+1}^{2})^{1/2}$
		\State $c_{i}=\bar{\rho}_{i}/\rho_{i},\ s_{i}=\beta_{i+1}/\rho_{i}$
		\State$\theta_{i+1}=s_{i}\alpha_{i+1},\ \bar{\rho}_{i+1}=-c_{i}\alpha_{i+1}$
		\State $\phi_{i}=c_{i}\bar{\phi}_{i},\ \bar{\phi}_{i+1}=s_{i}\bar{\phi}_{i} $
		\State $x_{i}=x_{i-1}+(\phi_{i}/\rho_{i})w_{i} $
		\State $w_{i+1}=q_{i+1}-(\theta_{i+1}/\rho_{i})w_{i}$
		\EndFor
	\end{algorithmic}
\end{algorithm}


From the above description, the procedure of updating $x_k$ is constituted of two parts: the Givens QR factorization and the computation of $x_i$ and $w_{i+1}$. First, we investigate the choice of proper computing precision for the Givens QR factorization. Denote the roundoff unit used in this process by $\tilde{\mathbf{u}}$ and assume the computations of other parts are exact. In finite precision arithmetic, after the above process, we have computed the $k$-th iterative solution $\tilde{x}_k=Q_k\tilde{y}_k$ with $\tilde{R}_k\tilde{y}_k=\tilde{f}_k$, where $\begin{pmatrix}
	\tilde{R}_k & \tilde{f}_k \\
		& \tilde{\phi}_{k+1}
\end{pmatrix}$
is the computed $R$-factor of $\begin{pmatrix}
	B_k & \beta_{1}e_{1}^{(k+1)}
\end{pmatrix}$. Using the backward error analysis result about the Givens QR factorization as \eqref{givens}, there exist an orthogonal matrix $\tilde{Q}_k\in\mathbb{R}^{(k+1)\times (k+1)}$ such that 
\begin{equation}\label{givens_pertur}
	\tilde{Q}_k\left[\begin{pmatrix}
		B_k & \beta_{1}e_{1}^{(k+1)}
	\end{pmatrix} + \begin{pmatrix}
		\Delta_{k}^{B} & \delta_{k}^{\beta}
	\end{pmatrix}\right] = \begin{pmatrix}
		\tilde{R}_k & \tilde{f}_k \\
			& \tilde{\phi}_{k+1}
	\end{pmatrix} 
\end{equation}
where $\tilde{R}_k\in\mathbb{R}^{k \times k}$ is upper triangular and
\[\left\|\Delta_{k}^{B}\right\| / \left\|B_k\right\| \leq c_1(k)\tilde{\mathbf{u}}+\mathcal{O}(\tilde{\mathbf{u}}^2), \ \ \
\|\delta_{k}^{\beta}\|/\beta_1  \leq c_1(k)\tilde{\mathbf{u}}+\mathcal{O}(\mathbf{u}^2) \]
with a moderate value $c_{1}(k)$ depending on $k$; see \cite[Theorem 19.10]{Higham2002}. The above relation means that $\tilde{Q}_k$ is the $Q$-factor of a perturbed $\begin{pmatrix}
	B_k & \beta_{1}e_{1}^{(k+1)}
\end{pmatrix}$, and $\tilde{R}_k$ is the $R$-factor of a perturbed $B_k$. Using the perturbation analysis result about QR factorizations, we have
\begin{align}
	\frac{\|\tilde{R}_k-R_k\|}{\|R_k\|} &\leq c_2(k)\kappa(R_k)\tilde{\mathbf{u}} + \mathcal{O}(\tilde{\mathbf{u}}^2), \\
	\|\tilde{Q}_k-\hat{Q}_k\| &\leq c_3(k)\kappa(R_k)\tilde{\mathbf{u}} + \mathcal{O}(\tilde{\mathbf{u}}^2),
\end{align}
where $\kappa(R_k)=\|R_k\|\|R_{k}^{-1}\|$ is the condition number of $R_k$, and $c_{2}(k)$ and $c_{3}(k)$ are two moderate values depending on $k$; see \cite[\S 19.9]{Higham2002}. Note that the $F$-norm result appeared in \cite[\S 19.9]{Higham2002} also applies to the above 2-norm result besides a difference of multiplicative factor depending on $k$. Thus, we have
\begin{align}
	\|\tilde{f}_k-f_k\|
	&\leq \left\|\tilde{Q}_{k}\left(\beta_{1}e_{1}^{(k+1)}+ \delta_{k}^{\beta}\right) -\hat{Q}_k\beta_{1}e_{1}^{(k+1)} \right\| \nonumber \\
	&\leq \beta_1\|\tilde{Q}_k-\hat{Q}_k\| + \|\tilde{Q}_{k}\delta_{k}^{\beta}\| \nonumber \\
	&\leq \beta_1\left[c_3(k)\kappa(R_k)+c_1(k)\right]\tilde{\mathbf{u}} + \mathcal{O}(\tilde{\mathbf{u}}^2). \label{er_fk}
\end{align}
Note that $R_ky_k=f_k$. Using the perturbation analysis result about this linear system \cite[\S 2.6.4]{Golub2013}, if $\|R_{k}^{-1}(\tilde{R}_k-R_k)\|<1$, we get
\begin{align*}
	\frac{\|\tilde{y}_k-y_k\|}{\|y_k\|}
	&\leq \frac{\kappa(R_k)}{1-\kappa(R_k)\frac{\|\tilde{R}_k-R_k\|}{\|R_k\|}}\left(\frac{\|\tilde{R}_k-R_k\|}{\|R_k\|}+\frac{\|\tilde{f}_k-f_k\|}{\|f_k\|}\right) \\
	&\leq \frac{\kappa(R_k)}{1-c_2(k)\kappa(R_k)^2\tilde{\mathbf{u}}}\left(c_2(k)\kappa(R_k)+\frac{\beta_1c_3(k)\kappa(R_k)+c_1(k)}{(\beta_{1}^2-\bar{\phi}_{k+1}^2)^{1/2}}\right)\tilde{\mathbf{u}}+ \mathcal{O}(\tilde{\mathbf{u}}^2).
\end{align*}
Notice that $x_k=Q_ky_k$ and $\tilde{x}_k=Q_k\tilde{y}_k$, where $Q_k$ is computed by LBFRO in finite precision with roundoff unit $\mathbf{u}$. We get $\|\tilde{x}_k-x_k\|\leq \|Q_k\|\|\tilde{y}_k-y_k\|\leq \|\tilde{y}_k-y_k\|\left(1+\mathcal{O}(\mathbf{u})\right)$, where we have used $\|Q_k-\bar{Q}_k\|=\mathcal{O}(\mathbf{u})$ and $\|\bar{Q}_k\|=1$ by \eqref{2.15}. Using Theorem \ref{lem3.1} and $\|y_k\|=\|\bar{Q}_{k}y_k\|=\|\bar{x}_k\|$, we obtain
\begin{align}
	& \ \ \ \ \frac{\|\tilde{x}_k-x_k\|}{\|x_k\|} 
	= \frac{\|\tilde{x}_k-x_k\|}{\|\bar{x}_k\|}\frac{\|\bar{x}_k\|}{\|x_k\|} 
	\leq \frac{\|\tilde{y}_k-y_k\|\left(1+\mathcal{O}(\mathbf{u})\right)}{\|y_k\|} \left(1+\mathcal{O}(\mathbf{u})\right)  \nonumber \\
	&\leq \frac{\kappa(R_k)}{1-c_2(k)\kappa(R_k)^2\tilde{\mathbf{u}}}\left(c_2(k)\kappa(R_k)+\frac{\beta_1c_3(k)\kappa(R_k)+c_1(k)}{(\beta_{1}^2-\bar{\phi}_{k+1}^2)^{1/2}}\right)\tilde{\mathbf{u}} \nonumber \\
	& \ \ \ \ + \mathcal{O}(\mathbf{u}\tilde{\mathbf{u}}+\tilde{\mathbf{u}}^2). \label{bnd0}
\end{align} 
The upper bound \eqref{bnd0} grows up at the speed of $\kappa(R_k)^2\tilde{\mathbf{u}}=\kappa(B_k)^2\tilde{\mathbf{u}}$. By Theorem \ref{thm2.1}, $B_k$ is the projection of $A+E$ on $\mathrm{span}\{\bar{P}_{k+1}\}$ and $\mathrm{span}\{\bar{Q}_{k}\}$, and it gradually becomes ill-conditioned since $A+E$ is a slight perturbation of the ill-conditioned matrix $A$. This implies that a lower computing precision used by the Givens QR factorization will lead to a loss of accuracy of the computed solution. Therefore, in practical computation, we need to use double precision to perform it. Fortunately, the Givens QR factorization (Line 3--6 in Algorithm \ref{alg2}) can always be performed very quickly since only operations of scalars are involved. In contrast, the updating of $x_i$ and $w_{i+1}$ involves vector operations, thereby it is better to be performed using lower precision. The proper choice of the computing precision for it will be analyzed in the following part.

In practical computation of the iterative regularization algorithm, the $k$-step LBFRO need not to be implemented in advance, while it should be done in tandem with the updating procedure. An early stopping criterion such as DP or L-curve criterion is used to estimate the semi-convergence point. The whole process can be summarized in Algorithm \ref{alg3} as a mixed precision variant of LSQR for linear ill-posed problems.

\begin{algorithm}[htb]
	\caption{Mixed precision variant of LSQR for \eqref{1.1}}\label{alg3}
	\begin{algorithmic}[1]
		\Require $A$, $b$, $x_{0}=\mathbf{0}$
		\For{$k=1,2,\ldots,$}
		\State Compute $p_k$, $q_k$, $\alpha_k$, $\beta_k$ by the LBFRO  \algorithmiccomment{roundoff unit is $\mathbf{u}$}
		\State Compute $\rho_k$, $\theta_{k+1}$, $\bar{\rho}_{k+1}$, $\phi_{k}$, $\bar{\phi}_{k+1}$ by the updating procedure  \algorithmiccomment{double precision}
		\State Compute $x_k$, $w_{k+1}$ by the updating procedure  \algorithmiccomment{roundoff unit is $\mathbf{\bar{u}}$}
		\If{Early stopping criterion is satisfied}  \algorithmiccomment{DP or L-curve criterion}
		\State The semi-convergence point is estimated as $k_1$
		\State Terminate the iteration
		\EndIf
		\EndFor
		\Ensure Final regularized solution $\hat{x}_{k_1}$  \algorithmiccomment{Computed solution corresponding to $x_{k_1}$}
	\end{algorithmic}
\end{algorithm}

To analyze the choice of $\mathbf{\bar{u}}$, we use the following model \cite[\S 2.7.3]{Golub2013} for the floating point arithmetic:
\begin{align}\label{model1}
	\mbox{fl}(a \ \mbox{op} \ b)=(a \ \mbox{op} \ b)(1+\epsilon), \ \ \
	|\epsilon|\leq \mathbf{\bar{u}}, \ \ \mbox{op}=+,-,*,/ .
\end{align}
Under this model, we have the following rounding error results for matrix and vector computations \cite[\S 2.7.8]{Golub2013}:
\begin{align}
	& \mbox{fl}(u+\alpha v)=u+\alpha v+w, \ \ \ 
	|w|\leq (|u|+2|\alpha v|)\mathbf{\bar{u}}+\mathcal{O}(\mathbf{\bar{u}}^2),\label{model2} \\
	& \mbox{fl}(AB)=AB+X, \ \ \ 
	|X|\leq n|A||B|\mathbf{\bar{u}}+\mathcal{O}(\mathbf{\bar{u}}^2). \label{model3}
\end{align}
where $u,v$ are vectors, $\alpha$ is a scalar, and $A, B$ are two matrices of orders $m\times n$ and $n\times l$, respectively. In \eqref{model2} and \eqref{model3}, the notation $|\cdot|$ is used to denote the absolute value of a matrix or vector and $``\leq"$ means the relation $``\leq"$ holds componentwise. We remark that \eqref{model3} applies to both dot-product and outer-product based procedures for matrix multiplications; see \cite[\S 1.1, \S 2.7]{Golub2013} for these two types of computation of matrix multiplications and the corresponding rounding error analysis results.

Denote by $\hat{x}_k$ and $\hat{w}_k$ the computed quantities where the roundoff unit of the  computing precision is $\mathbf{\bar{u}}$ and the Givens QR factorization is performed with double precision. To avoid cumbersome using of notations, in the following analysis, notations such as $v_k$, $f_k$, $x_k$, $w_k$ denote the computed quantities for the process that the LBFRO is implemented in finite precision with roundoff unit $\mathbf{u}$ while the Givens QR factorization is implemented in double precision and other computations are exact. Using the above model, the computation of $\hat{w}_i$ in finite precision arithmetic can be formed as:
\[\left\{
		\begin{array}{l}
			\hat{w}_1 = v_1+\delta_{1}^{w} \\
			\hat{w}_2 =v_{2}-(\theta_{2}/\rho_{1})\hat{w}_{1}+\delta_{2}^{w} \\
			\ \ \ \ \ \vdots \\
			\hat{w}_k =v_{k}-(\theta_{k}/\rho_{k-1})\hat{w}_{k-1}+\delta_{k}^{w}
			\end{array}
\right. \]
where 
\[\|\delta_{1}^{w}\|\leq \|v_1\|\mathbf{\bar{u}}+\mathcal{O}(\mathbf{\bar{u}}^2)=\mathbf{\bar{u}}+\mathcal{O}(\mathbf{\bar{u}}^2)\]
and 
\[\|\delta_{i}^{w}\|\leq (\|v_i\|+2\|(\theta_{i}/\rho_{i-1})\hat{w}_{i-1}\|)\mathbf{\bar{u}}+\mathcal{O}(\mathbf{\bar{u}}^2)= (1+2\|(\theta_{i}/\rho_{i-1})\hat{w}_{i-1}\|)\mathbf{\bar{u}}+\mathcal{O}(\mathbf{\bar{u}}^2).\]
Let $h_i=\hat{w}_i-w_i$. Then we have
\[h_{i+1}=v_{i+1}-(\theta_{i+1}/\rho_{i})\hat{w}_{i}+\delta_{i}^{w}-(v_{i+1}-(\theta_{i+1}/\rho_{i})w_{i})=-(\theta_{i+1}/\rho_{i})h_{i}+\delta_{i}^{w},\]
which leads to
\begin{equation}\label{4.7}
	H_k\widehat{R}_k = \Delta_{k}^{w}
\end{equation}
with $H_k=(h_{1},\dots,h_{k})$ and $\Delta_{k}^{w}=(\delta_{1}^{w},\dots,\delta_{k}^{w})$. Therefore, we have
\begin{align*}
	\|\Delta_{k}^{w}\| &\leq \sqrt{k}\max_{1\leq i \leq k}\|\delta_{i}^{w}\| 
	\leq \sqrt{k}[1+2\max_{1\leq i \leq k}(\theta_{i+1}/\rho_{i})\|w_i+h_i\|]\mathbf{\bar{u}}+\mathcal{O}(\mathbf{\bar{u}}^2) \\
	&\leq \sqrt{k}[1+2\|\widehat{R}_k\|(\|W_k\|+\|\Delta_{k}^{w}\|\|\widehat{R}_{k}^{-1}\|)]\mathbf{\bar{u}}+\mathcal{O}(\mathbf{\bar{u}}^2),
\end{align*}
where we have used $|\theta_{i+1}/\rho_{i}|\leq\|\widehat{R}_k\|$ and $\|h_i\|\leq \|H_{k}\|\leq \|\Delta_{k}^{w}\|\|\widehat{R}_{k}^{-1}\|$. This inequality leads to
\[(1-2\sqrt{k}\kappa(\widehat{R}_k)\mathbf{\bar{u}})\|\Delta_{k}^{w}\| \leq
	\sqrt{k}(1+2\|\widehat{R}_k\|\|W_k\|)\mathbf{\bar{u}}+\mathcal{O}(\mathbf{\bar{u}}^2),\]
where $\kappa(\widehat{R}_k)=\|\widehat{R}_k\|\|\widehat{R}_{k}^{-1}\|$ is the condition number of $\widehat{R}_k$. Notice that $\|W_{k}\|\leq\|Q_{k}\|\|\widehat{R}_{k}^{-1}\|=\|\widehat{R}_{k}^{-1}\|+\mathcal{O}(\mathbf{u})$ where we have used $\|Q_k-\bar{Q}_k\|=\mathcal{O}(\mathbf{u})$ and $\|\bar{Q}_k\|=1$ by \eqref{2.15}. We obtain the upper bound on $\|\Delta_{k}^{w}\|$:
\begin{equation}\label{4.8}
	\|\Delta_{k}^{w}\| \leq
	\sqrt{k}[1+2(1+\sqrt{k})\kappa(\widehat{R}_k)]\mathbf{\bar{u}}+\mathcal{O}(\mathbf{\bar{u}}^2+\mathbf{\bar{u}}\mathbf{u}).
\end{equation}

Now we can analyze the accuracy of $\hat{x}_k$.
\begin{theorem}\label{thm4.1}
	In Algorithm \ref{alg3}, denote by $x_k$ the computed solution where the Givens QR factorization is performed in double precision and other computations of the updating procedure are exact. Then at each iteration we have
	\begin{equation}\label{4.9}
		\dfrac{\|\hat{x}_k-x_k\|}{\|x_k\|} \leq \sqrt{k}[1+(2+2\sqrt{k}+k)\kappa(\widehat{R}_k)]\mathbf{\bar{u}}+
		\mathcal{O}(\mathbf{\bar{u}}^2+\mathbf{\bar{u}}\mathbf{u}).
	\end{equation}		
\end{theorem}	
\begin{proof}
	Notice from \eqref{up1_x} and \eqref{up2_x} that the formation of $\hat{x}_k$ is the matrix multiplication between $\widehat{W}_k=(\hat{w}_1,\dots,\hat{w}_k)$ and $D_{k}^{-1}f_k$ by the outer-product based procedure. Using model \eqref{model3} we have
	\begin{equation}\label{err_xk}
		\hat{x}_k-x_k = \widehat{W}_kD_{k}^{-1}f_k+\Delta_{k}^{x}-W_kD_{k}^{-1}f_k,
	\end{equation}
	where $\hat{x}_k=\widehat{W}_k(D_{k}^{-1}f_k)+\Delta_{k}^{x}$ with $|\Delta_{k}^{x}|\leq k|\widehat{W}_k||D_{k}^{-1}f_k|\mathbf{\bar{u}}+\mathcal{O}(\mathbf{\bar{u}}^2)$. Therefore, we have
	\begin{align*}
		\|\Delta_{k}^{x}\| &\leq \||\Delta_{k}^{x}|\|
		\leq k\||\widehat{W}_k|\|_F\||D_{k}^{-1}f_k|\|\mathbf{\bar{u}}+\mathcal{O}(\mathbf{\bar{u}}^2)\\
		&= k\|\widehat{W}_k\|_F\|D_{k}^{-1}f_k\|\mathbf{\bar{u}}+\mathcal{O}(\mathbf{\bar{u}}^2)\\
		&\leq k^{3/2}\|\widehat{W}_k\|\|D_{k}^{-1}f_k\|\mathbf{\bar{u}}+\mathcal{O}(\mathbf{\bar{u}}^2),
	\end{align*}
	where $\|\cdot\|_F$ is the Frobenius norm of a matrix. By \eqref{4.7}, we have $\widehat{W}_k-W_k=H_k=\Delta_{k}^{w}\widehat{R}_{k}^{-1}$, and thus \[\|\widehat{W}_k\|\leq\|W_k\|+\|H_k\|\leq\|W_k\|+\|\Delta_{k}^{w}\|\|\widehat{R}_{k}^{-1}\|.\]
	Substituting it into the inequality about $\|\Delta_{k}^{x}\|$ and noticing $\widehat{R}_{k}^{-1}D_{k}^{-1}f_k=R_{k}^{-1}f_k=y_k$, we obtain
	\begin{align*}
		\|\Delta_{k}^{x}\| &\leq  k^{3/2}(\|W_k\|+\|\Delta_{k}^{w}\|\|\widehat{R}_{k}^{-1}\|)\|\widehat{R}_k(\widehat{R}_{k}^{-1}D_{k}^{-1}f_k)\|
		\mathbf{\bar{u}}+\mathcal{O}(\mathbf{\bar{u}}^2) \\
		&\leq k^{3/2}\kappa(\widehat{R}_k)\|y_k\|\mathbf{\bar{u}}+\mathcal{O}(\mathbf{\bar{u}}^2+\mathbf{\bar{u}}\mathbf{u}),
	\end{align*}
	where we have used $\|W_{k}\|\leq\|\widehat{R}_{k}^{-1}\|+\mathcal{O}(\mathbf{u})$.
	Using $\widehat{W}_k-W_k=\Delta_{k}^{w}\widehat{R}_{k}^{-1}$ again, we get
	\[\|\widehat{W}_kD_{k}^{-1}f_k-W_kD_{k}^{-1}f_k\| = \|(\widehat{W}_k-W_k)D_{k}^{-1}f_k\|=\|\Delta_{k}^{w}(\widehat{R}_{k}^{-1}D_{k}^{-1}f_k)\| \leq \|\Delta_{k}^{w}\|\|y_k\|. \]
	By \eqref{err_xk} and combining with \eqref{4.8}, we obtain
	\begin{align*}
		\|\hat{x}_k-x_k\|
		&\leq \|\Delta_{k}^{x}\| + \|\widehat{W}_kD_{k}^{-1}f_k-W_kD_{k}^{-1}f_k\| \\
		&\leq \sqrt{k}[1+(2+2\sqrt{k}+k)\kappa(\widehat{R}_k)]\|y_k\|\mathbf{\bar{u}}+\mathcal{O}(\mathbf{\bar{u}}^2+\mathbf{\bar{u}}\mathbf{u}).
	\end{align*}
	Since $\|Q_k-\bar{Q}_k\|=\mathcal{O}(\mathbf{u})$ and $\bar{Q}_k$ is orthonormal, we have $\|Q_{k}^{-1}\|\leq 1/(1-\mathcal{O}(\mathbf{u}))=1+\mathcal{O}(\mathbf{u})$. Using the relations $\|y_k\|=\|Q_{k}^{-1}x_k\|\leq\|x_k\|\left(1+\mathcal{O}(\mathbf{u})\right)$ and $\|\hat{x}_k-x_k\|/\|x_k\|=\|\hat{x}_k-x_k\|/\|y_k\|\cdot \|y_k\|/\|x_k\|$, we finally obtain \eqref{4.9}. \qed
\end{proof}	

Note that $\widehat{R}_k$ is obtained by scaling $R_k$ using the diagonal of it, and the diagonal scaling step can often dramatically reduces the condition number of $R_k$. In fact, we will show in the numerical experiments section that $\kappa(\widehat{R}_k)$ is a moderate value even for an iteration $k$ bigger than the semi-convergence point. By Theorem \ref{thm4.1}, if $\mathbf{u}$ has been chosen such that the best solution among $x_{k}$ can achieve the same accuracy as the best LSQR regularized solution to \eqref{1.1} obtained in exact arithmetic, in order to make the practical updated $\hat{x}_k$ can also achieve the same accuracy, $\mathbf{\bar{u}}$ should be chosen such that the upper bound in \eqref{4.9} is much smaller than $\|x_{opt}-x_{ex}\|/\|x_{ex}\|$. Thanks to \eqref{3.7}, for a not very small noise level, the  single precision roundoff unit $\mathbf{\bar{u}}$ is enough. This ensures that we can use lower precision for updating $x_k$, which is more efficient than using double precision. 

Now we discuss methods for estimating the optimal early stopping iteration, i.e., the semi-convergence point. By \ref{lem3.1} and \eqref{res} we have
\[\bar{\phi}_{k+1}=\|B_ky_{k}-\beta_1 e_{1}^{(k+1)}\|= \lVert (A+E)x_k-(b+\delta_b) \lVert .\]
For the proper choice of $\mathbf{u}$, the noise norm of \eqref{3.2} is 
$\|e-Ex_{ex}+\delta_b\| \approx \|e\|$ since $\|-Ex_{ex}+\delta_b\| \ll \|e\|$. Since $\hat{x}_k$, $x_k$ and $\bar{x}_k$ have the same accuracy for the proper $\mathbf{u}$ and $\bar{\mathbf{u}}$, we only need to estimate the semi-convergence point of LSQR applied to \eqref{3.2}. The discrepancy principle corresponding to \eqref{3.2} can be written as
\[\lVert (A+E)x_k-(b+\delta_b) \lVert \lesssim \tau\|e\| \]
with $\tau> 1$ slightly, and we should stop iteration at the first $k$ satisfying 
\begin{equation}\label{discrepancy}
	\bar{\phi}_{k+1}=\|B_ky_{k}-\beta_1 e_{1}^{(k+1)}\| \leq \tau\|e\|,
\end{equation}
and use this $k$ as the estimate of semi-convergence point, where $\bar{\phi}_{k+1}$ can be efficiently computed by using \ref{alg2}. Numerical experiments will show that this estimate is almost the same as that obtained by the discrepancy principle for LSQR in double precision arithmetic. The discrepancy principle method usually suffers from under-estimating and thus the solution is over-regularized.

Another approach is the L-curve criterion, which does not need $\|e\|$ in advance. The motivation is that one can plot $\left(\log\|Ax_{k}-b\|,\log\|x_{k}\|\right)$ in the shape of an L-curve, and the corner of the curve is a good estimate of the semi-convergence point. The L-curve for \eqref{3.2} is
\[\left(\log\|(A+E)x_k-
	(b+\delta_b)\|,\log\|x_{k}\|\right) ,\]
which is just 
\begin{equation}\label{5.5}
	\left(\log\bar{\phi}_{k+1},\log\|x_{k}\|\right) ,
\end{equation}
where the norm of $x_k$ should be computed at each iteration. A modification of \eqref{5.5} computes the norm of $\hat{x}_k$ instead of $x_k$, and this may make a little difference with the estimate by \eqref{5.5}. Numerical experiments will show that these two estimates are almost the same as that obtained by the L-curve criterion for LSQR in double precision arithmetic. 

Finally, we give a model for comparing computing efficiency between the double and mixed precision implementations of LSQR. We also perform full reorthogonalization of the Lanczos bidiagonalization for the double precision implementation, since without reorthogonalization the convergence behavior is irregular and the convergence rate is much slow. We count the computations involving matrix/vector operations in the two main parts of the algorithm:
\begin{itemize}
	\item[$\bullet$] For the LBFRO process, at each step it takes $\mathcal{O}(mn)$ flops for matrix-vector products and $\mathcal{O}(m+n)$ flops for scalar-vector multiplications; besides, the reorthogonalization at the $k$-th step takes $\mathcal{O}((m+n)k^2)$ flops. Therefore, at each $k$-th iteration, LBFRO takes $\mathcal{O}(mn+(m+n)(k^2+1))$ flops.
	\item[$\bullet$] For the updating procedure, the most time-consuming part is the computation of $x_i$ and $w_{i+1}$, and it takes $\mathcal{O}(n)$ flops.
\end{itemize}
From the above investigation, we find that the matrix-vector products in LBFRO are the most dominant computations in the entire algorithm. In the ideal case, the performance of 32-bit operations is at least twice as fast as that of 64-bit operations on modern computing architectures \cite{abdelfattah2021survey}. Therefore, the proposed mixed precision algorithm can save approximately half the time compared to the original double precision algorithm.

In pracital computations, to give a convincing comparison between the two implementations, the mixed precision algorithm need to be performed on a specific computing architecture supporting well for lower precision computations such as NVIDIA Tesla V100 GPU \cite{V100}, and the codes should be optimized to take full advantage of the computing power. This will de considered in our future work.

\section{Numerical Experiments}\label{sec5}
In this section, we present some numerical experiments to justify the theoretical results obtained. Two mixed precision variants of LSQR are implemented to be compared with the double precision LSQR for several test linear ill-posed problems. We use ``d" to denote the algorithm implemented using double precision, and use ``s+d" and ``s+s" to denote the algorithms that use single precision for LBFRO while use double and single precisions for updating $x_k$, respectively. Note that for ``d'' the Lanczos bidiagonalization is also implemented using full reorthogonalization to avoid delay of convergence. 

For these different implementations, we compare accuracy of the regularized solutions by using the relative reconstruction error
\begin{equation}
	\mathrm{RE}(k) = \frac{\|x_{k}-x_{ex}\|}{\|x_{ex}\|}
\end{equation}
to plot semi-convergence curves, where $x_k$ (for ``s+s'' it should be $\hat{x}_k$) denote the computed solutions produced by the three implementations. We emphasis that computing efficiency in terms of time-to-solution between ``d'', ``s+d'' and ``s+s'' is not compared here, since the purpose of this paper is to verify the feasibility of lower precision LSQR. 

In this paper, we implement the MATLAB codes with MATLAB R2019b to perform numerical experiments, where the roundoff units for double and single precision are $2^{-53}\approx 1.11\times 10^{-16}$ and $2^{-24}\approx 5.96\times 10^{-8}$, respectively. The codes are available at \url{https://github.com/Machealb/Lower_precision_solver}. We choose some one dimensional (1-D) problems from the regularization toolbox \cite{Hansen2007}, and two dimensional (2-D) image deblurring problems from \cite{Gazzola2019}. The description of all test examples is listed in Table \ref{tab5.1}. 

\begin{table}[htp]
	\centering
	\caption{The description of test problems}
	\begin{tabular}{llll}
		\toprule
		Problem     & $m\times n$  & Ill-posedness & Description                \\
		\midrule
		{\sf shaw}  & $1000\times 1000$ & severe & 1-D image restoration model      \\
		{\sf deriv2} & $1000\times 1000$ & moderate & Computation of second derivative  \\
		{\sf gravity}  & $2000\times 2000$ & severe & 1-D gravity surveying problem    \\
		{\sf heat} & $2000\times 2000$  & moderate & Inverse heat equation     \\
		{\sf PRblurspeckle}  & $16384\times 16384$ & mild & 2-D image deblurring problem \\
		{\sf PRblurdefocus} & $65536\times 65536$ & mild  & 2-D image deblurring problem  \\
		\bottomrule
	\end{tabular}
	\label{tab5.1}
\end{table}

\subsection{One dimensional case}\label{sebsec5.1}
For one dimensional problems {\sf shaw}, {\sf deriv2}, {\sf gravity} and {\sf heat}, we use the codes from \cite{Hansen2007} to generate $A$, $x_{ex}$ and $b_{ex}=Ax_{ex}$, and then add a white Gaussian noise $e$ with a prescribed noise level $\varepsilon = \|e\|/\|b_{ex}\|$ to $b_{ex}$ and form the noisy $b=b_{ex}+e$. 

\begin{figure}[htp]
	\begin{minipage}{0.48\linewidth}
		\centerline{\includegraphics[width=5cm,height=3.5cm]{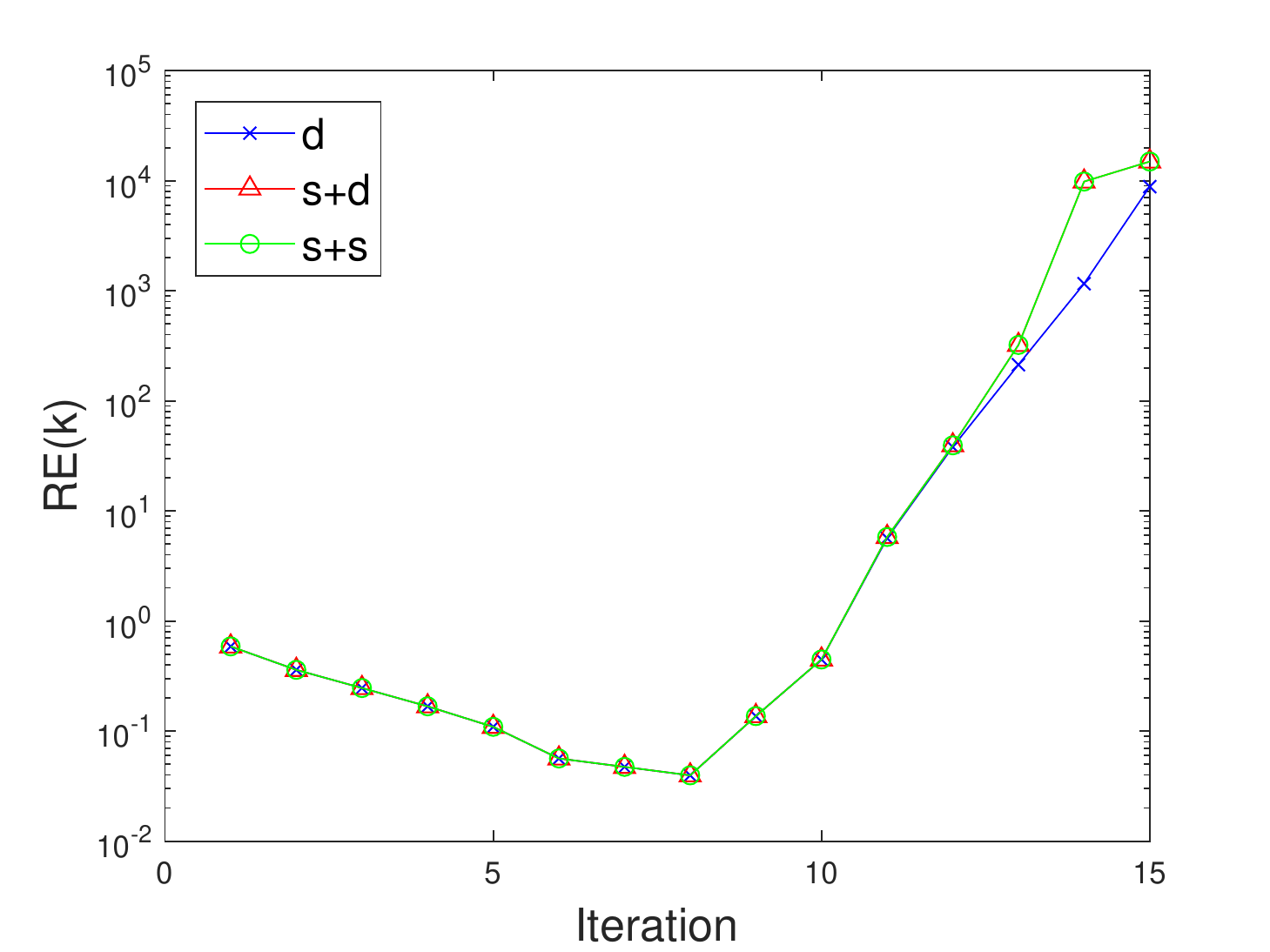}}
		\centerline{(a) {\sf shaw}}
	\end{minipage}
	\hfill
	\begin{minipage}{0.48\linewidth}
		\centerline{\includegraphics[width=5cm,height=3.5cm]{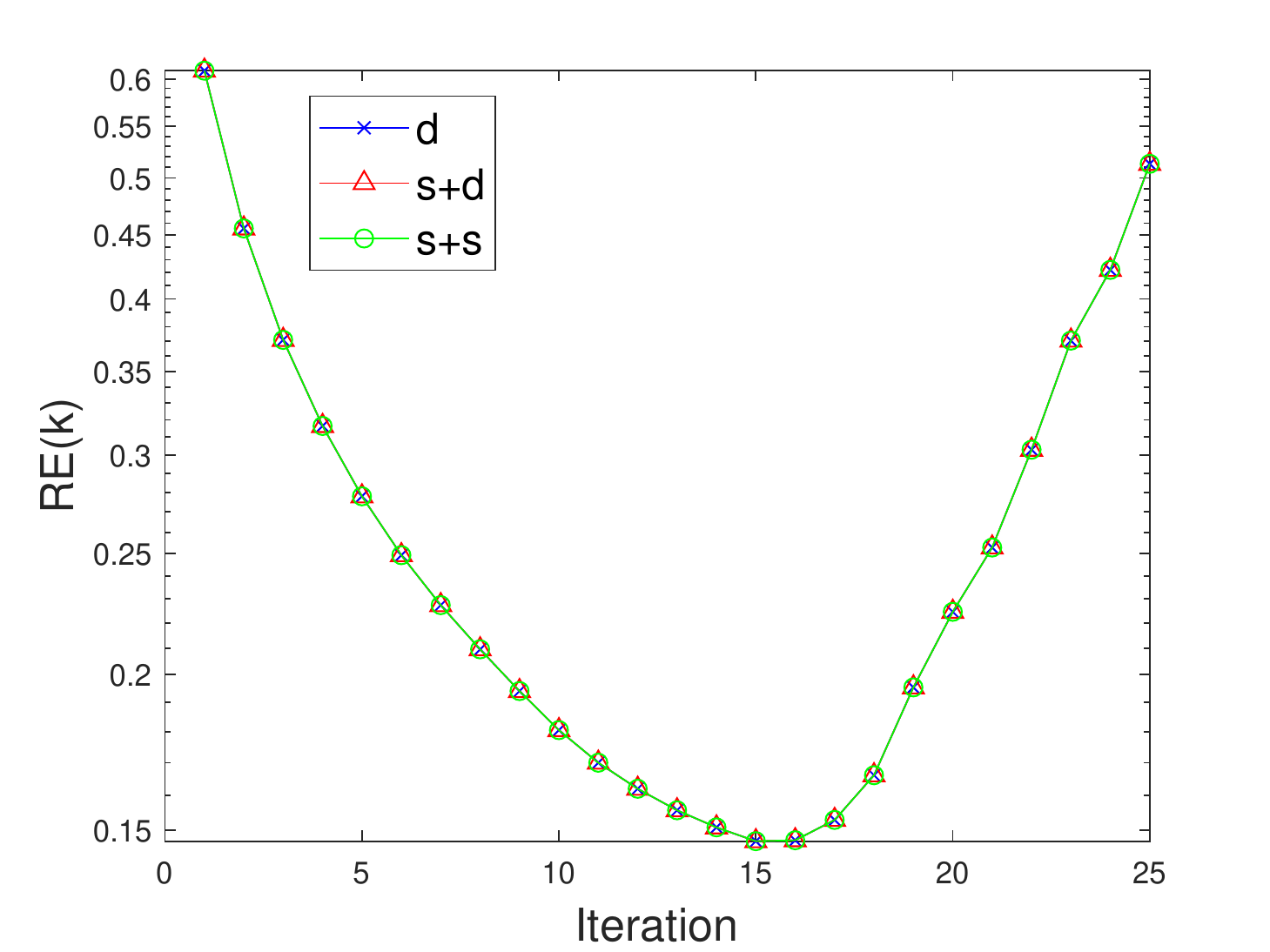}}
		\centerline{(b) {\sf deriv2}}
	\end{minipage}
	\vfill
	\begin{minipage}{0.48\linewidth}
		\centerline{\includegraphics[width=5cm,height=3.5cm]{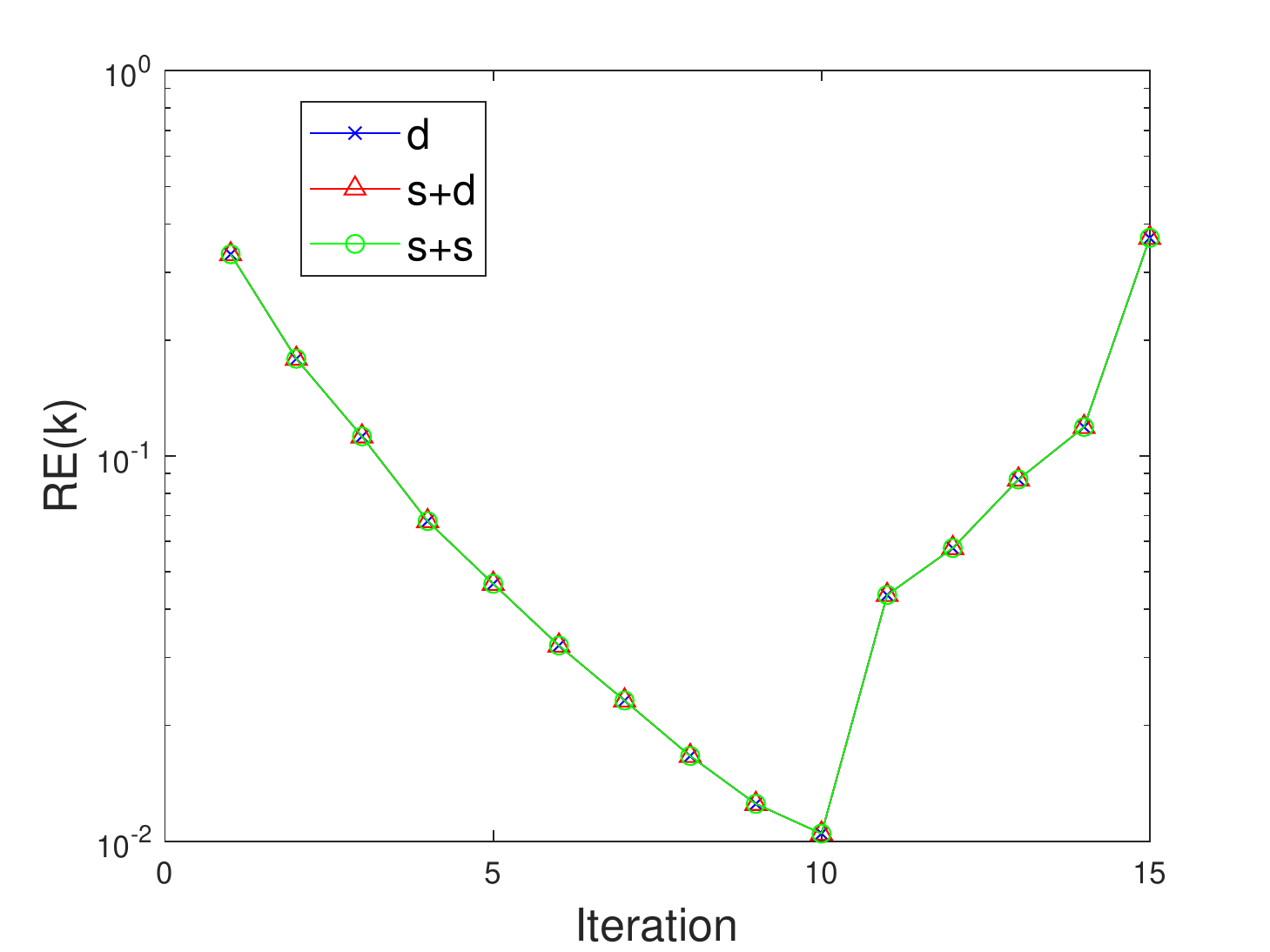}}
		\centerline{(c) {\sf gravity}}
	\end{minipage}
	\hfill
	\begin{minipage}{0.48\linewidth}
		\centerline{\includegraphics[width=5cm,height=3.5cm]{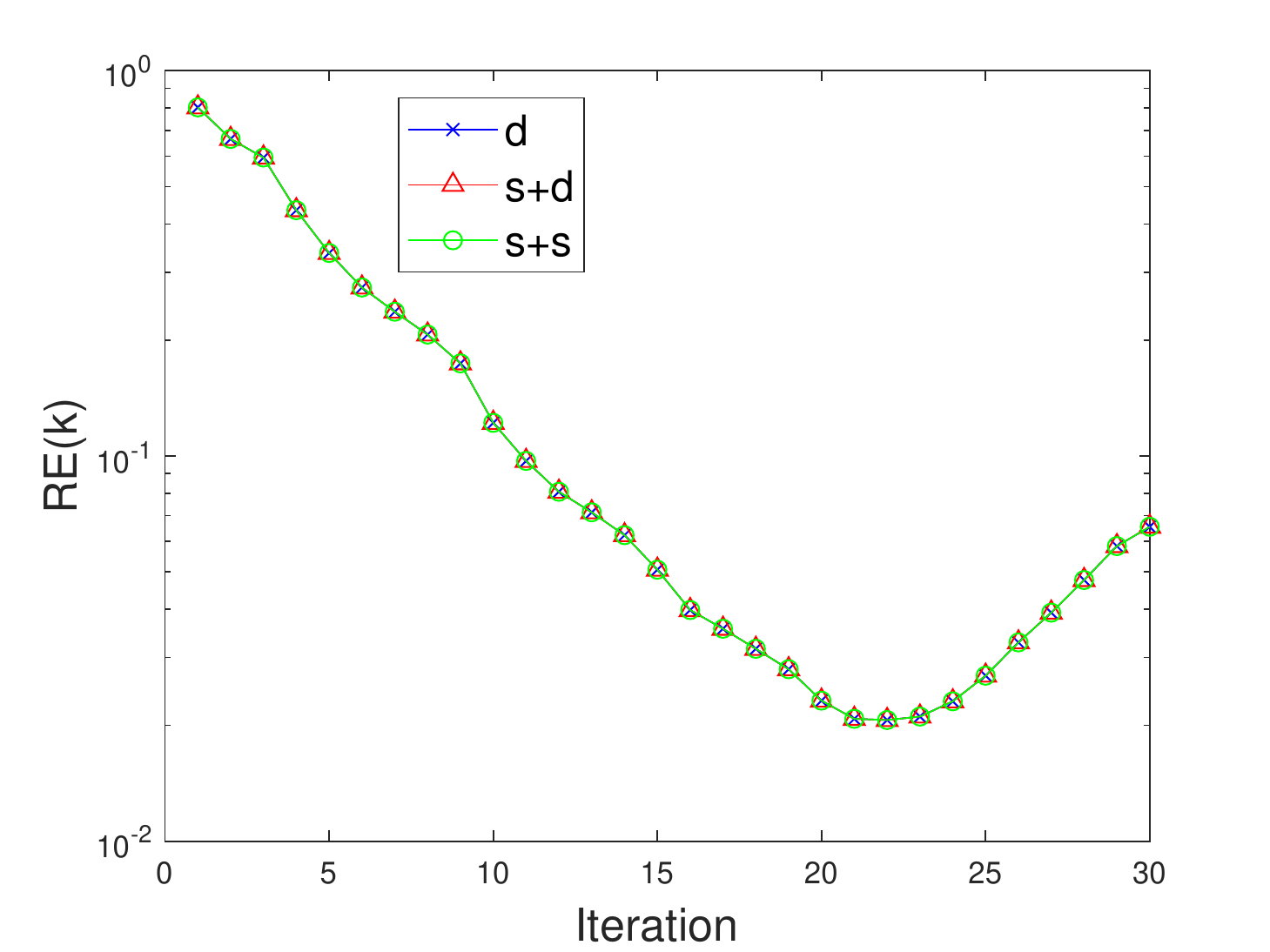}}
		\centerline{(d) {\sf heat}}
	\end{minipage}
	\caption{Semi-convergence curves for LSQR implemented using different computing precisions, $\varepsilon=10^{-3}$.}
	\label{fig5.1}
\end{figure}

\begin{figure}[htp]
	\begin{minipage}{0.48\linewidth}
		\centerline{\includegraphics[width=5cm,height=3.5cm]{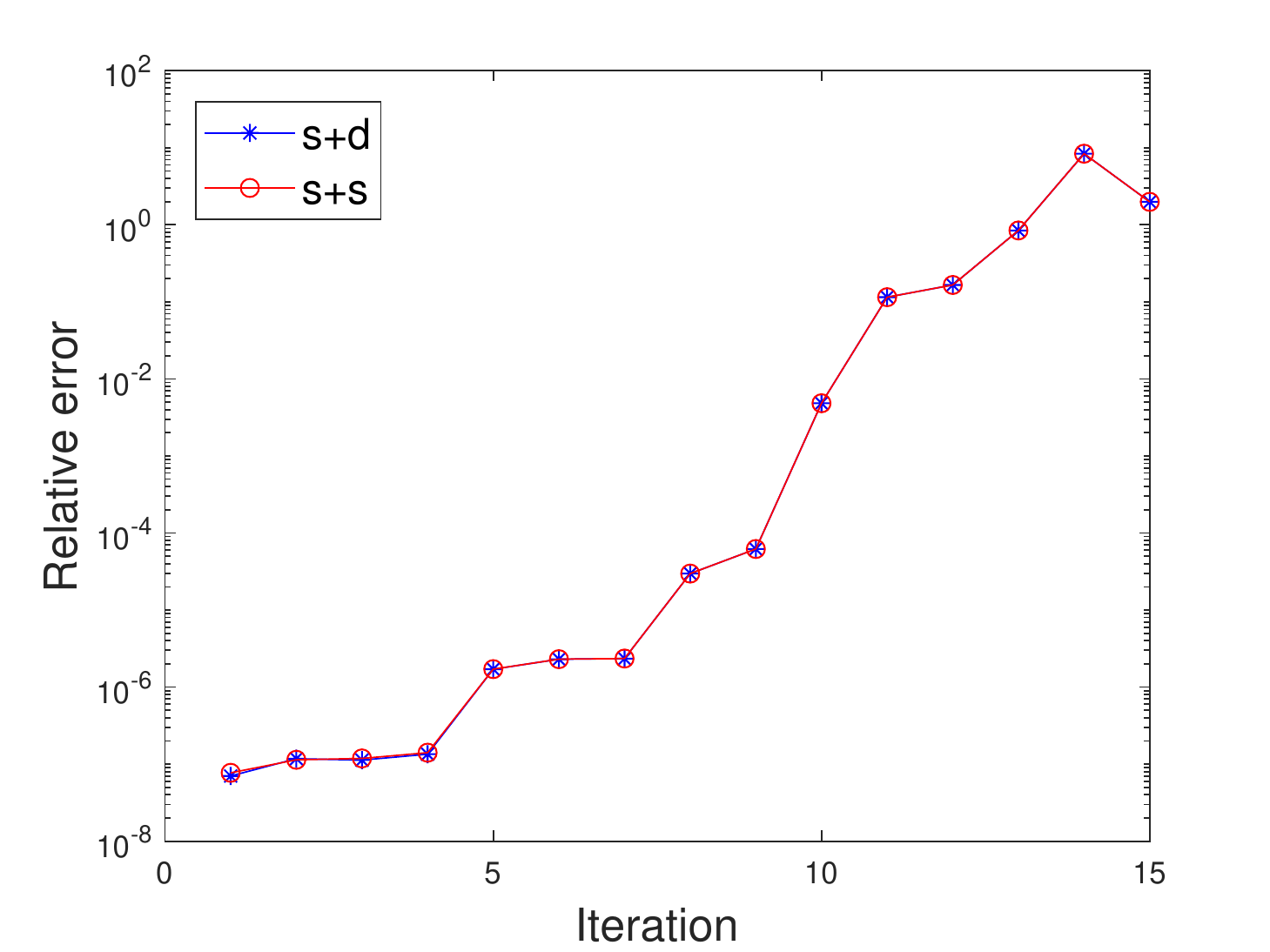}}
		\centerline{(a) {\sf shaw}}
	\end{minipage}
	\hfill
	\begin{minipage}{0.48\linewidth}
		\centerline{\includegraphics[width=5cm,height=3.5cm]{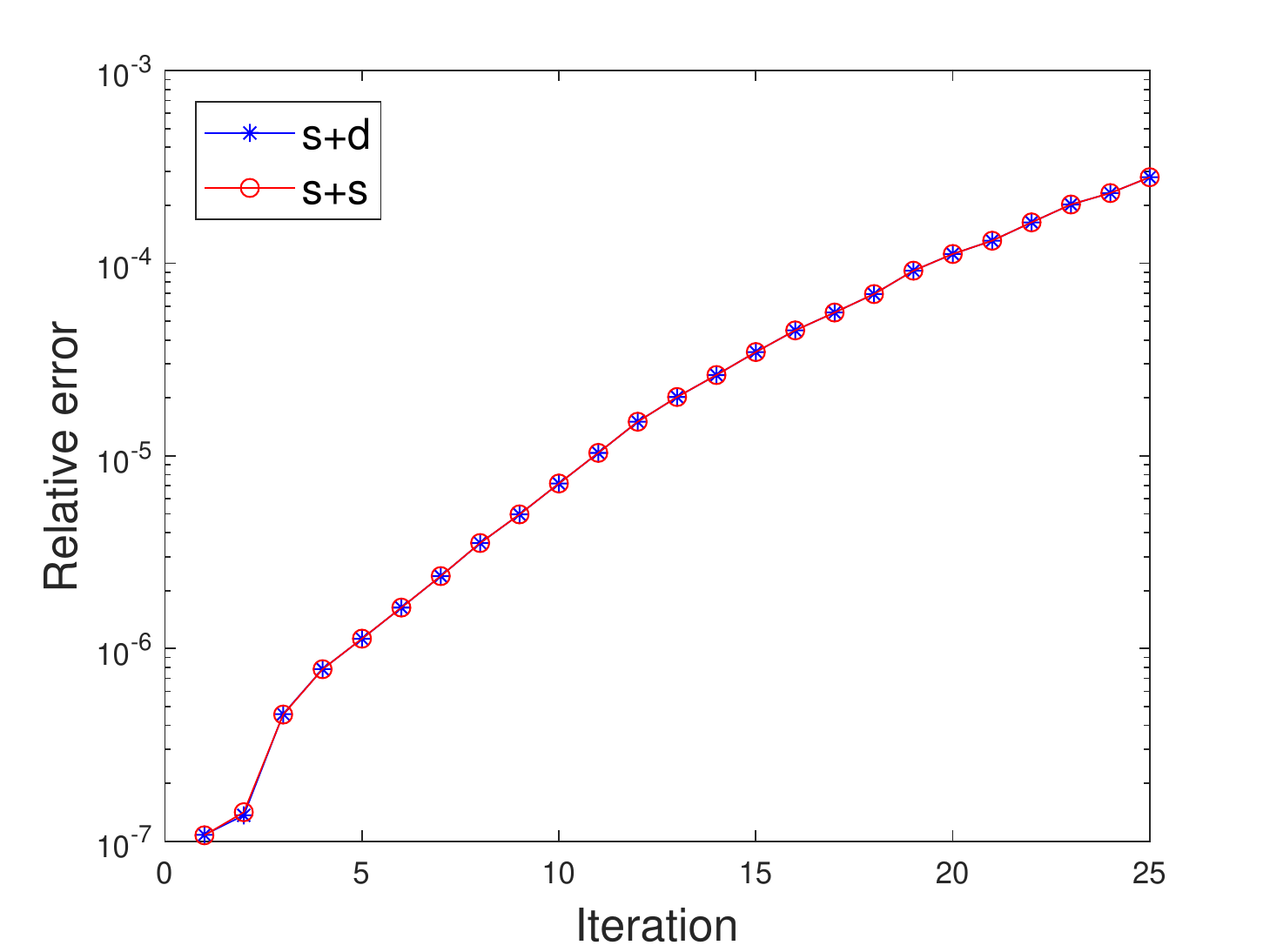}}
		\centerline{(b) {\sf deriv2}}
	\end{minipage}
	\vfill
	\begin{minipage}{0.48\linewidth}
		\centerline{\includegraphics[width=5cm,height=3.5cm]{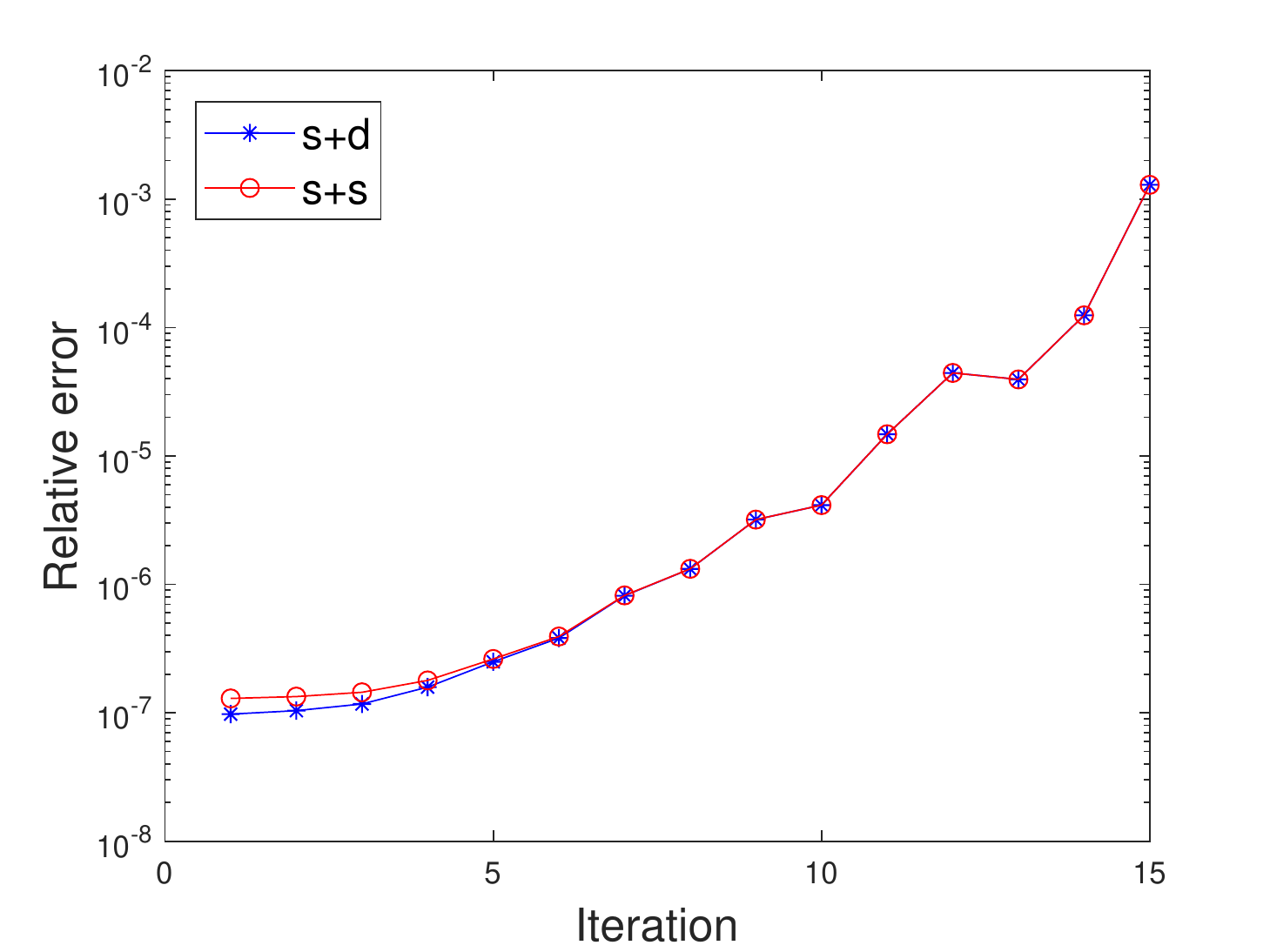}}
		\centerline{(c) {\sf gravity}}
	\end{minipage}
	\hfill
	\begin{minipage}{0.48\linewidth}
		\centerline{\includegraphics[width=5cm,height=3.5cm]{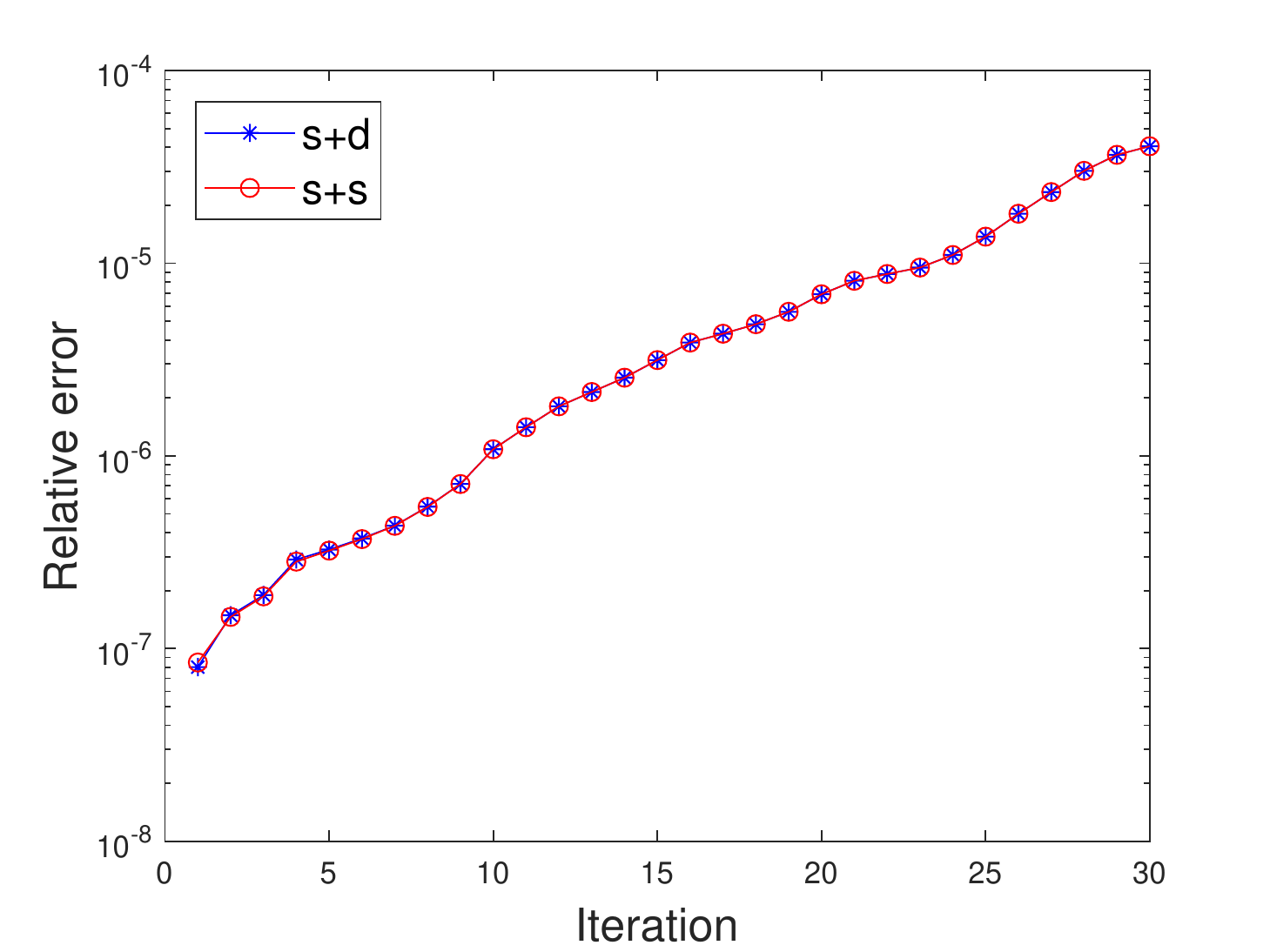}}
		\centerline{(d) {\sf heat}}
	\end{minipage} 
	\caption{Relative errors of the regularized solutions computed by ``s+d"/``s+s" with respect to that by ``d", $\varepsilon=10^{-3}$.}
	\label{fig5.2}
\end{figure}

First, we compare the relative errors $\mathrm{RE}(k)$ for the three different implementations of LSQR when $\varepsilon=10^{-3}$. From Figure \ref{fig5.1} we can find the convergence behaviors of the three implementations ``d", ``s+d" and ``s+s" are of highly consistence. The semi-convergence points $k_{0}$ are the same and the relative error curves coincide until many steps after semi-convergence, and thus the optimal regularized solutions computed by ``d", ``s+d" and ``s+s" have the same accuracy. In order to give a more clear comparison about accuracy of solutions, we also plot the relative error curves of $x_k$/$\hat{x}_k$ computed by ``s+d"/``s+s" with respect to that by ``d". Figure \ref{fig5.2} shows that these two relative errors are much smaller than $\mathrm{RE}(k)$ of ``d" until semi-convergence occurs and this is also true for many iterations afterwards. These results confirm that both the LBFRO and updating procedure can be implemented using single precision without sacrificing any accuracy of final regularized solutions for $\varepsilon=10^{-3}$.

\begin{figure}[htp]
	\begin{minipage}{0.48\linewidth}
		\centerline{\includegraphics[width=5cm,height=3.5cm]{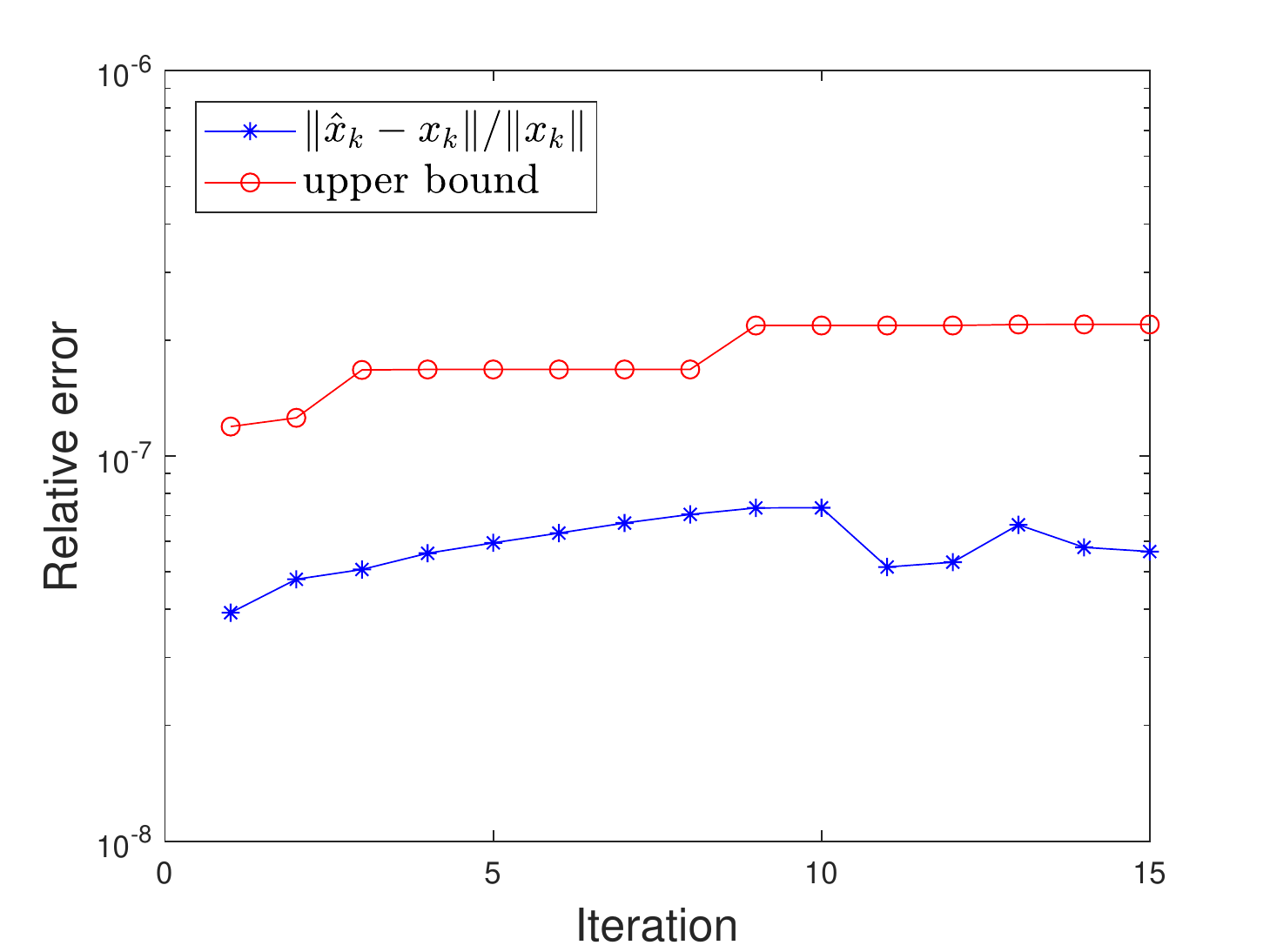}}
		\centerline{(a) {\sf shaw}}
	\end{minipage}
	\hfill
	\begin{minipage}{0.48\linewidth}
		\centerline{\includegraphics[width=5cm,height=3.5cm]{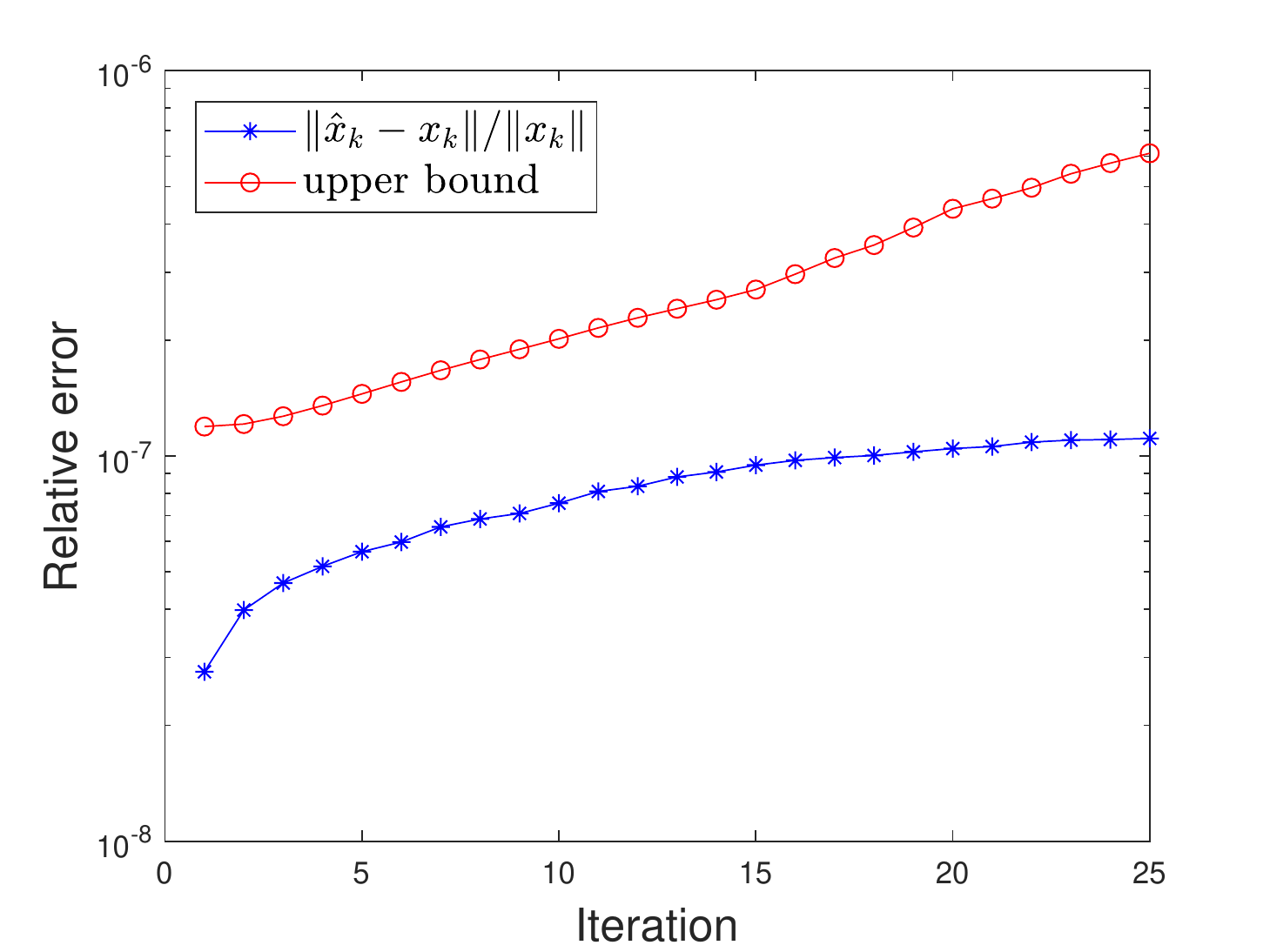}}
		\centerline{(b) {\sf deriv2}}
	\end{minipage}
	\vfill
	\begin{minipage}{0.48\linewidth}
		\centerline{\includegraphics[width=5cm,height=3.5cm]{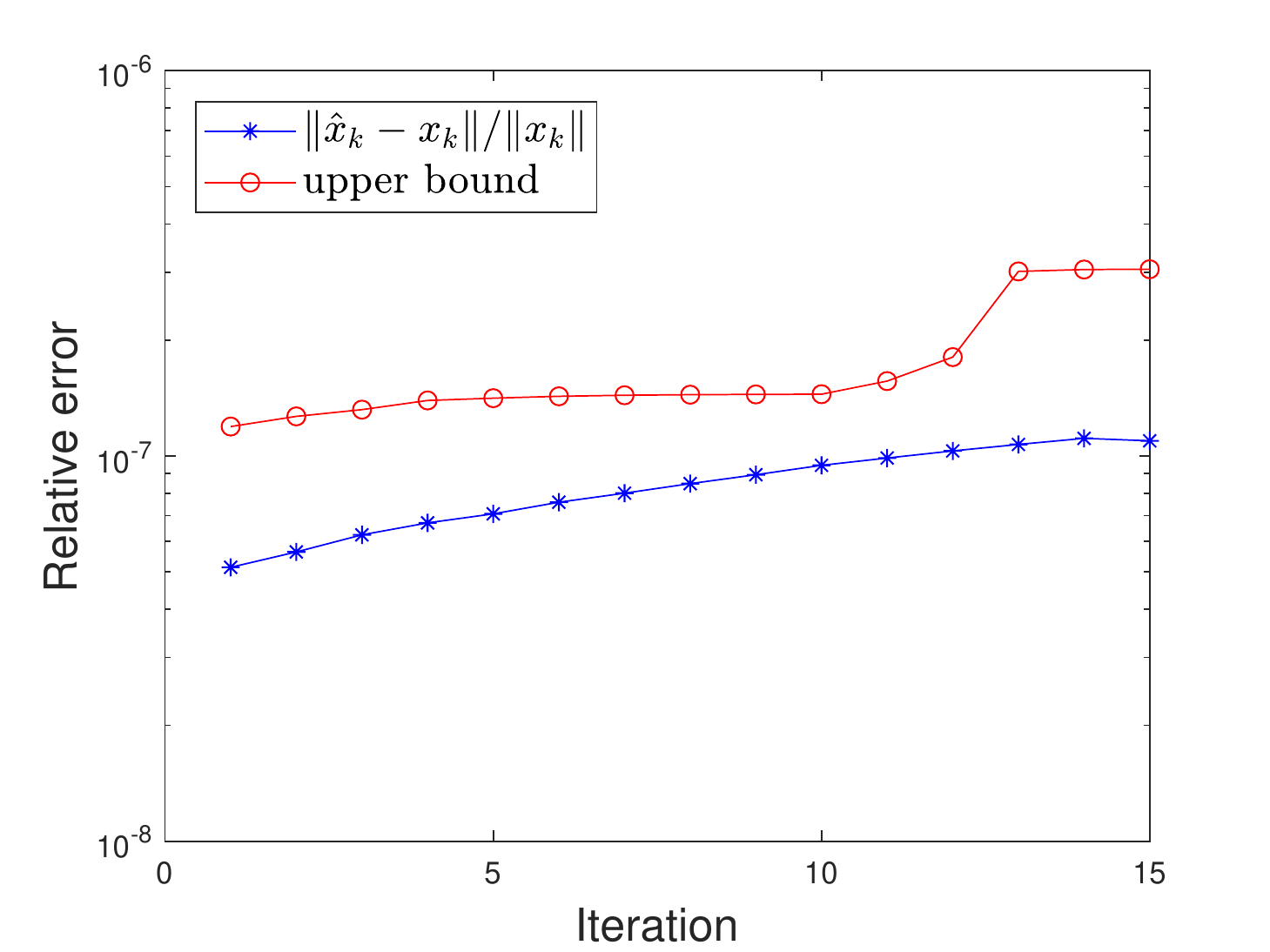}}
		\centerline{(c) {\sf gravity}}
	\end{minipage}
	\hfill
	\begin{minipage}{0.48\linewidth}
		\centerline{\includegraphics[width=5cm,height=3.5cm]{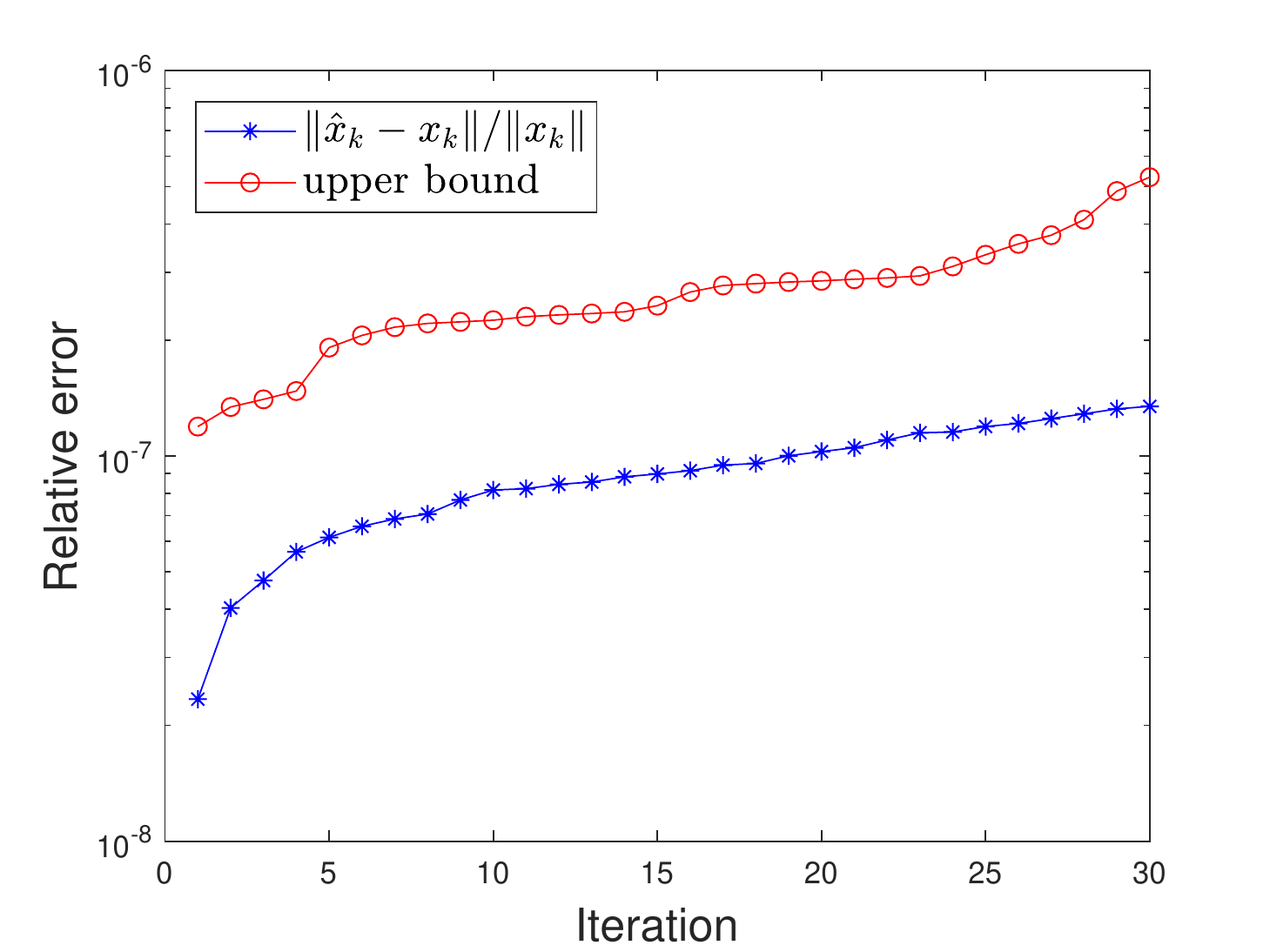}}
		\centerline{(d) {\sf heat}}
	\end{minipage}
	\caption{Relative errors between regularized solutions computed by ``s+d" and ``s+s", $\varepsilon=10^{-3}$.}
	\label{fig5.3}
\end{figure}

To further compare the accuracy of solutions computed by ``s+d" and ``s+s", we plot in Figure \ref{fig5.3} the relative error curves of these two solutions and their upper bounds in \eqref{4.9}. Here we set the upper bounds as $\kappa(\widehat{R}_k)\mathbf{\bar{u}}$ with $\mathbf{\bar{u}}$ the roundoff unit of single precision. From Figure \ref{fig5.3} we can find that $\kappa(\widehat{R}_k)$ for the four test problems grow very slightly, which lead to the upper bounds much smaller than $\mathrm{RE}(k)$. Therefore, the regularized solutions computed by ``s+d" and ``s+s" have the same accuracy, which has already been clearly observed from Figure \ref{fig5.1}.

\begin{table}[htp]
	\centering
	\caption{Comparison of relative errors $\mathrm{RE}(k)$ and estimates of the optimal iteration $k_0$ by L-curve criterion and DP ($\tau=1.001$), $\varepsilon=10^{-3}$.}
		\begin{tabular}{*{5}{c}}
			\toprule
			Work precision 	&{\sf shaw}   &{\sf deriv2}   &{\sf gravity}  &{\sf heat} \\
			\midrule
			&  \multicolumn{3}{c}{Optimal} &    \\
			\midrule
			d  & $0.0396$ ($8$)   & $0.1471$ ($15$)     & $0.0105$ ($10$) & $0.0206$ ($22$)  \\
			s+d &$0.0396$ ($8$) & $0.1471$ ($15$) & $0.0105$ ($10$) & $0.0206$ ($22$)  \\
			s+s &$0.0396$ ($8$) & $0.1471$ ($15$) & $0.0105$ ($10$)  & $0.0206$ ($22$)   \\
			\midrule
			&  \multicolumn{3}{c}{L-curve} &   \\
			\midrule
			d  &$0.0396$ ($8$) & $0.1529$ ($17$) & $0.0105$ ($10$) & $0.0392$ ($27$)  \\
			s+d &$0.0396$ ($8$) & $0.1529$ ($17$)  & $0.0105$ ($10$)  & $0.0392$ ($27$)  \\
			s+s &$0.0396$ ($8$) & $0.1529$ ($17$)  & $0.0105$ ($10$)  & $0.0392$ ($27$)   \\
			\midrule
			&   \multicolumn{3}{c}{Discrepancy principle} &    \\
			\midrule
			d  & $0.0473$ ($7$) & $0.1557$ ($13$) & $0.0166$ ($8$) & $0.0280$ ($19$)  \\
			s+d/s & $0.0473$ ($7$)  & $0.1557$ ($13$) & $0.0166$ ($8$) & $0.0280$ ($19$) \\
			\bottomrule[0.6pt]
	\end{tabular}
	\label{tab5.2}
\end{table}

Table \ref{tab5.2} shows the relative errors of the regularized solutions at the semi-convergence point $k_0$ and the estimates of $k_0$ by L-curve criterion and discrepancy principle, where the corresponding iteration number is in brackets. We find that the three optimal iterations and corresponding $\mathrm{RE}(k)$ for ``d", ``s+d" and ``s+s" are the same, which has also been observed from Figure \ref{fig5.1}. This is also true for the L-curve criterion, and the method gets an over-estimate of $k_0$ for {\sf heat}. For the discrepancy principle, by \eqref{discrepancy} we know that the estimates of $k_0$ for ``s+d" and ``s+s" are always the same since they compute the same $\bar{\phi}_{k+1}$, and we find these estimates are the same at that for ``d". The discrepancy principle gets under-estimates of $k_0$ for the four test problems and thus the solutions are over-regularized.

\begin{figure}[htp]
	\begin{minipage}{0.48\linewidth}
		\centerline{\includegraphics[width=5cm,height=3.5cm]{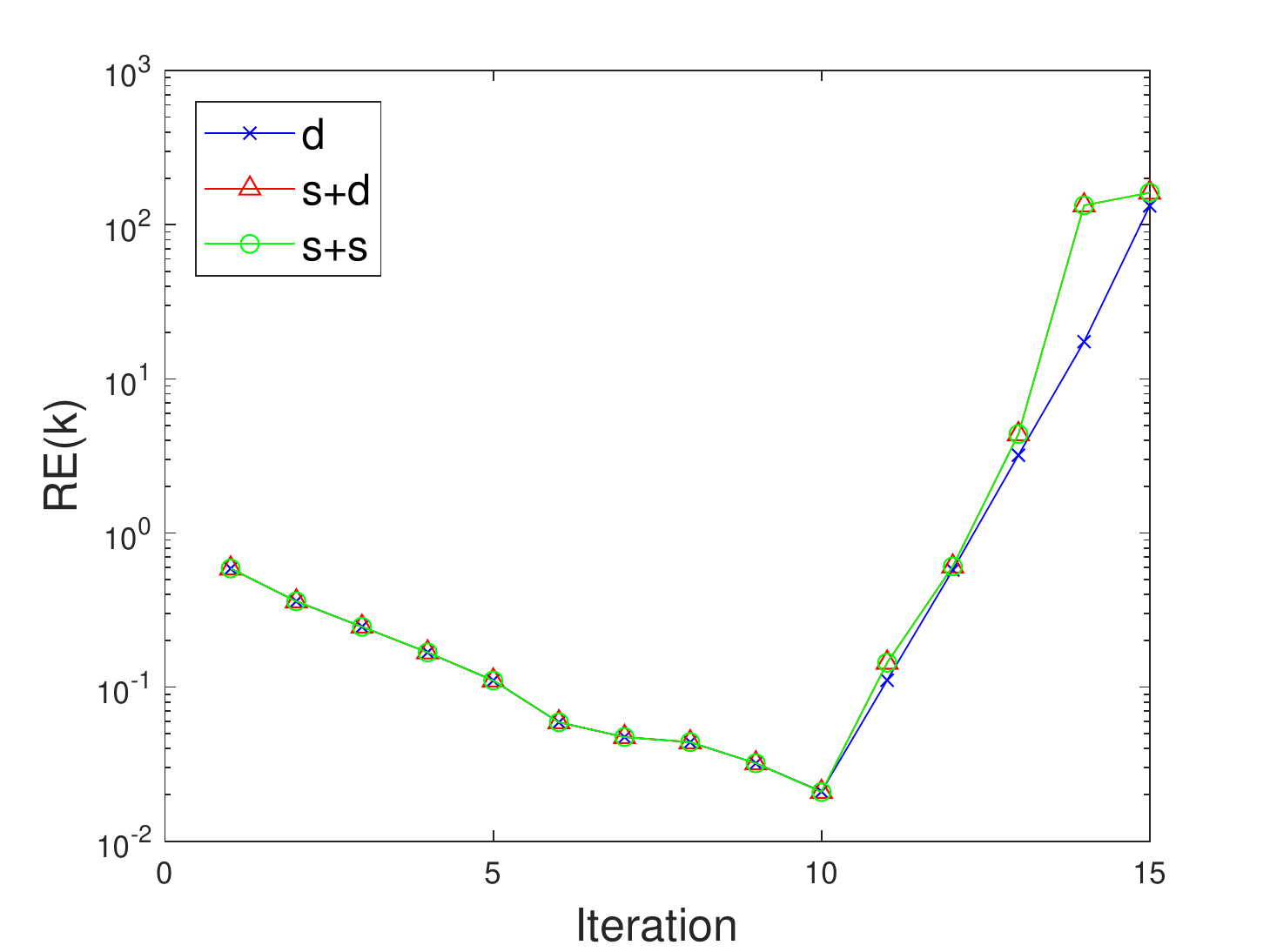}}
		\centerline{(a) {\sf shaw}}
	\end{minipage}\label{fig5.4a}
	\hfill
	\begin{minipage}{0.48\linewidth}
		\centerline{\includegraphics[width=5cm,height=3.5cm]{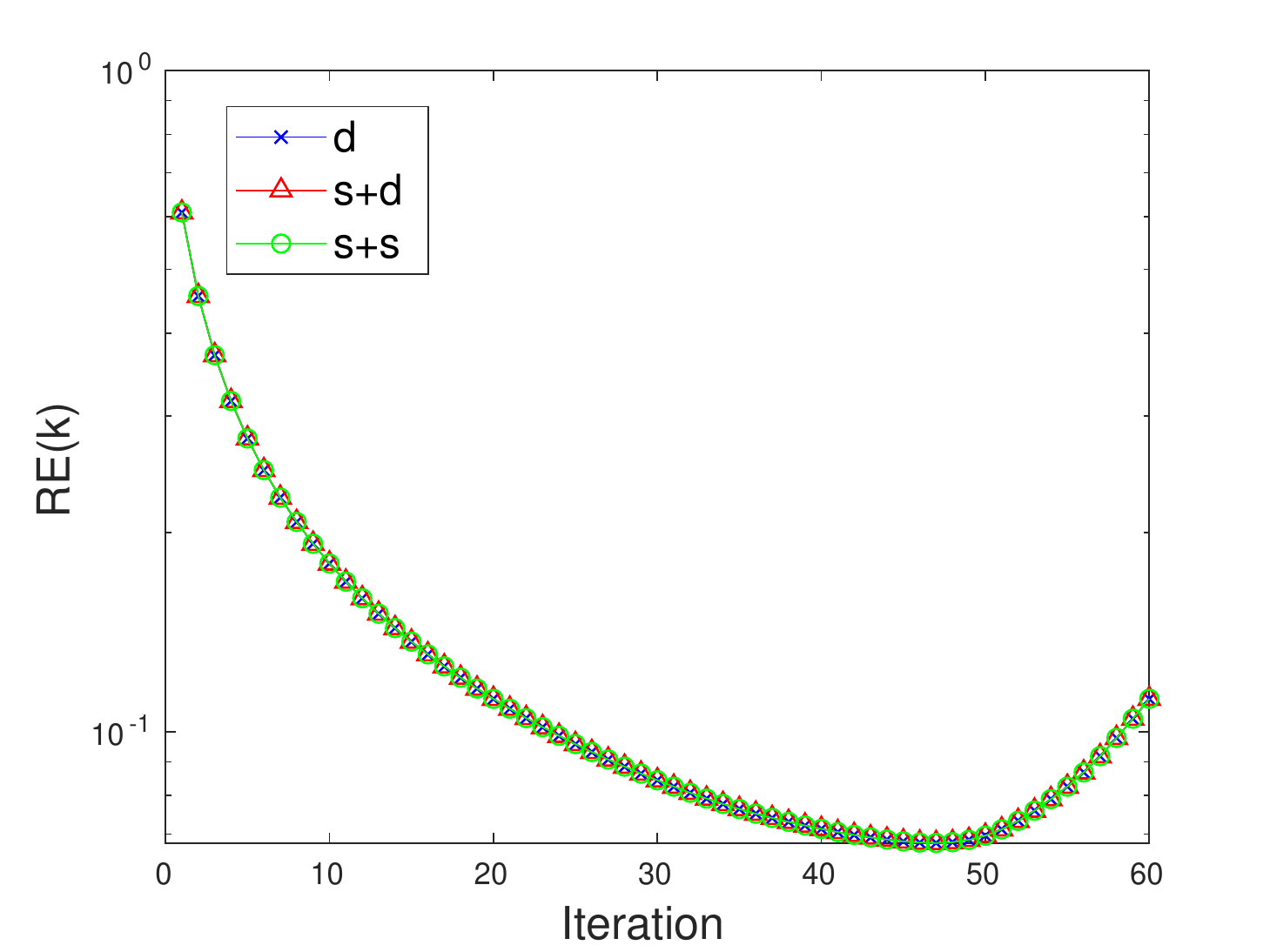}}
		\centerline{(b) {\sf deriv2}}
	\end{minipage}
	\vfill
	\begin{minipage}{0.48\linewidth}
		\centerline{\includegraphics[width=5cm,height=3.5cm]{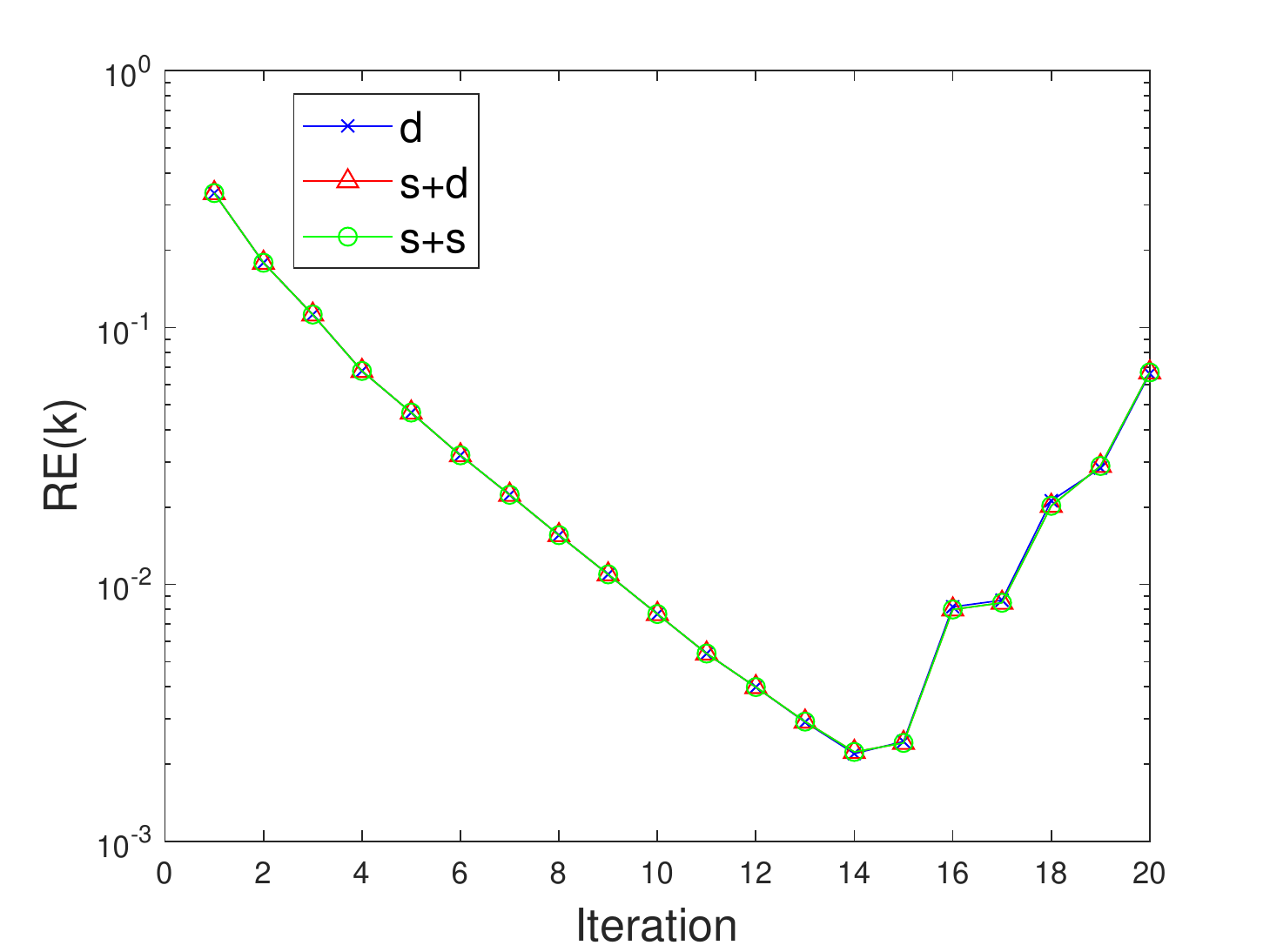}}
		\centerline{(c) {\sf gravity}}
	\end{minipage}
	\hfill
	\begin{minipage}{0.48\linewidth}
		\centerline{\includegraphics[width=5cm,height=3.5cm]{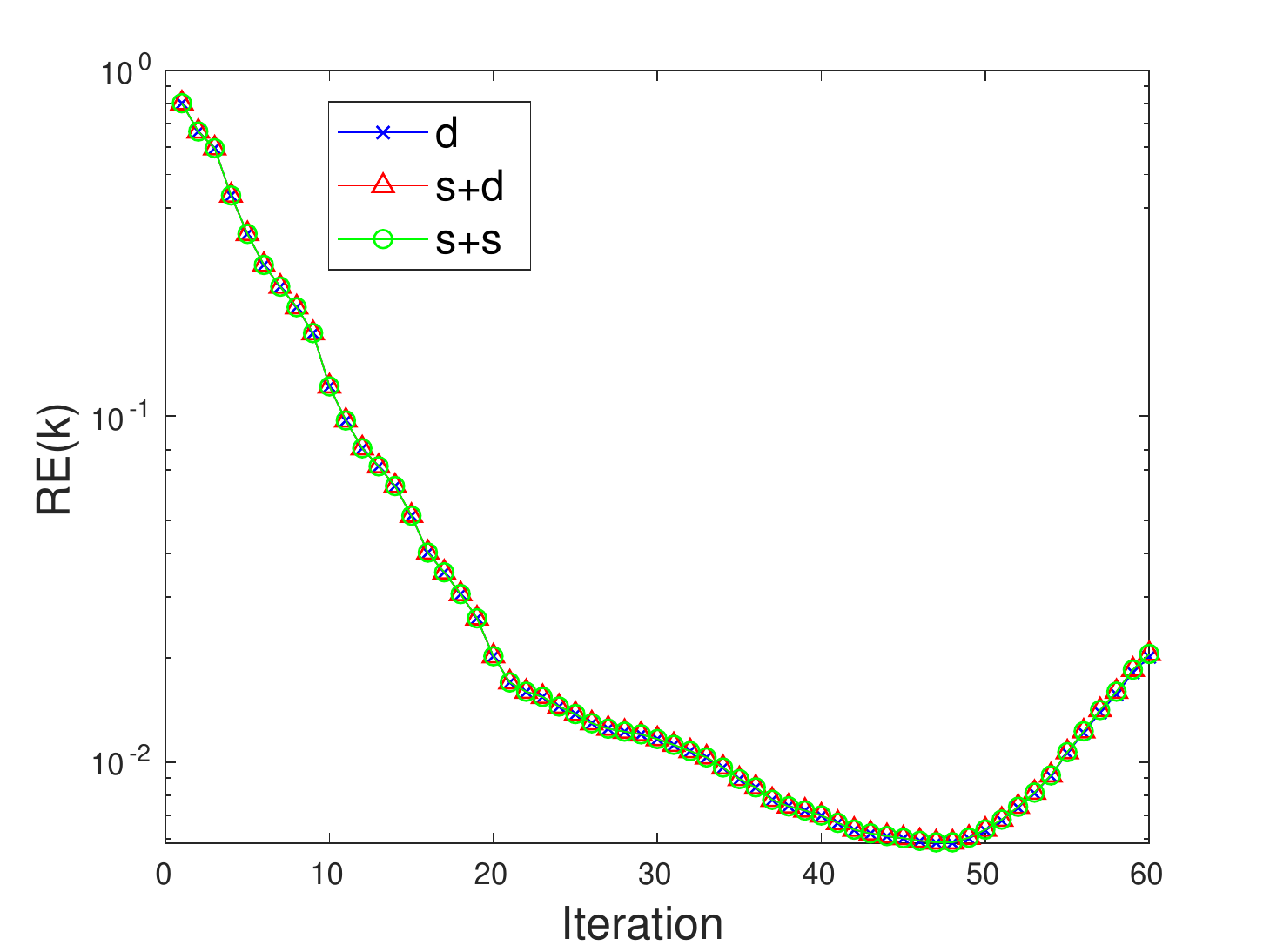}}
		\centerline{(d) {\sf heat}}
	\end{minipage}
	\caption{Semi-convergence curves for LSQR implemented using different computing precisions, $\varepsilon=10^{-5}$.}
	\label{fig5.4}
\end{figure}

\begin{figure}[htp]
	\begin{minipage}{0.48\linewidth}
		\centerline{\includegraphics[width=5cm,height=3.5cm]{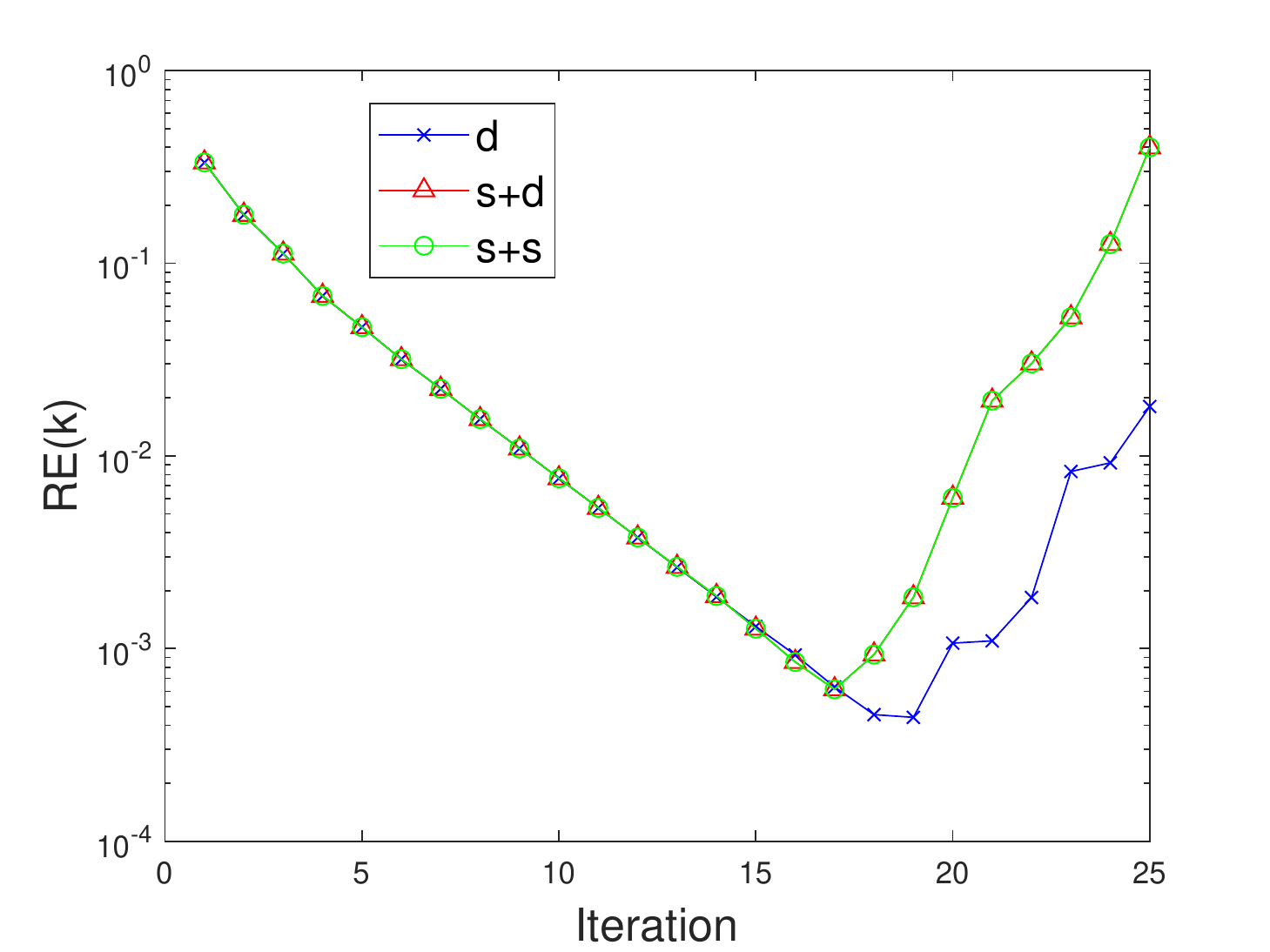}}
		\centerline{(a) {\sf gravity}}
	\end{minipage}
	\hfill
	\begin{minipage}{0.48\linewidth}
		\centerline{\includegraphics[width=5cm,height=3.5cm]{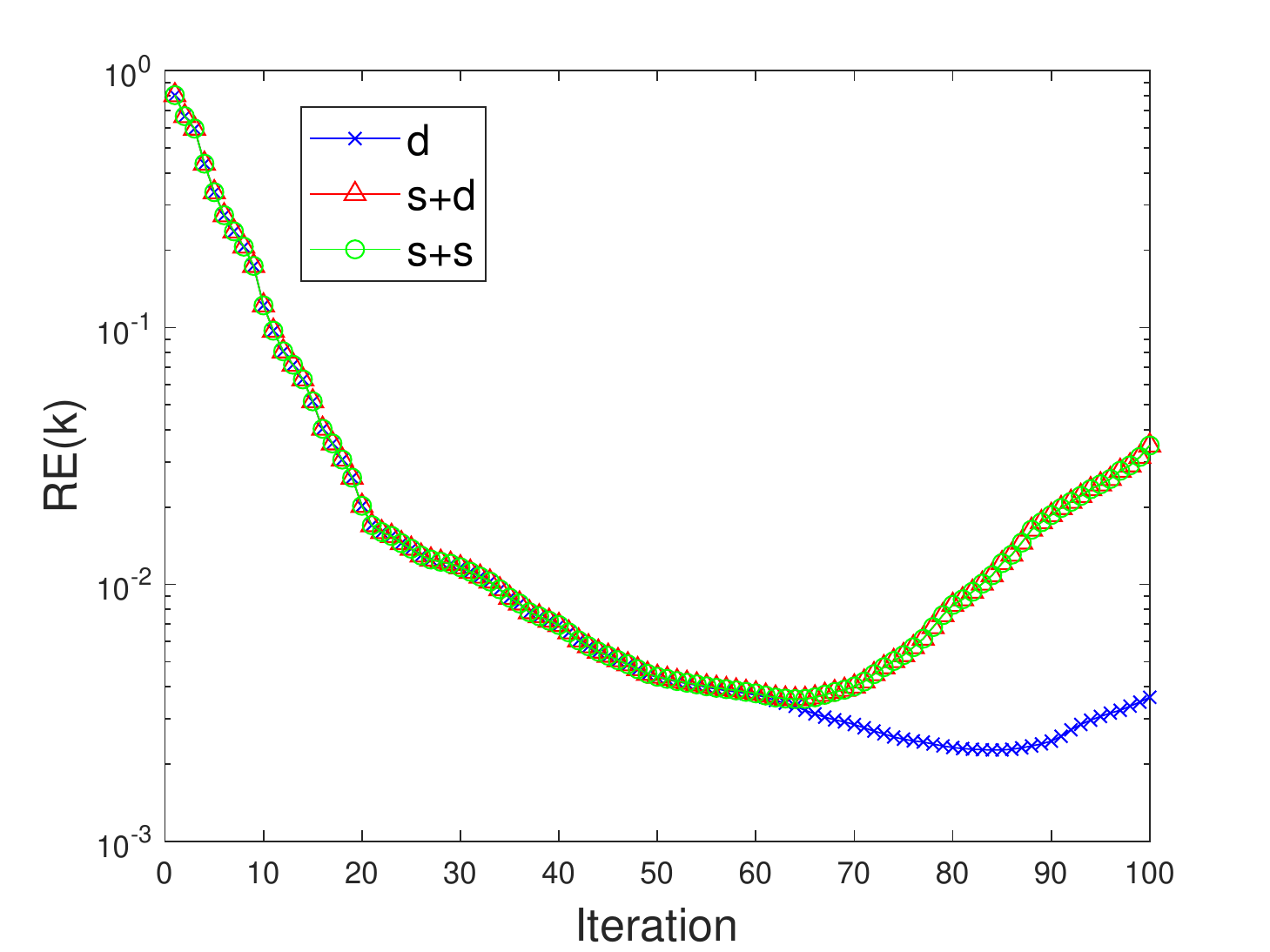}}
		\centerline{(b) {\sf heat}}
	\end{minipage}
	\caption{Semi-convergence curves for LSQR implemented using different computing precisions, $\varepsilon=10^{-7}$.}
	\label{fig5.5}
\end{figure}

From the above experimental results, we can make sure that the two mixed precision variants ``s+d" and ``s+s" can compute regularized solutions with the same accuracy as the double precision LSQR when $\varepsilon=10^{-3}$. We have also made numerical experiments for $\varepsilon=10^{-4}, 10^{-5}$ and get similar results. Here we only show the semi-convergence curves for $\varepsilon=10^{-5}$ in Figure \ref{fig5.4}. For this noise level, single precision computing of the LBFRO and updating $x_k$ is also enough for LSQR for solving the four test problems. Note from subfigure (a) that the relative errors for ``s+d" and ``s+s" quickly become bigger than that for ``d" after semi-convergence. This reminds us that for {\sf shaw}, if $\varepsilon$ is smaller than $10^{-5}$, simply implementing the LBFRO with single precision will lead to a loss of accuracy of regularized solutions. For extremely small noise level $\varepsilon=10^{-7}$, we plot the semi-convergence curves for {\sf gravity} and {\sf heat} in Figure \ref{fig5.5}. We can clearly find that neither ``s+d" nor ``s+s" can compute a regularized solution with accuracy as good as the best one achieved by ``d". In real applications, the noise levels are seldom so small, and it is almost always possible to implement LSQR using lower precisions to compute a regularized solution without sacrificing accuracy.

\subsection{Two dimensional case}\label{subsec5.2}
We generate two image deblurring problems using codes from \cite{Gazzola2019} for testing two dimensional linear ill-posed problems, with the goal to restore an image from a blurred and noisy one $b=Ax_{ex}+e$, where $x_{ex}$ denotes the true image and $A$ denotes the blurring operator. For the background of image deblurring, we refer the readers to \cite{Hansen2006}. For {\sf PRblurspeckle}, which simulates spatially invariant blurring caused by atmospheric turbulence, we use the true image ``Hubble Space Telescope" with image size of $N = 128$ (i.e., the true and blurred images have $128\times 128$ pixels), and the blur level is set to be medium.  For {\sf PRblurdefocus}, which simulates a spatially invariant, out-of-focus blur, we use the true image ``Cameraman" with image size of $N = 256$, and the blur level is set to be severe. Zero boundary condition is used for both the two blurs to construct $A$. The two true images are shown in Figure \ref{fig5.6}.

\begin{figure}[htp]
	\begin{minipage}{0.48\linewidth}
		\centerline{\includegraphics[width=4.0cm,height=3.0cm]{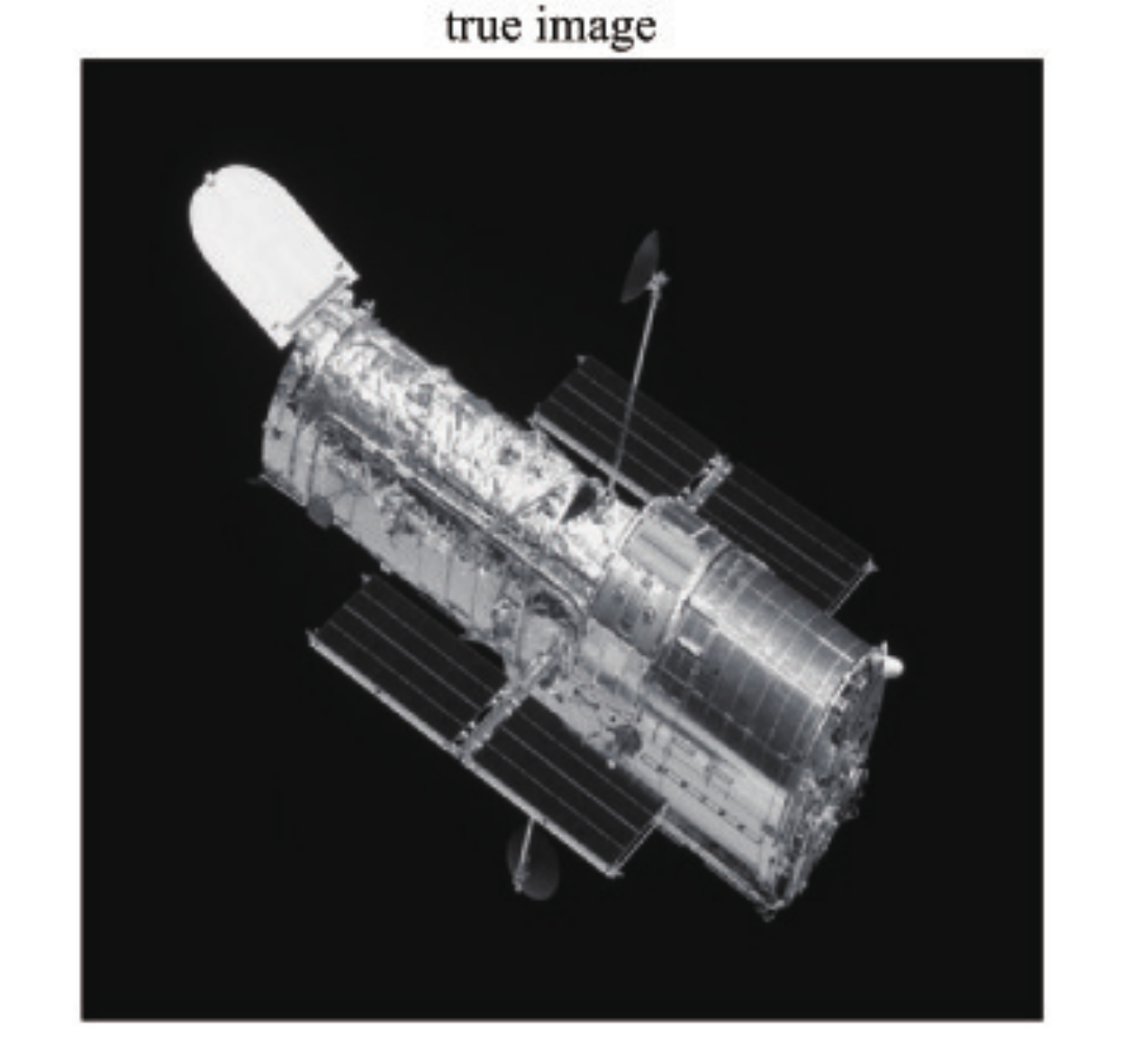}}
		\centerline{(a) Hubble Space Telescope}
	\end{minipage}
	\hfill
	\begin{minipage}{0.48\linewidth}
		\centerline{\includegraphics[width=4.0cm,height=3.0cm]{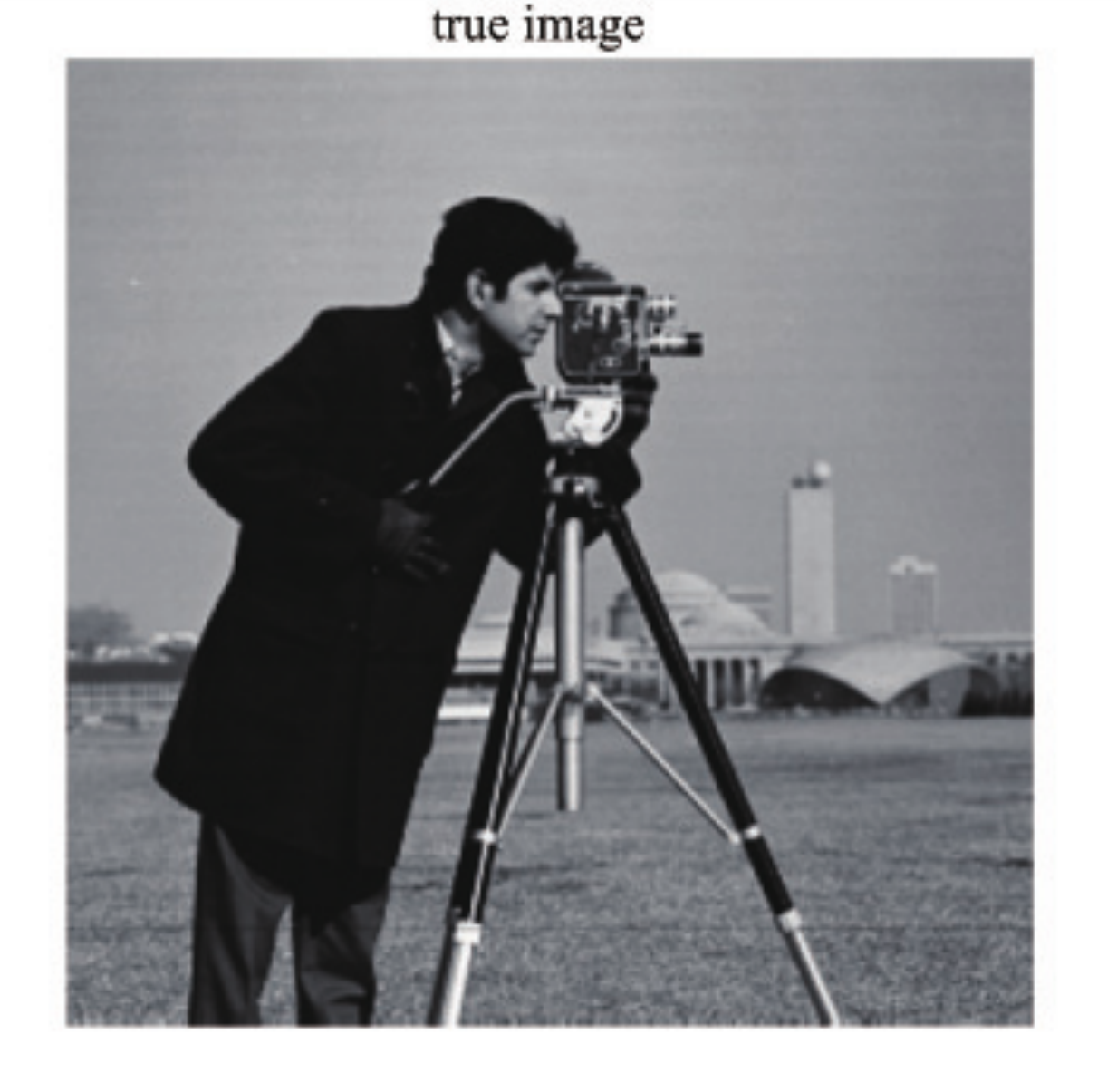}}
		\centerline{(b) Cameraman}
	\end{minipage}
	\caption{True images for testing {\sf PRblurspeckle} and {\sf PRblurdefocus}.}
	\label{fig5.6}
\end{figure}

For {\sf PRblurspeckle} with image ``Hubble Space Telescope", the noise levels are set as $\varepsilon=10^{-2}$ and $\varepsilon=10^{-3}$. Figure \ref{fig5.7} depicts the semi-convergence curves for ``d", ``s+d" and ``s+s" as well as relative error curves of $x_k$ computed by ``s+d" and ``s+s" with respect to that by ``d". For both the two noise levels, we find that the curves of $\mathrm{RE}(k)$ coincide until many steps after semi-convergence and the three semi-convergence points are the same. The relative errors of $x_k$/$\hat{x}_k$ computed by ``s+d"/``s+s" with respect to that by ``d" are shown in subfigures (c) and (d), and they are much smaller than $\mathrm{RE}(k)$ of ``d" until many steps after semi-convergence. The numerical results confirm that for {\sf PRblurspeckle}, the LSQR can be implemented using lower precisions without sacrificing accuracy of regularized solutions for $\varepsilon=10^{-2}$ or $\varepsilon=10^{-3}$. Figure \ref{fig5.8} shows the blurred images and corresponding restored ones for the two noise levels, where the restored images are obtained from the best regularized solution at the semi-convergence point ($k_0$ and $\mathrm{RE}(k_0)$ are the same for ``d", ``s+d" and ``s+s", and we choose $\hat{x}_{k_0}$ computed by ``s+s"). The result shows a good deblurring effect of LSQR implemented using single precision. For noise levels smaller than $10^{-4}$, we find that the semi-convergence behavior does not appear, which means that noise amplification is tolerable even without regularization, and thus we need not test for smaller noise cases. 

\begin{figure}[htp]
	\begin{minipage}{0.45\linewidth}
		\centerline{\includegraphics[width=5cm,height=3.5cm]{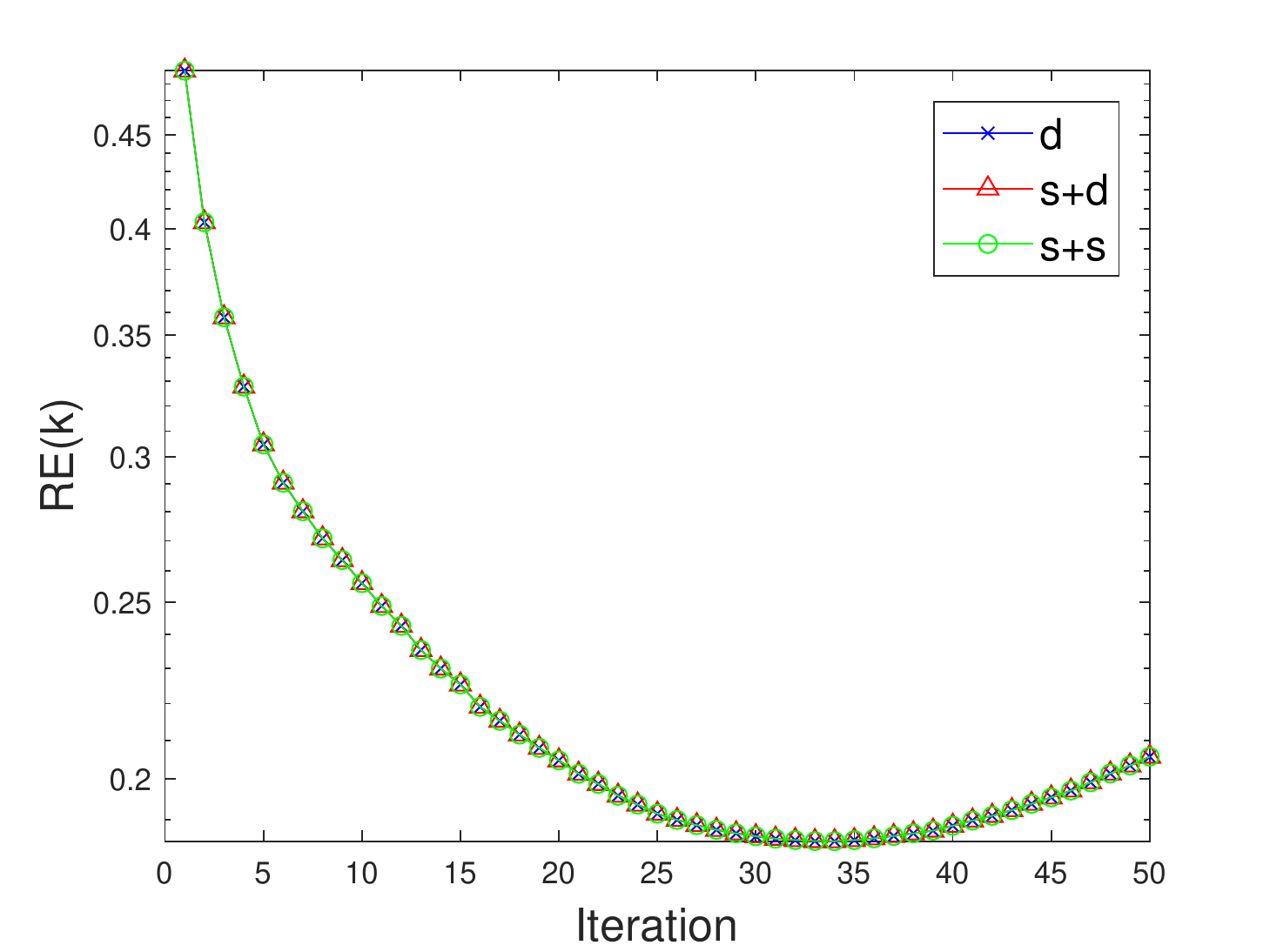}}
		\centerline{(a) $\varepsilon=10^{-2}$}
	\end{minipage}
	\hfill
	\begin{minipage}{0.45\linewidth}
		\centerline{\includegraphics[width=5cm,height=3.5cm]{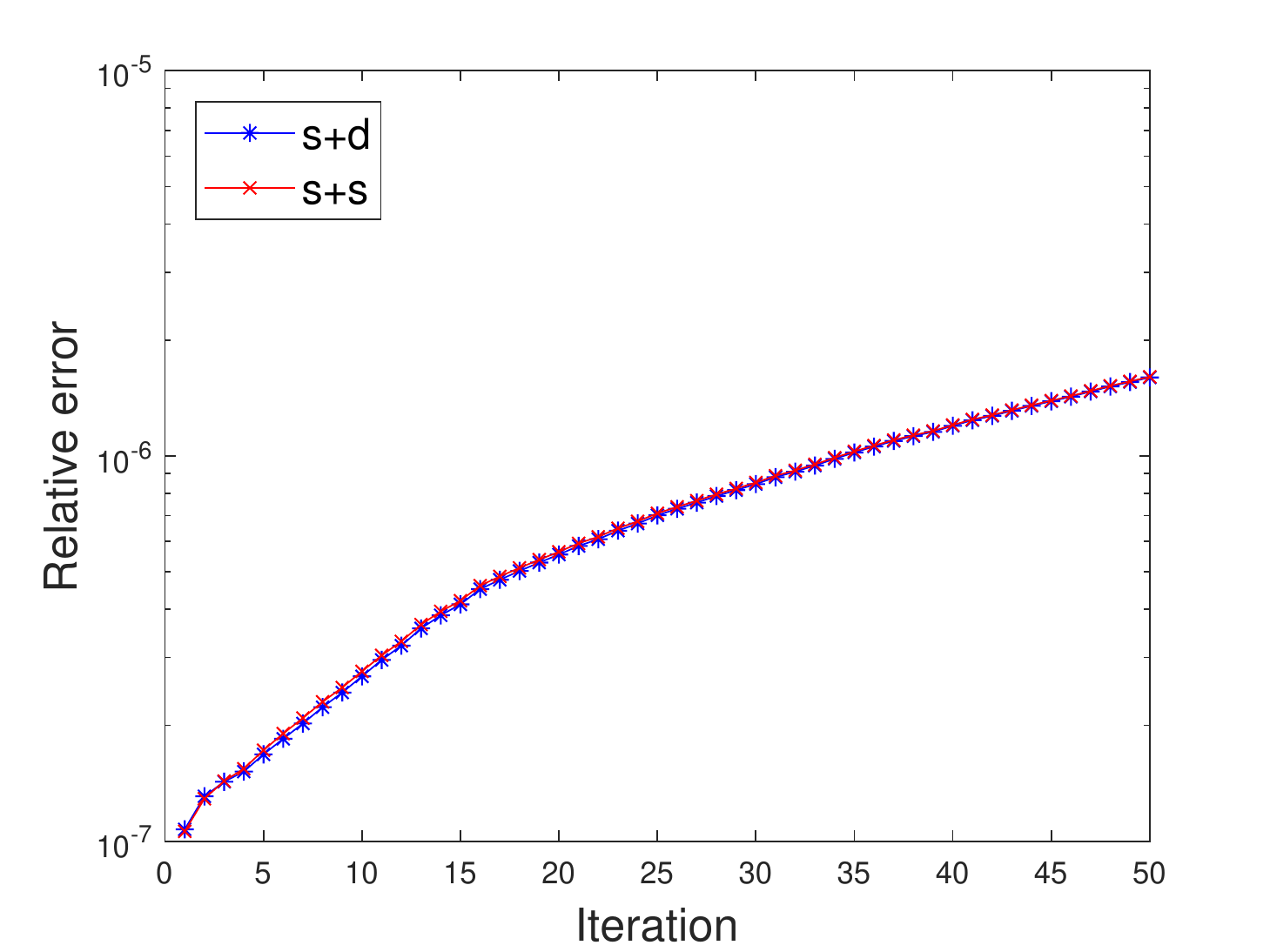}}
		\centerline{(b) $\varepsilon=10^{-2}$}
	\end{minipage}
	\vfill
	\begin{minipage}{0.45\linewidth}
		\centerline{\includegraphics[width=5cm,height=3.5cm]{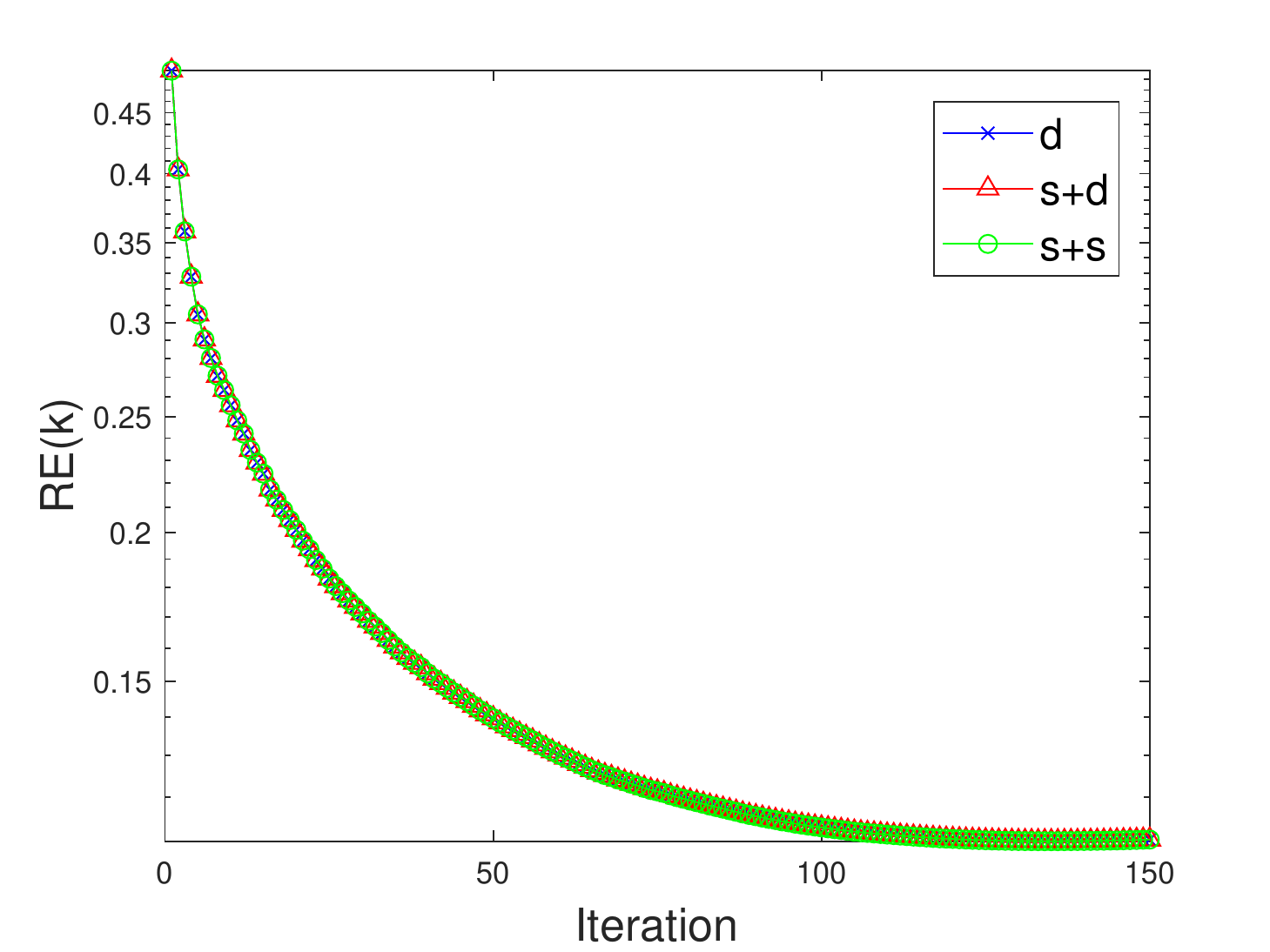}}
		\centerline{(c) $\varepsilon=10^{-3}$}
	\end{minipage}
	\hfill
	\begin{minipage}{0.45\linewidth}
		\centerline{\includegraphics[width=5cm,height=3.5cm]{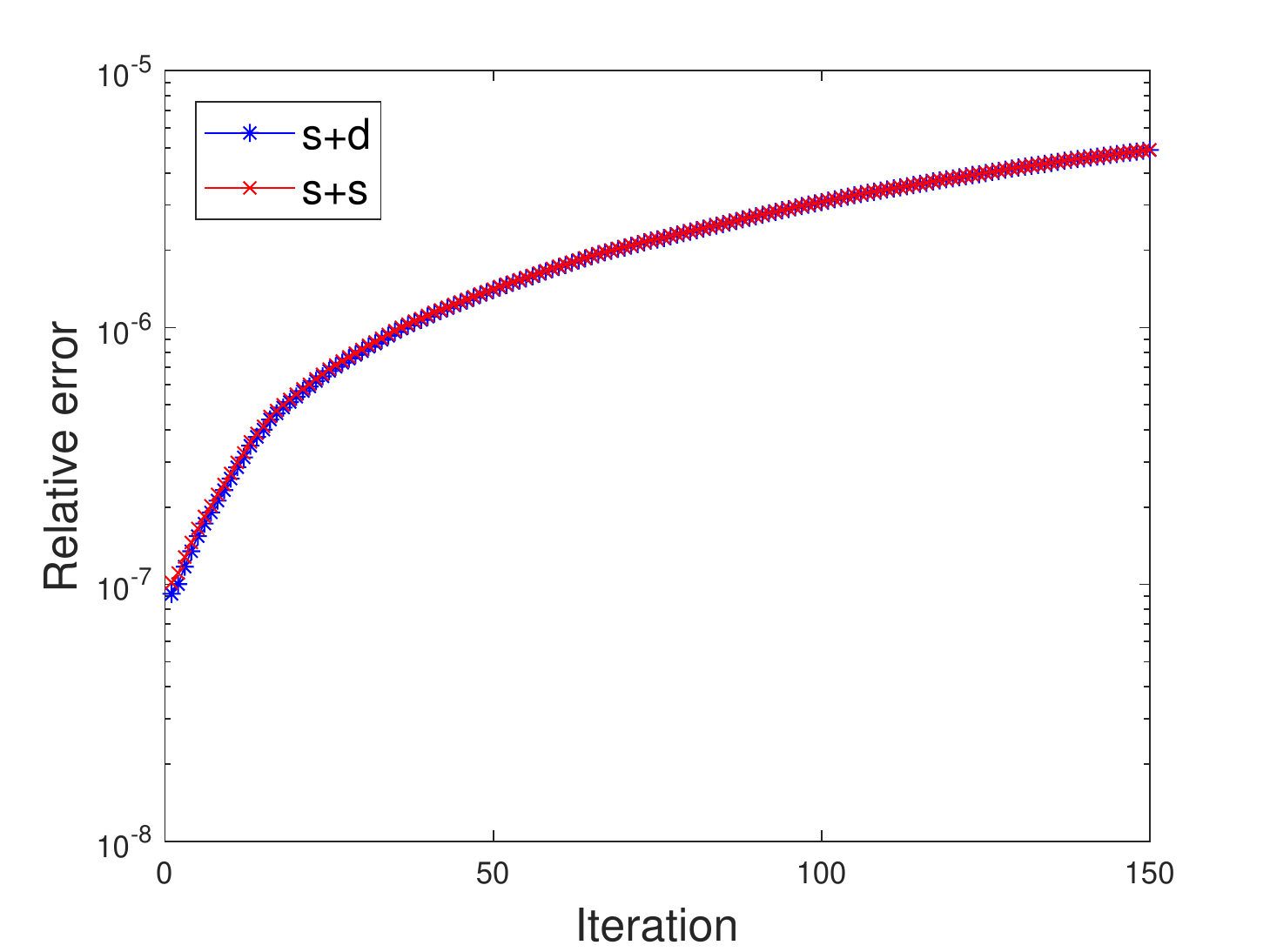}}
		\centerline{(d) $\varepsilon=10^{-3}$}
	\end{minipage}
	\caption{Semi-convergence curves for LSQR implemented using different computing precisions, and relative errors of regularized solutions computed by ``s+d"/``s+s" with respect to that by ``d", {\sf PRblurspeckle}.}
	\label{fig5.7}
\end{figure}

\begin{figure}[htp]
	\begin{minipage}{0.48\linewidth}
		\centerline{\includegraphics[width=4cm,height=3.0cm]{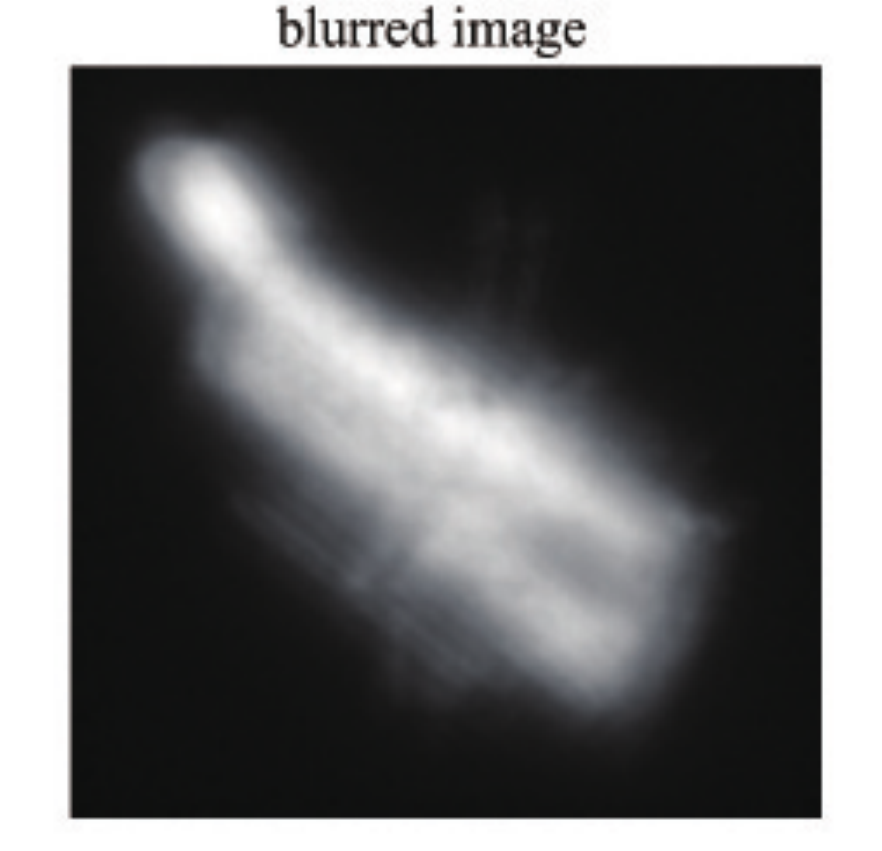}}
		\centerline{(a) $\varepsilon=10^{-2}$}
	\end{minipage}
	\hfill
	\begin{minipage}{0.48\linewidth}
		\centerline{\includegraphics[width=4cm,height=3.0cm]{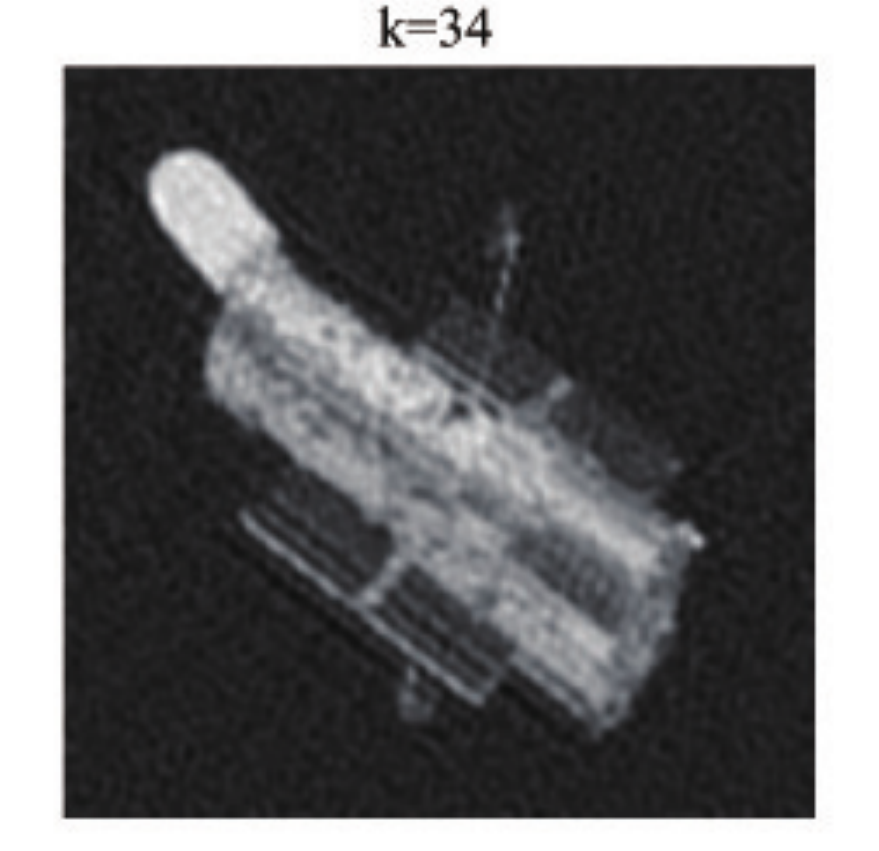}}
		\centerline{(b) $\varepsilon=10^{-2}$}
	\end{minipage}
	\vfill
	\begin{minipage}{0.48\linewidth}
		\centerline{\includegraphics[width=4cm,height=3.0cm]{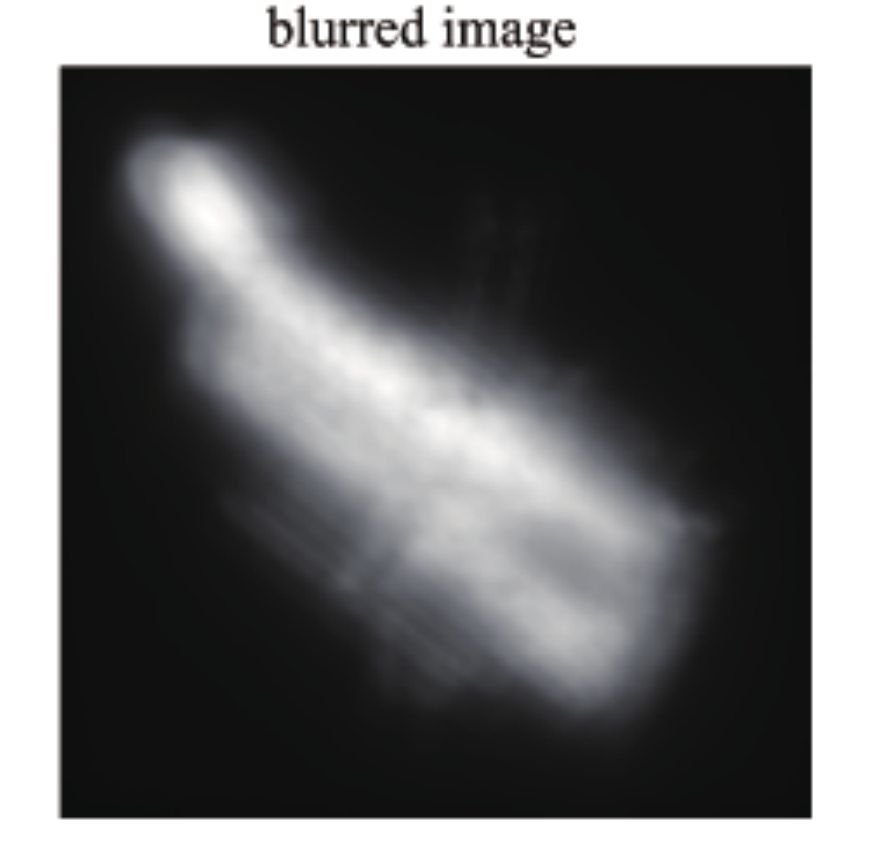}}
		\centerline{(c) $\varepsilon=10^{-3}$}
	\end{minipage}
	\hfill
	\begin{minipage}{0.48\linewidth}
		\centerline{\includegraphics[width=4cm,height=3.0cm]{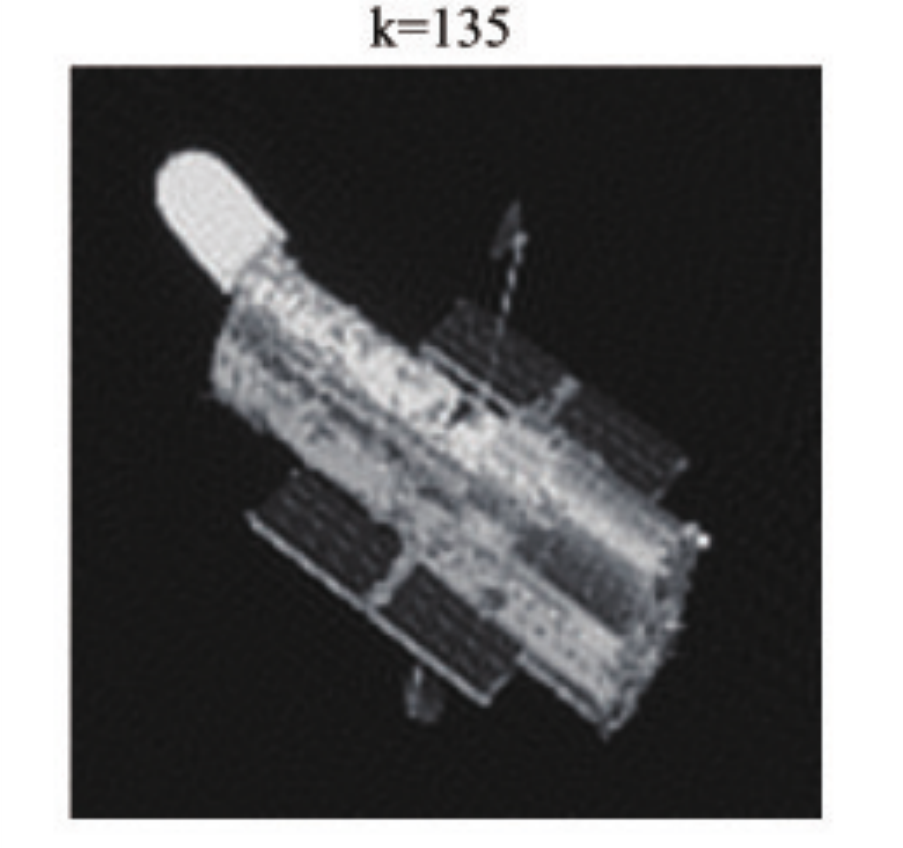}}
		\centerline{(d) $\varepsilon=10^{-3}$}
	\end{minipage}
	\caption{Images ``Hubble Space Telescope" blurred by {\sf PRblurspeckle} and restored by the mixed precision LSQR ``s+s'' at the optimal iteration: (a), (b) $\varepsilon=10^{-2}$; (c), (d) $\varepsilon=10^{-3}$.}
	\label{fig5.8}
\end{figure}

For {\sf PRblurdefocus} with image ``Cameraman", the noise levels are set as $\varepsilon=10^{-3}$ and $\varepsilon=10^{-4}$. For noise levels smaller than $10^{-5}$, the semi-convergence behavior will not appear and we need not test for those cases. Figure \ref{fig5.9} and Figure \ref{fig5.10} show the relative error curves and blurred and restored images. Subfigures (b) and (d) of Figure \ref{fig5.9} depict the relative errors between regularized solutions computed by ``s+d" and ``s+s" with upper bounds $\kappa(\widehat{R}_k)\mathbf{\bar{u}}$. We can find that $\kappa(\widehat{R}_k)$ for the two noise levels grow very slightly, which lead to the upper bounds much smaller than $\mathrm{RE}(k)$. Therefore, the regularized solutions computed by ``s+d" and ``s+s" have the same accuracy, which can clearly observed from subfigures (a) and (c). The other experimental results are similar to those of {\sf PRblurdefocus} and we do not illustrate them in detail any longer. 

\begin{figure}[htp]
	\begin{minipage}{0.45\linewidth}
		\centerline{\includegraphics[width=5.0cm,height=3.5cm]{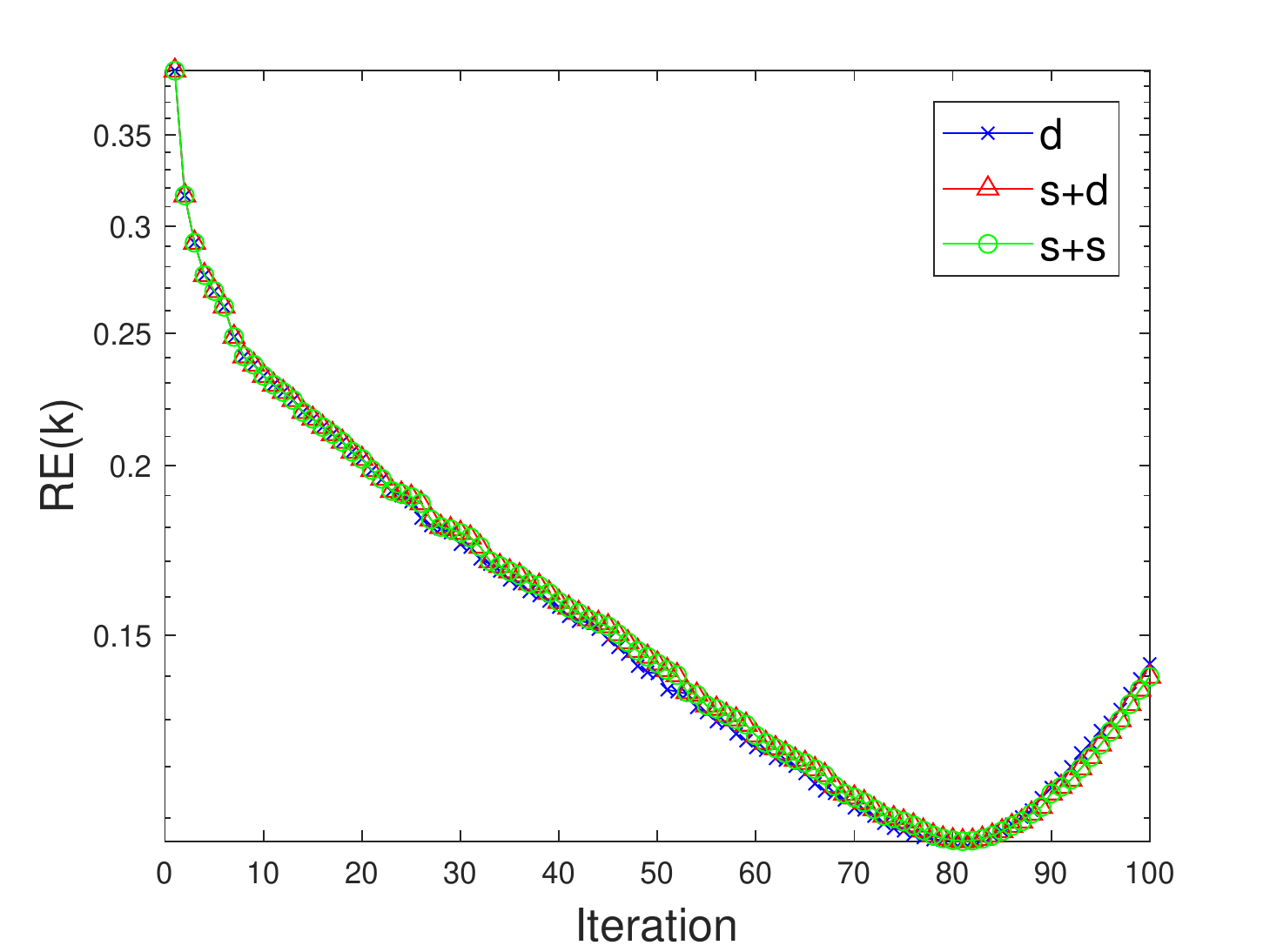}}
		\centerline{(a) $\varepsilon=10^{-3}$}
	\end{minipage}
	\hfill
	\begin{minipage}{0.45\linewidth}
		\centerline{\includegraphics[width=5.0cm,height=3.5cm]{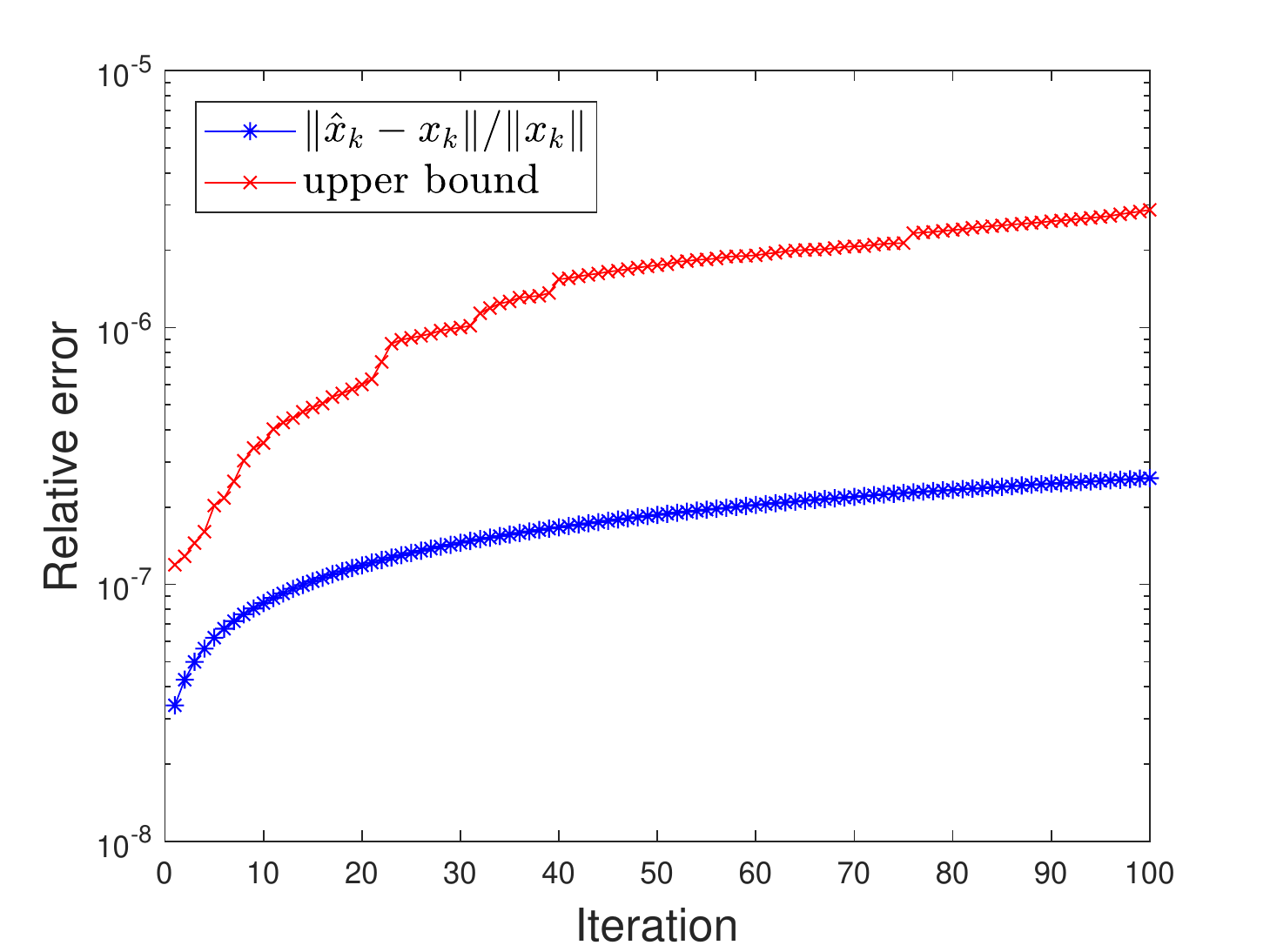}}
		\centerline{(b) $\varepsilon=10^{-3}$}
	\end{minipage}
	\vfill
	\begin{minipage}{0.45\linewidth}
		\centerline{\includegraphics[width=5.0cm,height=3.5cm]{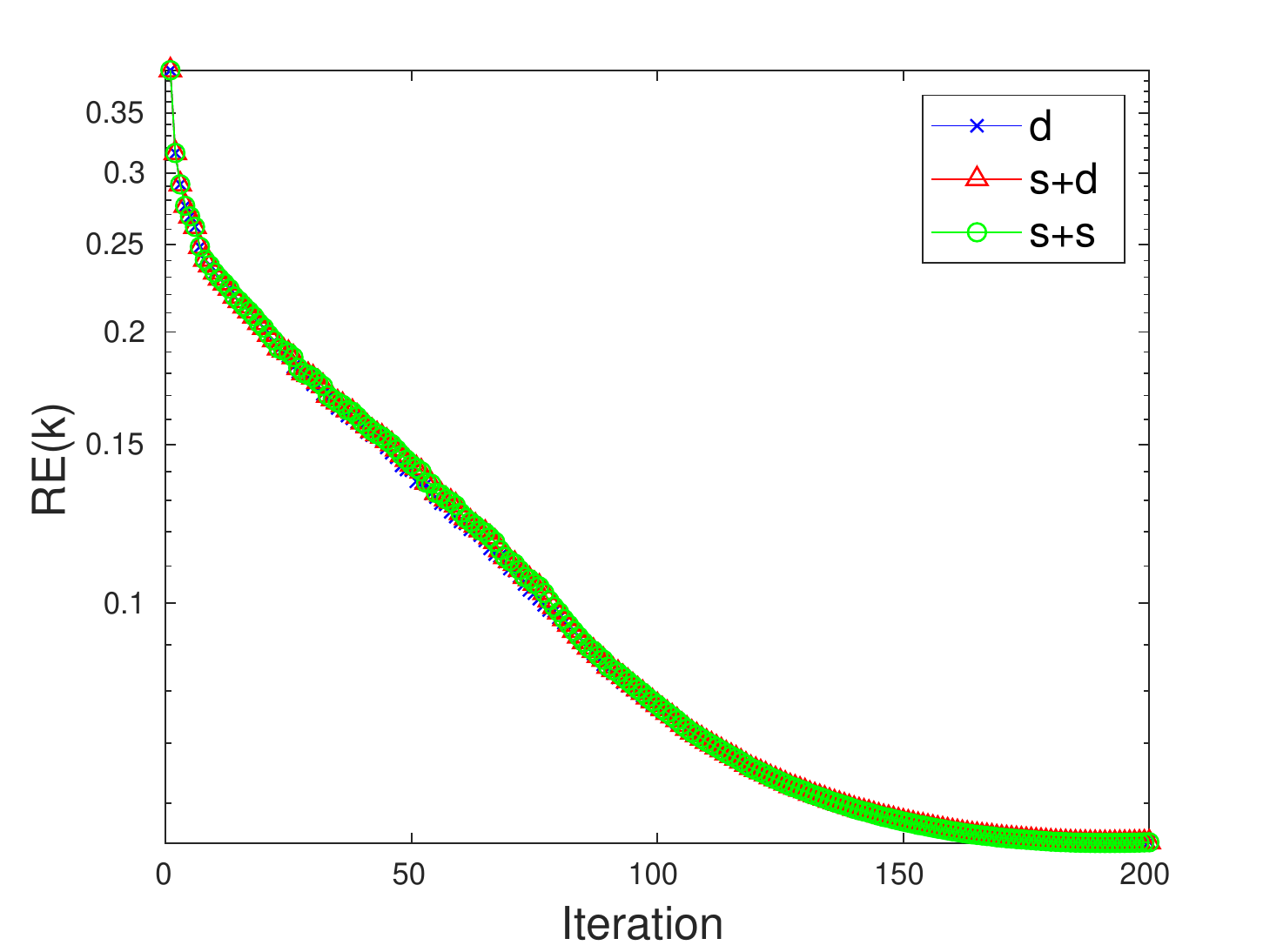}}
		\centerline{(c) $\varepsilon=10^{-4}$}
	\end{minipage}
	\hfill
	\begin{minipage}{0.45\linewidth}
		\centerline{\includegraphics[width=5.0cm,height=3.5cm]{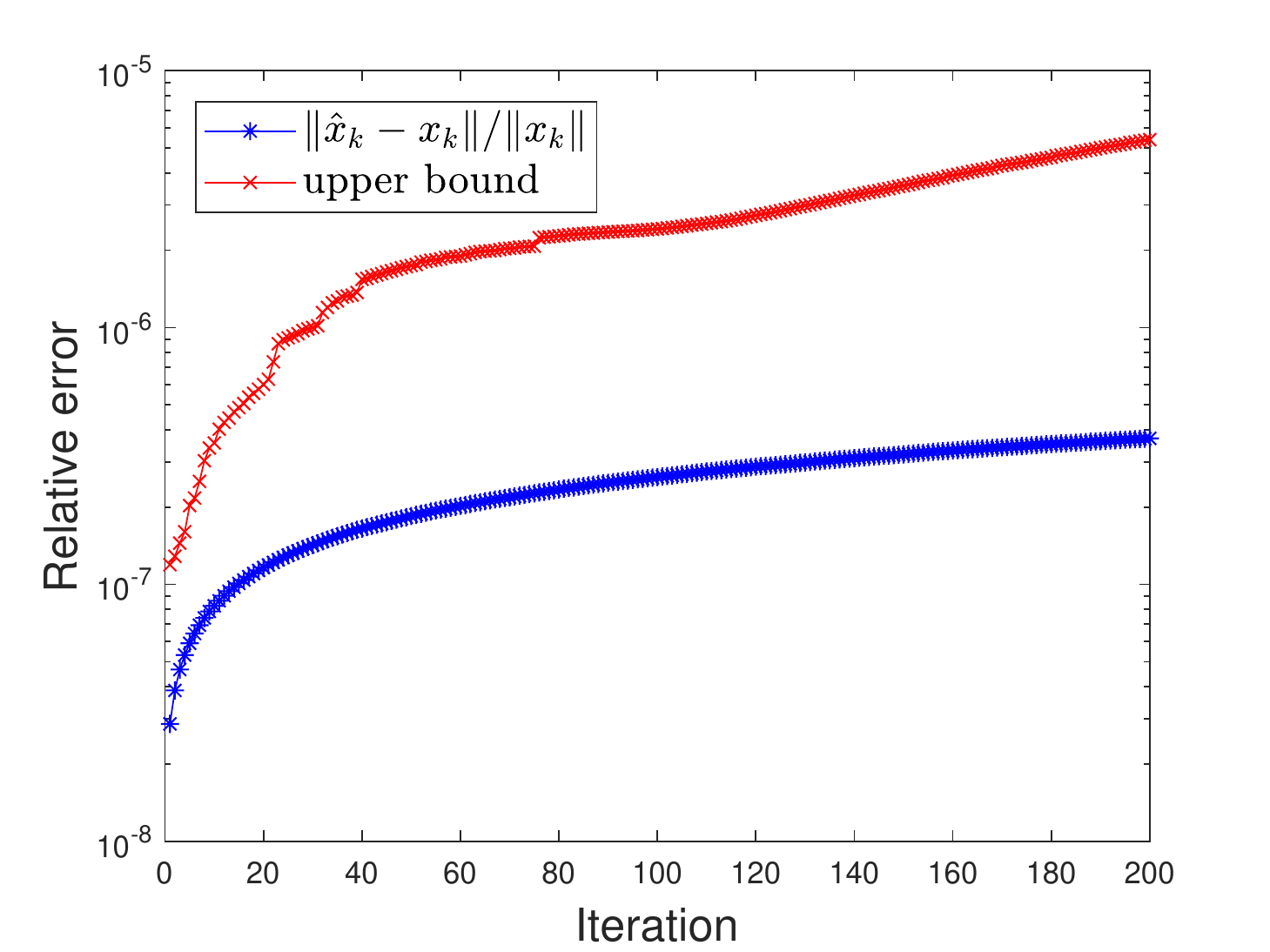}}
		\centerline{(d) $\varepsilon=10^{-4}$}
	\end{minipage}
	\caption{Semi-convergence curves for LSQR implemented using different computing precisions, and relative errors between regularized solutions computed by ``s+d" and ``s+s", {\sf PRblurdefocus}.}
	\label{fig5.9}
\end{figure}

\begin{figure}[htbp]
	\begin{minipage}{0.48\linewidth}
		\centerline{\includegraphics[width=4cm,height=3.0cm]{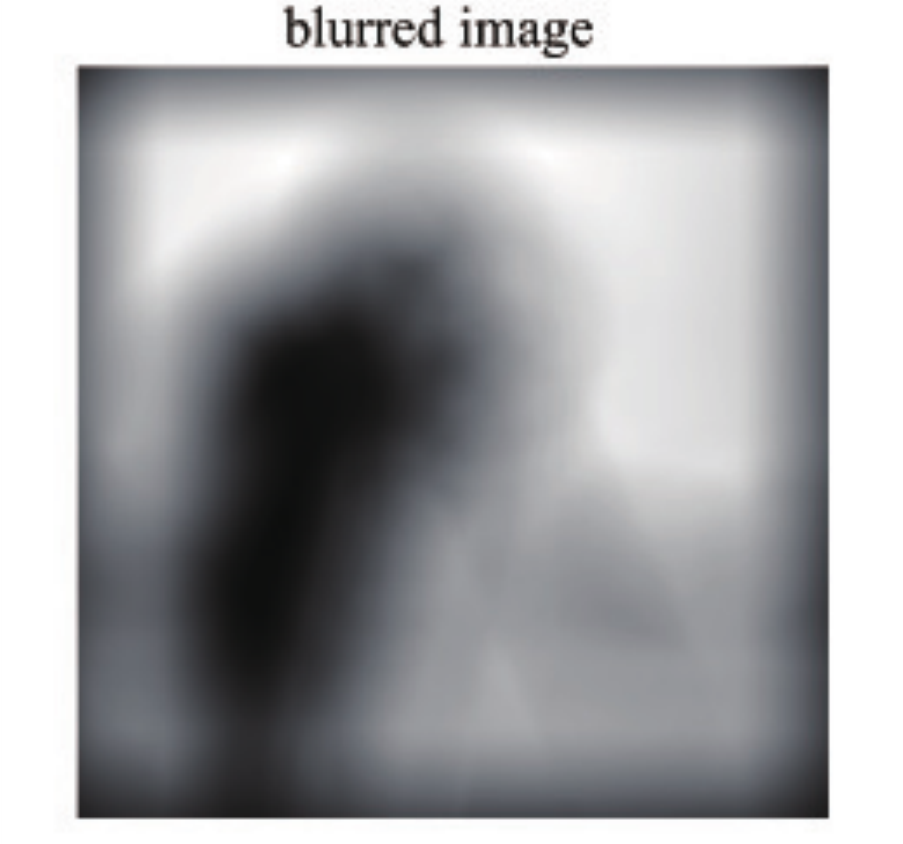}}
		\centerline{(a) $\varepsilon=10^{-3}$}
	\end{minipage}
	\hfill
	\begin{minipage}{0.48\linewidth}
		\centerline{\includegraphics[width=4cm,height=3.0cm]{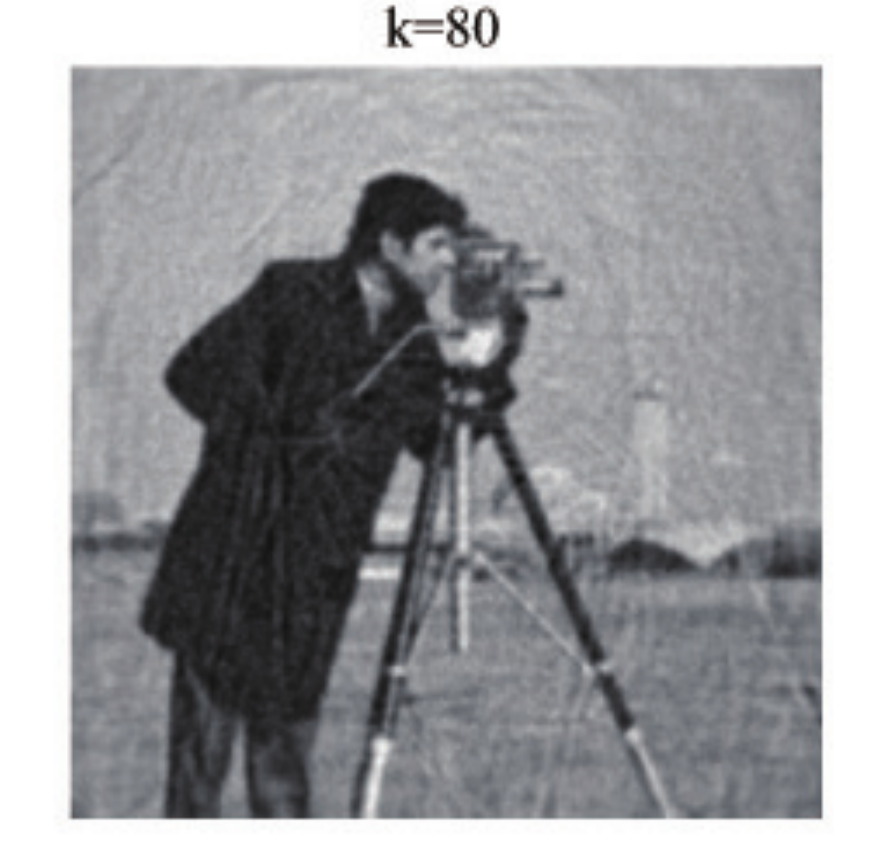}}
		\centerline{(b) $\varepsilon=10^{-3}$}
	\end{minipage}
	\vfill
	\begin{minipage}{0.48\linewidth}
		\centerline{\includegraphics[width=4cm,height=3.0cm]{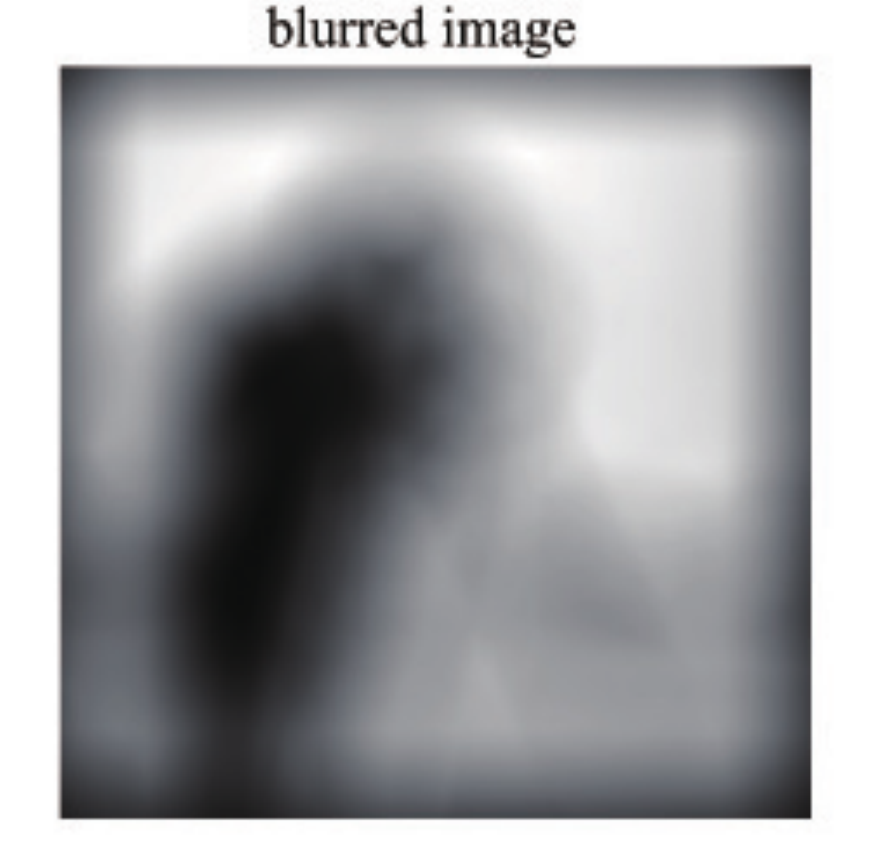}}
		\centerline{(c) $\varepsilon=10^{-4}$}
	\end{minipage}
	\hfill
	\begin{minipage}{0.48\linewidth}
		\centerline{\includegraphics[width=4cm,height=3.0cm]{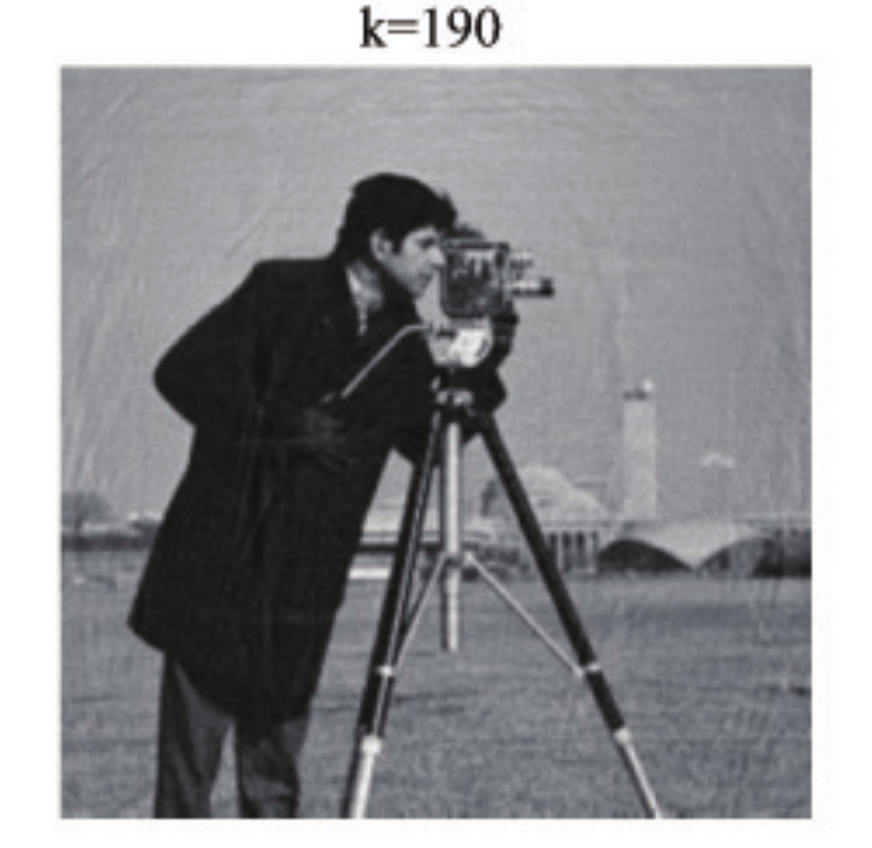}}
		\centerline{(d) $\varepsilon=10^{-4}$}
	\end{minipage}
	\caption{Images ``Cameraman" blurred by {\sf PRblurdefocus} and restored by the mixed precision LSQR ``s+s'' at the optimal iteration: (a), (b) $\varepsilon=10^{-3}$; (c), (d) $\varepsilon=10^{-4}$.}
	\label{fig5.10}
\end{figure}

Table\ref{tab5.3} shows the relative errors of the regularized solutions at the semi-convergence point $k_0$ and the estimates by L-curve criterion and discrepancy principle. For {\sf PRblurspeckle}, the three optimal iterations $k_0$ and corresponding $\mathrm{RE}(k_0)$ for ``d", ``s+d" and ``s+s" are the same, while for {\sf PRblurdefocus} the $k_0$ for ``d" and ``s+d/s" are differed only by one and the corresponding $\mathrm{RE}(k_0)$ are the same. For the discrepancy principle, the estimates of $k_0$ are the same for ``d" and ``s+d/s", and the corresponding $\mathrm{RE}(k_0)$ is only slightly different for {\sf PRblurdefocus} with $\varepsilon=10^{-3}$. For the L-curve criterion, the estimates of $k_0$ and the corresponding $\mathrm{RE}(k_0)$ for ``d", ``s+d" and ``s+s" are not the same for some cases, but the differences are very slight.

From the above experimental results, we can make sure that single precision computing of LBFRO and updating $x_k$ is enough for LSQR for solving the 2-D image deblurring ill-posed problems if the noise level is not extremely small. For those large scale problems, the mixed precision variant of LSQR has a great potential of defeating the double precision implementation in computation efficiency. We will consider implementing and optimizing mixed precision codes of LSQR for large scale problems on a high performance computing architecture in the future work.

\begin{table}[htbp]
	\centering
	\caption{Comparison of relative errors $\mathrm{RE}(k)$ and estimates of the optimal iterations $k_0$ by L-curve criterion and DP ($\tau=1.001$).}
		\begin{tabular}{*{5}{c}}
			\toprule
			Work precision & \multicolumn{2}{c}{\sf PRblurspeckle} & \multicolumn{2}{c}{\sf PRblurdefocus} \\
			\cmidrule(lr){2-3} \cmidrule(lr){4-5}
			  &  $\varepsilon=10^{-2}$ & $\varepsilon=10^{-3}$ &  $\varepsilon=10^{-3}$ &  $\varepsilon=10^{-4}$ \\
			\midrule
			  &  \multicolumn{3}{c}{Optimal} &    \\
			\midrule
			d    & $0.1850$ ($34$) & $0.1102$ ($135$) & $0.1058$ ($80$) & $0.0542$ ($190$) \\
			s+d  & $0.1850$ ($34$) & $0.1102$ ($135$) & $0.1058$ ($81$) & $0.0542$ ($191$) \\
			s+s  & $0.1850$ ($34$) & $0.1102$ ($135$) & $0.1058$ ($81$) & $0.0542$ ($191$) \\
			\midrule
			&  \multicolumn{3}{c}{L-curve} &   \\
			\midrule
			d    & $0.1992$ ($47$) & $0.1105$ ($149$) & $0.1293$ ($96$) & $0.0764$ ($100$) \\
			s+d  & $0.1992$ ($47$) & $0.1105$ ($149$) & $0.1157$ ($91$) & $0.0771$ ($100$) \\
			s+s  & $0.1992$ ($47$) & $0.1103$ ($145$) & $0.1157$ ($91$) & $0.0771$ ($100$) \\
			\midrule
			&   \multicolumn{3}{c}{Discrepancy principle} &    \\
			\midrule
			d    & $0.1937$ ($24$) & $0.1200$ ($78$) & $0.1081$ ($74$) & $0.0599$ ($137$) \\
			s+d/s  & $0.1937$ ($24$) & $0.1200$ ($78$) & $0.1099$ ($74$) & $0.0599$ ($137$) \\
			\bottomrule[0.6pt]
	\end{tabular}
	\label{tab5.3}
\end{table}

\section{Conclusion}\label{sec6}
For the most commonly used iterative regularization algorithm LSQR for solving linear discrete ill-posed problems, we have investigated how to get a mixed precision implementation by analyzing the choice of proper computing precision for the two main parts of the algorithm, including the construction of Krylov subspace and updating procedure of iterative solutions. Based on the commonly used regularization model for linear inverse problems, we have shown that, for not extremely small noise levels, single precision is enough for computing Lanczos vectors with full reorthogonalization without loss of any accuracy of the final regularized solution. For the updating procedure, we have shown that the update of $x_k$ and $w_k$, which is the most time consuming part, can be performed using single precision without sacrificing any accuracy as long as $\kappa(\widehat{R}_k)$ is not very big that is almost always satisfied. Several numerical experiments are made to test two mixed precision variants of LSQR and confirm the theoretical results.

Our results indicate that several highly time consuming parts of the algorithm can be implemented using lower precisions, and provide a theoretical guideline for implementing a robust and efficient mixed precision variant of LSQR for solving discrete linear ill-posed problems. Future work includes developing practical C codes on high performance computing architectures for specific real applications.





%

\bibliographystyle{spmpsci}      
\bibliography{refs}

\end{document}